\documentclass[10pt,a4paper]{article}

\RequirePackage{my_packages}
\RequirePackage{my_theorems}
\RequirePackage{my_macros}
\RequirePackage{my_tikz}
\RequirePackage{csquotes}

\usepackage[backend = biber,        
            language = english ,
            style = numeric-comp ,   
            giveninits = true,
            isbn = false,
            url = false,
            doi = false,
            sorting = nyt,
            maxnames = 6,
            backref=false
            ]{biblatex}           

\renewbibmacro{in:}{%
  \ifentrytype{article}{}{%
  \printtext{\bibstring{in}\intitlepunct}}}
\usepackage{tikz}
\tikzset{every loop/.style={}}
\usetikzlibrary{arrows,shapes,decorations.markings,automata,backgrounds,petri,bending,calc}

\usepackage{tikz-cd} 
\tikzset{
    labl/.style={anchor=south, rotate=90, inner sep=.5mm}
}

\usepackage[capitalize]{cleveref}
\usepackage{paralist}

\AtEveryBibitem{%
  \clearfield{note}%
  \clearfield{edition}%
  \clearlist{language}
}

\DeclareNameAlias{labelname}{given-family}

\title{The finitely generated intersection property in fundamental groups of graphs of groups}
\author{
    Jordi Delgado, \small{\url{jorge.delgado@upc.edu}}\\
    \small{Universitat Politècnica de Catalunya - BarcelonaTech (UPC), Barcelona, Spain}\\
    \small{ORCID: \url{https://orcid.org/0000-0002-8365-8929}}
    \and
    Marco Linton, \small{\url{marco.linton@icmat.es}}\\
    \small{Instituto de Ciencias Matemáticas (ICMAT), Madrid, Spain}\\
    \small{ORCID: \url{https://orcid.org/0000-0002-1081-5268}}
    \and
    Jone Lopez de Gamiz Zearra, \small{\url{jone.lopezdegamiz@ehu.eus}}\\
    \small{University of the Basque Country - EHU, Bilbao, Spain}\\
    \small{ORCID: \url{https://orcid.org/0000-0002-3725-9740}}
    \and
    Mallika Roy, \small{\url{mallikaroy75@gmail.com}}\\
    \small{Harish-Chandra Research Institute, A CI of HBNI, India}\\
    \small{ORCID: \url{https://orcid.org/0000-0002-9730-5980}}
    \and
    Pascal Weil, \small{\url{pascal.weil@cnrs.fr}}\\
    \small{CNRS, ReLaX, IRL 2000, Siruseri, India}\\
    \small{CNRS, Univ. Sorbonne Paris Nord, LIPN, UMR 7030, F-93430 Villetaneuse, France}\\
    \small{ORCID: \url{https://orcid.org/0000-0003-2039-5460}}
}
\date{\today}

\begin{document}
\maketitle

\begin{abstract}
A group $G$ is said to satisfy the finitely generated intersection property (\fgip) if the intersection of any two finitely generated subgroups of $G$ is again finitely generated.
The aim of this article is to understand when the fundamental group of a graph of groups has the \fgip\ Our main results are general criteria for the \fgip\ in graphs of groups which depend on properties of the vertex groups, properties of certain double cosets of the edge groups and the structure of the underlying graph. For acylindrical graphs of groups, we also obtain criteria for the strong \fgip\ (\sfgip). Our results generalise classical results due to Burns and Cohen on the \fgip\ for amalgamated free products and HNN extensions. As a concrete application, we show that a graph of locally quasi-convex hyperbolic groups with virtually $\ZZ$ edge groups (for instance, a generalised Baumslag--Solitar group) has the \fgip\ if and only if it does not contain $F_2\times\ZZ$ as a subgroup. In addition, we show that this condition is decidable. The main tools we use are the explicit constructions of pullbacks of immersions into a graph of group, obtained by the authors in a previous paper, and a technical condition on coset interactions, introduced in this paper.
\end{abstract}

\vspace{15pt}


\newpage
\tableofcontents

\section*{Introduction}
\addcontentsline{toc}{section}{Introduction}

A group $G$ is said to satisfy the \emph{finitely generated intersection property}\footnote{Also known as the~Howson property.} (\fgip) if the intersection of any two finitely generated subgroups of $G$ is again finitely generated.
Interest in the \fgip\ dates back to the mid-20th century, when Howson proved that free groups enjoy this property~\cite{how54}.
Since Howson's foundational work, the \fgip\ has been revisited from multiple perspectives and across increasingly general classes of groups --- notably including amalgamated free products and HNN extensions ---
producing a variety of obstructions and of sufficient conditions for a product of groups to have the \fgip

We recall that an \emph{amalgamated free product} $A *_C B = \langle A, B \mid c = \psi(c),\, c \in C \rangle$ is formed by taking the free product of two groups $A$ and $B$ and identifying a common subgroup $C$ (embedded via a monomorphism $\psi$); whereas an \emph{HNN extension} ${A\ast_\psi = \langle A, t\mid t^{-1} c t = \psi(c),\ c\in C \rangle}$ starts from a single group $A$ with two $\psi$-isomorphic subgroups $C, D \leqslant A$ and introduces a new stable letter that conjugates one subgroup to the other.
Both constructions are special (single-edge) cases of the more general notion of a \emph{graph of groups}, where groups are assigned to the vertices and edges of a graph, with each edge group embedded into its adjacent vertex groups. The associated \emph{fundamental group} captures the global structure by amalgamating vertex groups along edge groups according to the graph’s connectivity. These constructions effectively encode the idea of “gluing groups along subgroups” over an arbitrary graph, rather than just a single edge (see Section~\ref{sec: define gog} for precise definitions and details).

Our goal in this paper is to unify and extend previous results on the \fgip\ (and related properties) summarised in the technical overview below, to the broader setting of fundamental groups of graphs of groups, using a common framework based on the pullback construction in~\cite{dllrw_pullback}, and a notion of finite coset interaction introduced in this paper.

\medskip

The earliest extensions of Howson’s theorem began with Greenberg’s 1960 result in \cite{gre60}, showing that surface groups also satisfy the \fgip, thus identifying a significant class of non-free groups where the property holds.
This was followed by Baumslag's extension of Howson's result to free products (the free product of two groups with the \fgip\ has again the \fgip\
\cite{bau66}). 
This result marked the starting point for the investigation of the \fgip\ under more general constructions, such as amalgamated free products, HNN extensions, and more generally, fundamental groups of graphs of groups.

For example, in~\cite{mol68} Moldavanskiĭ adresses the \fgip\ within the family 
of Baumslag--Solitar groups (which are HNN extensions of $\ZZ$) and proves that it holds just for those of the form $\bs(1,n)$.
In the same paper, a prototypical example of a non-\fgip\ group
(showing, in particular, that the \fgip\ property does not pass through direct products) is also provided:
the direct product $F_2 \times \ZZ$ does not have the \fgip\
Since the \fgip\ is clearly subgroup hereditary, containing $F_2 \times \ZZ$ as a subgroup constitutes a first natural obstruction for a group to have the \fgip\
It is worth noting that, within certain relevant families of groups (\eg right-angled Artin groups \cite{sds90}, and Baumslag--Solitar groups \cite{pa12}), this turns out to be the only obstruction to having the \fgip\ In this paper we extend this characterisation to  generalised Baumslag-Solitar groups and beyond.

During the 1970s, substantial progress was made on the \fgip\ in the context of amalgamated products and HNN extensions. Karrass and Solitar~\parencite{ks70,ks71} proved that both constructions inherit the \fgip\ from the factors when the respective associated subgroup is finite. 
However, it soon became clear that the class of \fgip\ groups is not closed under amalgamations and HNN extensions along infinite subgroups.
Shortly afterward, Burns~\parencite{bur72,bur73} undertook a detailed study of conditions under which amalgamated products and HNN extensions inherit the \fgip\ 
This work was later expanded by Cohen \parencite{coh74,coh76}, who, drawing on the newly developed Bass--Serre theory,  further clarified and extended Burns' approach.
These developments converged into the notion of \emph{Burns subgroup}, a somewhat technical condition (discussed in the Appendix) requiring that the amalgamated subgroup is embedded in each factor in a ``tame'' way. Specifically, avoiding large coset intersections and ensuring that intersections with finitely generated subgroups remain finitely generated.

An important step forward in this program was made two decades later in \cite{kap97}, where Kapovich  extended the Burns--Cohen results to the realm of hyperbolic groups by essentially showing that any maximal cyclic subgroup of a torsion-free locally quasi-convex hyperbolic group is a Burns subgroup. 
In the following years, the \fgip\ was confirmed in other significant classes. Quite remarkably, limit groups --- those sharing the same first order theory as free groups --- also satisfy the \fgip, a result fully established by Dahmani~\cite{dah2003} after Kapovich~\cite{kap02} had proved it for the hyperbolic case.

As a counterpart of his positive results in \cite{kap97}, Kapovich showed in \cite{kap99} that the \fgip\ can also fail within the geometrically desirable class of word hyperbolic groups, providing new counterexamples to this property.
In a similar direction, other classical counterexamples to the \fgip, such as free-by-cyclic groups \cite{bb79}, have recently been extended: first to ascending HNN extensions of finitely generated free groups by Bamberger and Wise~\cite{bw2022}, and more generally, to all ascending HNN extensions of arbitrary free groups by Linton~\cite{li25}.

Finally, we mention that the \fgip\ is a generic property for finitely presented groups \cite{Ar98} --- for each integer $\ell$, there exists a generic class of presentations in which every $\ell$-generated subgroup is quasi-convex. Using the same method in \cite{Ar2000} Arzhantseva showed a generic property of normal subgroups of a free group $F_n$: for a fixed finitely generated infinite index subgroup $H$ of $F_n$, a generic normal subgroup of $F_n$ trivially intersects with $H$. Also, we note that the \fgip\ is not a first order property
\cite[Theorem 10.4.13]{fgmrs2014}, reinforcing its geometric nature and dependence on subgroup structure.
Moreover, as a Markov property, it is algorithmically undecidable whether a given finite presentation defines a group with the \fgip\ A consequence of our results is that this obstruction can be overcome when restricting to fundamental groups of graphs of locally quasi-convex hyperbolic groups with virtually $\mathbb{Z}$ edge groups. 

\medskip

Our main tools for investigating the \fgip\ in graphs of groups are, first, the construction of pullbacks of immersions into connected pointed graphs of groups, developed in~\cite{dllrw_pullback} precisely for computing intersections of subgroups of the fundamental group of a graph of groups, and, second, the \emph{finite coset interaction property} (\fcip),  
which identifies situations where finite graphs of groups give rise to an infinite pullback and hence to a non-finitely generated intersection.
This technical notion and its many variants (notably the $k$-\fcip, where $k$ is an integer) are defined in the technical overview below. For now, we emphasise the fact that the \fcip\ of the vertex groups relative to their adjacent edge groups, implies local finiteness properties of the pullbacks of finite graphs of groups (and, in fact, it is a necessary condition for these finiteness properties, in a precise sense discussed in the corresponding section).

This allows us to prove a number of theorems on the preservation of the \fgip\ and the \sfgip\ by graphs of groups, some of which are rather technical. Highlights of these results, which are described in more detail below, are as follows. Let $(\AA,u_0)$ be a pointed graph of groups.

\begin{itemize}
    \item Suppose that the edge groups of $\AA$ are finite. We retrieve Cohen's result \cite[Theorem~7]{coh74}, which states that the fundamental group $\pi_1(\AA,u_0)$ has the \fgip\ if and only if its vertex groups have the \fgip\ We show that this equivalence does not hold for the \sfgip, unless $\AA$ is acylindrical. In general, we show that $\pi_1(\AA,u_0)$ has the \sfgip\ if and only if its vertex groups have the \fgip\ and if it does not contain a subgroup of the form $C\times \ZZ$, where $C$ is non-trivial, finite and cyclic.
    \item If the vertex groups of $\AA$ have the \fcip\ relative to their incoming edge groups and if $\pi_1(\AA,u_0)$ has the \fgip\ relative to its edge stabilisers,  we show that $\pi_1(\AA,u_0)$ has the \fgip\ if and only if its vertex groups do. The equivalence extends to the \sfgip\ if $\AA$ is acylindrical. 
    
    \item Let $G = A\ast_\phi$ be an HNN extension, where $\phi$ is an injective morphism from a subgroup $C$ of $A$ into $A$. We give sufficient conditions for $G$ to have the \fgip\ in terms of the interactions between $A$ and the subgroups $C$ and $\phi(C)$, which generalise Burns' result~\cite[Theorem 1.2]{bur73}.

    \item Let $G = A\ast_CB$ be an amalgamated product, where $C$ is given by injective morphisms into $A$ and $B$. We give sufficient conditions for $G$ to have the \fgip\ in terms of the interactions between $A$, $B$ and the images of $C$, which generalise Cohen's result~\cite[Theorem 2]{coh76}.
    
    \item If $\AA$ is a graph of locally quasi-convex hyperbolic groups with virtually $\ZZ$ edge groups, we show that its fundamental group has the \fgip\ if and only if it avoids the subgroup $F_2\times \ZZ$. Moreover, this property is decidable. In the case where $\AA$ has a single edge, that is, its fundamental group is an HNN extension or an amalgamated free product, we characterise the \fgip\ under more general conditions than having virtually $\ZZ$ edge groups.
\end{itemize}

\subsection*{Technical overview}
\addcontentsline{toc}{subsection}{Technical overview}

In this section we present a technical overview of the main points of our strategy for proving \fgip\ criteria for graphs of groups. We shall assume some familiarity with graphs of groups and their morphisms, which will be explained in detail in Section \ref{sec: graphs of groups}.

We first emphasise that if $\AA$ is a graph of groups, then each subgroup $B\leqslant \pi_1(\AA, u_0)$ can be realised uniquely by an immersion of a pointed core graph of groups $(\BB, v_0)\to (\AA, u_0)$, see Proposition \ref{main_prop_A} below. Moreover, if $(\BB, v_0)\to (\AA, u_0)$ and $(\CC, w_0)\to (\AA, u_0)$ are immersions of graphs of groups, then:
\begin{itemize}
    \item The (pointed) pullback $(\BB\times_{\AA}\CC, x_0)$ exists and $\pi_1(\BB\times_{\AA}\CC, x_0) \cong \pi_1(\BB, v_0)\cap \pi_1(\CC, w_0)$.
    \item The (pointed) pullback is precisely the component of the $\AA$-product $\BB\wtimes_{\AA}\CC$ containing the basepoint, a construction which was introduced and described in depth in \cite{dllrw_pullback}.
    \item Each component of the $\AA$-product $\BB\wtimes_{\AA}\CC$ corresponds in a precise way to an intersection $\pi_1(\BB, v_0)^g\cap \pi_1(\CC, w_0)$ for $g\in \pi_1(\AA, u_0)$. 
\end{itemize} 
In particular, understanding the structure of the $\AA$-product $\BB\wtimes_{\AA}\CC$ will help understand intersections of subgroups of $\pi_1(\AA, u_0)$ and their conjugates. More details on the facts above can be found in Sections~\ref{sec: graphs of groups} and \ref{sec: pullbacks}.

We shall be interested in some variations of the \fgip\ property, which will allow us to constrain the groups appearing in the $\AA$-products: 
\begin{itemize}
\item If $G$ is a group and $\mathcal{A}$ is a collection of subgroups of $G$, we say that $G$ has the \emph{\fgip\ relative to $\mathcal{A}$} if, for all finitely generated subgroups $H\leqslant G$ and all $A\in \mathcal{A}$, we have that $H\cap A$ is finitely generated.
\item If $G$ is a group, we say that $G$ has the \emph{\sfgip} if $G$ has the \fgip\ and if, for all pairs of finitely generated subgroups $B, C\leqslant G$, there are finitely many $(B, C)$-double cosets $BgC$ so that $B^g\cap C\neq 1$. 
\item If $A\leqslant G$ is a subgroup, we say that \emph{$(G, A)$ has the $A$-\fgip} if, for all pairs of finitely generated subgroups $B, C\leqslant G$, we have that for all but finitely many elements $g\in A$ in distinct $(B, C)$-double cosets, the intersection $B^g\cap C$ lies in $A$. 
\end{itemize}
Precise definitions and some generalisations can be found in Section \ref{ssec: intersection}.
The reason we are interested in these properties are the following propositions. 

If $\AA$ is any graph of groups, we define the \emph{edge subgroups} of $\AA$ to be the subgroups of the form $p\alpha_e(A_e)p^{-1}\leqslant\pi_1(\AA, u_0)$, where $A_e$ is an edge group of $\AA$ and where $p$ is any $\AA$-path connecting $u_0$ with $o(e)$. These are precisely the subgroups which fix an edge in the $\pi_1(\AA, u_0)$-action on the Bass--Serre tree associated with $\AA$.

\begin{mainproposition}[Theorem \ref{thm: bijection subgroups immersions covers}]
\label{main_prop_A}
    Let $\AA$ be a graph of groups and let $B\leqslant \pi_1(\AA, u_0)$ be a finitely generated subgroup. If $\pi_1(\AA, u_0)$ has the \fgip\ relative to its collection of edge subgroups, then there is a pointed immersion $\mu^B\colon (\BB, v_0) \to (\AA, u_0)$ such that $\mu_*^B(\pi_1(\BB, v_0)) = B$ and such that $(\BB, v_0)$ is a finite core pointed graph of finitely generated groups.
\end{mainproposition}

Recall that an edge $e$ in a graph of groups $\AA$ is not reduced if the edge group inclusion $\alpha_e\colon A_e\to A_{o(e)}$ is an isomorphism.

\begin{mainproposition}[Proposition \ref{prop: fgip_summary}]\label{prop: B}
    Let $\AA$ be a graph of groups, let $\BB, \CC\to\AA$ be immersions of finite graphs of finitely generated groups. Then:
    \begin{itemize}
        \item If each vertex group (respectively, edge group) of $\AA$ has the \fgip, then each vertex group (respectively, edge group) of $\BB\wtimes_{\AA}\CC$ is finitely generated.
        \item If each vertex group (respectively, edge group) of $\AA$ has the \sfgip, then $\BB\wtimes_{\AA}\CC$ has finitely many non-trivial vertex groups (respectively, edge groups).
        \item If $e\in E(\gr{A})$ is such that $(A_{o(e)}, \alpha_{e}(A_e))$ has the $\alpha_e(A_e)$-\fgip, then all but finitely many edges in $\BB\wtimes_{\AA}\CC$ that map to $e$ are non-reduced.
    \end{itemize}
\end{mainproposition}

The first two points in Proposition \ref{prop: fgip_summary} hold because vertex and edge groups of the $\AA$-product are intersections of conjugates of vertex and edge groups of $\BB$ with vertex and edge groups of $\CC$. The third point is more technical and has to do with the adjacency map in $\AA$-products.
 
We continue with the key technical definition of the article. The map appearing in the definition below may appear unnatural, but it is essentially the edge adjacency map for the $\AA$-product $\BB\wtimes_{\AA}\CC$. As such, it is deeply related with the structure of the underlying graphs of $\AA$-products (and hence, pullbacks).

\begin{defn*}[Definition \ref{defn: finite coset interaction2}]
\label{def: fcip_short}
    If $G$ is a group and $\mathcal{A}$ is a finite collection of finitely generated subgroups, we say that $G$ has the \emph{\fcip\ relative to $\mathcal{A}$} if the following holds: for each pair of finitely generated subgroups $B, C\leqslant G$ and for any collection of $(B, A)$-double cosets $\mathcal{F}_A$ and $(C, A)$-double cosets $\mathcal{G}_A$ as $A$ ranges over elements of $\mathcal{A}$, the following map
    \begin{align*}
    \bigsqcup_{A\in \mathcal{A}}\bigsqcup_{\substack{f\in \mathcal{F}_A}{g\in \mathcal{G}_A}}(A\cap B^f)\backslash A/(A\cap C^g) &\to B\backslash G/ C\\
    (A\cap B^f)a(A\cap C^g) &\mapsto Bfag^{-1}C
    \end{align*}
    is injective on all but finitely many double cosets. 

    We say that $G$ has the \emph{$k$-\fcip\ relative to $\mathcal{A}$} if each map as above is at most $k$-to-one on all but finitely many double cosets.
\end{defn*}

Let us first mention some examples of the \fcip\ As an important degenerate case, a group $G$ has the $0$-\fcip\ relative to $\mathcal{A}$ if and only if $\mathcal{A}$ is a finite collection of finite subgroups. In the appendix, we show that if $G$ is a group and $A\leqslant G$ is a Burns subgroup --- a technical definition introduced by Burns in \cite{bur72} to study the \fgip\ --- then $G$ has the \fcip\ relative to $A$ and $(G, A)$ has the $A$-\fgip\ We do not know whether the converse holds. We also show that the same holds if $A$ is replaced by a Burns collection of subgroups --- a variation introduced by Kapovich \cite{kap02}. In some examples in Section \ref{sec: fcip} we explicitly characterise the \fcip\ in abelian groups. With more work, we also obtain the following important examples of groups with the \fcip\ 

\begin{maintheorem}[Corollaries \ref{cor: A-fgip} \& \ref{fcip_qc_hyperbolic_2}]\label{prop: C}
    If $G$ is a locally quasi-convex hyperbolic group and $\mathcal{A}$ is an almost malnormal collection of finitely generated subgroups, then $G$ has the \fcip\ relative to $\mathcal{A}$ and, for all $A\in \mathcal{A}$, $(G, A)$ has the $A$-\fgip\
\end{maintheorem}

The point of requiring that the vertex groups of a graph of groups $\AA$ have the \fcip\ is that it implies local finiteness properties of the $\AA$-products of immersions into $\AA$. For instance, we have the following result.

\begin{mainproposition}[Corollary \ref{cor: k-fcip}]\label{prop: D}
Let $k\geqslant 0$ and let $\AA$ be a finite graph of groups such that each vertex group has the $k$-\fcip\ relative to the collection of adjacent edge groups. Then for any pair of immersions $\mu^B\colon \BB\to \AA$, $\mu^C\colon \CC\to \AA$ of finite graphs of finitely generated groups, the $\AA$-product $\BB\wtimes_{\AA}\CC$ is locally finite and has finitely many vertices with strictly more than $k$ outgoing edges. In particular:
\begin{enumerate}[(1)]
    \item If $k = 0$, then $\BB\wtimes_{\AA}\CC$ has finitely many edges.
    \item If $k = 1$, then each component of $\BB\wtimes_{\AA}\CC$ has finitely many edges.
\end{enumerate} 
\end{mainproposition}

It is straightforward to deduce from Propositions~\ref{prop: B} and~\ref{prop: D}
that a graph of groups in which vertex groups have the \fgip\ and edge groups are finite has the \fgip\ This result is originally due to Cohen \cite[Theorem 7]{coh74}. We go further with our analysis and also characterise the \sfgip\ in this setting.

\begin{maintheorem}[Theorem~\ref{thm: charact sfgip finite edge groups}]\label{thm: main A}
     Let $\AA$ be a finite graph of groups with finite edge groups. The following are equivalent:
    \begin{enumerate}[(1)]
        \item $\pi_1(\AA, u_0)$ has the \sfgip
        \item Each vertex group of $\AA$ has the \sfgip\ and $\pi_1(\AA, u_0)$ does not contain any subgroup isomorphic to $\Z/k\Z\times\Z$ for any $k\geqslant 2$.
    \end{enumerate}
\end{maintheorem}

Under the weaker assumption of \fcip, we are also able to obtain criteria for the \fgip\ and \sfgip\ thanks to Proposition~\ref{prop: D}. 

\begin{maintheorem}[Theorem~\ref{fgip_criterion_1}]\label{thm: main B}
Let $(\AA,u_0)$ be a core pointed graph of groups with finite underlying graph and finitely generated vertex and edge groups. Suppose that each vertex group has the \fcip\ relative to the collection of adjacent edge groups and that $\pi_1(\AA,u_0)$ has the \fgip\ relative to all edge subgroups. Then the following holds:
\begin{enumerate}[(1)]
\item $\pi_1(\AA, u_0)$ has the \fgip\ if each vertex group of $\AA$ has the \fgip. 
\item $\pi_1(\AA, u_0)$ has the \sfgip\ if each vertex group of $\AA$ has the \sfgip\ and $\AA$ is acylindrical.
\end{enumerate} 
\end{maintheorem}

Without the assumptions from Theorem \ref{thm: main B}, it becomes more difficult to constrain the structure of the $\AA$-product. By adding stronger \fgip\ assumptions, we may weaken our \fcip\ assumptions and still obtain \fgip\ results. Since, from this point on, our results become much more technical, we shall only state their single edge graphs of groups version. In other words, we only state the criteria for the special cases of amalgamated free products and HNN extensions.

\begin{maintheorem}[Corollary~\ref{cor: single_edge_2}]\label{thm: G}
    Let $G = A*_{\varphi}$ be a HNN extension, where $\varphi\colon C_1\to C_2$ identifies two subgroups of $A$. Suppose $A$ has the \fgip, $A$ has the \fcip\ relative to $C_i$ and $(A, C_i)$ has the $C_i$-\fgip\ for $i = 1, 2$. Then $G$ has the \fgip\ if and only if $G$ has the \fgip\ relative to $C_1$.
\end{maintheorem}

Note that Theorem~\ref{thm: G} does not follow from Theorem \ref{thm: main B} since we are not assuming that $A$ has the \fcip\ relative to $\{C_1, C_2\}$. In fact, for Theorem~\ref{thm: G}, we are using the fact that $A$ has the 2-\fcip\ with respect to the pair of adjacent edge groups. 

We remark that Theorem~\ref{thm: G} recovers a result due to Burns \cite[Theorem 1.2]{bur73}. In the same article, Burns asks whether the assumptions from \cite[Theorem 1.2]{bur73} need only be made on one of the conjugated subgroups in $A$ for $A*_{\varphi}$ to have the \fgip\ In Remark \ref{rem: no_fgip} we provide a counterexample to this. However, we also show that Burns' question can be answered positively if one adds the assumption that $C$ has the \sfgip\

\begin{maintheorem}[Corollary~\ref{cor: single_edge_3}]\label{thm: main D}
    Let $G = A*_{\varphi}$ be a HNN extension, where $\varphi\colon C\to \varphi(C)$ identifies two subgroups of $A$. Suppose $A$ has the \fcip\ relative to $C$, $(A, C)$ has the $C$-\fgip\ and $C$ has the \sfgip\ Then $G$ has the \fgip\ if and only if $G$ has the \fgip\ relative to $C$.
\end{maintheorem}

An immediate corollary of Theorem \ref{thm: main D} is that Baumslag--Solitar groups $\bs(1, n)$ have the \fgip, recovering a classical result due to Moldavanskiĭ~\cite{mol68}.

In our final criterion, we recover a result of Cohen's~\cite[Theorem 2]{coh76} on the \fgip\ for amalgamated free products.

\begin{maintheorem}[Corollary~\ref{cor: single_edge_4}]\label{thm: main E}
    Let $G = A*_CB$, where $A$ has the \fcip\ relative to $C$, $(A, C)$ has the $C$-\fgip, and $B$ has the \fgip\ Then $G$ has the \fgip\ if and only if $G$ has the \fgip\ relative to $C$.
\end{maintheorem}

As far as we are aware, our criteria above recover and generalise all known criteria for the \fgip\ in graphs of groups.

Before concluding, we mention some more concrete applications of our criteria. In Section~\ref{sec: fcip for qcv subgroups} we study the \fcip\ property and its variations in hyperbolic groups and, in combination with Theorem \ref{thm: main D} and Theorem \ref{thm: main E}, establish the following two results which generalise results due to Kapovich \cite{kap97}.

\begin{maintheorem}[Theorem~\ref{hyperbolic_amalgam}]\label{thm: main F}
    Let $G \cong A*_CB$ be an amalgamated free product where $A$ is locally quasi-convex hyperbolic group, $B$ has the \fgip\ and $C$ is finitely generated and almost malnormal in $A$. Then $G$ has the \fgip\ relative to $C$ if and only if $G$ has the \fgip\
\end{maintheorem}

\begin{maintheorem}[Theorem~\ref{hyperbolic_HNN}]\label{thm: main G}
  Let $G = A*_{\varphi}$ be a HNN extension, where $\varphi\colon C\to \varphi(C)$ identifies two subgroups of $A$, where $A$ is locally quasi-convex hyperbolic and where $C$ is an almost malnormal subgroup of $A$. Then $G$ has the \fgip\ relative to $C$ if and only if $G$ has the \fgip
\end{maintheorem}

Note that it was known that, under the hypotheses of Theorem \ref{thm: main F}, if $B$ is also locally quasi-convex hyperbolic, then $G$ is locally quasi-convex hyperbolic (and thus has the \fgip) if and only if $G$ has the \fgip\ relative to $C$. This follows from a Theorem of Short's \cite{sho91} and a result of Tomar's \cite[Proposition 5.1]{To25}. By the same results and under the hypotheses of Theorem \ref{thm: main G}, it was also already known that $G$ is locally quasi-convex hyperbolic if and only if $G$ has the \fgip\ relative to $C$ and if the HNN-extension is weakly acylindrical (there exists an integer $k$ such that segments of length $k$ in the associated Bass--Serre tree have finite stabiliser).

With some more work, we obtain a complete characterisation of the \fgip\ for graphs of locally quasi-convex hyperbolic groups with virtually $\Z$ edge groups, generalising results of Burns and Brunner (\cite[Theorem 1]{bur72} and \cite[Theorem 2]{bb79}), who characterised when an amalgamated free product or an HNN extension of a free group over a cyclic group has the \fgip\ 

\begin{maintheorem}[Corollaries~\ref{cor: main locally qc characterization}, \ref{cor: decidability for locally qcv hyperbolic}]\label{thm: main H}
    The fundamental group of a connected graph of locally quasi-convex hyperbolic groups with virtually $\Z$ edge groups has the \fgip\ if and only if it does not contain any subgroup isomorphic to $F_2\times \Z$. In addition, this property is decidable.
\end{maintheorem}

Note that it was known that a group as in Theorem \ref{thm: main H} is locally quasi-convex hyperbolic if and only if the graph of groups is weakly acylindrical. This can be shown to be equivalent to no Baumslag--Solitar subgroups being present. As before, this follows from a Theorem of Short's \cite{sho91} and a result of Tomar's \cite[Proposition 5.1]{To25}.

In the special case of graphs of virtually $\ZZ$ groups with virtually $\ZZ$ edges, Theorem~\ref{thm: main H} translates to a particularly simple form.

\begin{maintheorem}[Theorem~\ref{thm: vcyclic_fgip}]\label{thm: main M}
    The fundamental group of a connected graph of virtually $\ZZ$ groups with virtually $\Z$ edge groups has the \fgip\ if and only if it is virtually $\Z$, an amalgamated product $A\ast_CB$ of two virtually $\ZZ$ groups such that $C$ has index 2 in both $A$ and $B$, or an HNN extension $A\ast_{\varphi}$, where $\varphi$ identifies $A$ with a subgroup of itself.
\end{maintheorem}

We also include the following application of the $\AA$-product construction. The \emph{free factor rank} of a group $G$ is the maximal rank of a freely indecomposable free factor of $G$ ---  that is, the maximum rank of a non-free factor in the Grushko decomposition of $G$, if such a decomposition exists. The \emph{finite factor rank intersection property} (\ffrip) for $G$ expresses the fact that any  
intersection of two finitely generated subgroups of $G$ has finite free factor rank. The strong \ffrip\ is a technical strengthening of this condition, in the spirit of the definition of the strong \fgip\ We show the following resut.

\begin{maintheorem}[Theorem~\ref{thm: on ffrip}]\label{thm: N}
    The fundamental group of a graph of groups, where each vertex group has the strong \ffrip\ and each edge group is trivial or infinite cyclic, has the strong \ffrip
\end{maintheorem}

\subsection*{Outline of the paper}

Section \ref{sec: graphs and groups} sets the terminology and notation on the basic structures we manipulate, namely graphs and groups. It also contains (Section~\ref{ssec: intersection}) the precise definitions of the variants of the \fgip\ we consider.

Basics about graphs of groups, their fundamental groups, their morphisms and the representation of subgroups by immersions are summarised in Section~\ref{sec: graphs of groups}.

Section~\ref{sec: pullbacks} outlines the construction of the $\AA$-product of immersions into a graph of groups $\AA$, and the properties of this product that will be used in this paper. The content of this section is taken from \cite{dllrw_pullback}.

The notion of the strong free factor rank intersection property of a group $G$ is discussed in Section~\ref{sec: ffrip}, and Theorem~\ref{thm: N} is proved there.

Section~\ref{sec: fcip} contains the essential technical discussion of the finite coset interaction property, its variants, and their consequences on the structure of $\AA$-products. Proposition~\ref{prop: D} is a first example of the translation of the \fcip\ into finiteness properties of $\AA$-products. Turning to applications of the \fcip\ to the \fgip, Theorem~\ref{thm: main A} is proved there, using the notion of 0-\fcip\ The notion of 1-\fcip\ is mobilised to prove Theorem~\ref{thm: main B}. Theorem~\ref{thm: G} is proved using the notion of 2-\fcip\ Finally, Theorems~\ref{thm: main D} is proved using a restriction of the 1-\fcip\ to an orientation of the underlying graph of $\AA$, while Theorem~\ref{thm: main E} uses a further restriction of the 1-\fcip\ to an acyclic orientation.

The application to graphs of virtually $\ZZ$ groups with virtually $\ZZ$-edges (Theorem~\ref{thm: main M}) is discussed in Section~\ref{sec: vcyclic}, while Section~\ref{sec: hyperbolic} contains the applications to graphs of locally quasi-convex hyperbolic groups (Proposition~\ref{prop: C} and Theorems~\ref{thm: main F}, \ref{thm: main G} and \ref{thm: main H}). The latter uses crucially the results of Section~\ref{sec: vcyclic}.

\section{Graphs and groups}\label{sec: graphs and groups}

In Sections~\ref{ssec: graphs} and~\ref{ssec: on groups} we recall basic definitions and well-known concepts from graph theory and group theory, and set up conventions for the notation and terminology used throughout the article. In Section~\ref{ssec: intersection} we discuss the finitely generated intersection property and the variations which shall appear in this paper.

\subsection{On graphs}\label{ssec: graphs}

In this paper, by \emph{graph} we mean a \emph{directed graph}. Formally, a \emph{graph} $\gr{A}$ consists of a set of vertices $V(\gr{A})$, a set of edges $E(\gr{A})$, origin and target maps $o, t\colon E(\gr{A})\to V(\gr{A})$, and a fixpoint free involution $e \mapsto e\inv$  on $E(\gr{A})$, such that $t(e\inv) = o(e)$ (equivalently, $o(e\inv) = t(e)$) for every edge $e$; we say that $e^{-1}$ is the \emph{inverse edge} of $e$.

That is, in our graphs the set of edges can be partitioned into pairs $\{e,e^{-1}\}$ of mutually inverse edges.
An \emph{orientation} in a graph $\gr{A}$ is a subset $O \subseteq E(\gr{A})$ consisting of exactly one edge from each such pair; or, equivalently, satisfying $E(\gr{A}) = O \sqcup O^{-1}$. When an orientation $O$ is chosen in $\gr{A}$, it is usually denoted $E^+(\gr{A})$.

The \emph{star} of a vertex $v\in V(\gr{A})$ is defined as $\Star(v) = \{e \mid e\in E(\gr{A}), o(e) = v\}$, the collection of edges in $\gr{A}$ with origin $v$.

In a graph $\gr{A}$, a \emph{path} is a sequence $p= (e_1,\dots, e_k)$ of consecutive edges, that is, $t(e_i) = o(e_{i+1})$ for each $i \in [1,k-1]$. The \emph{length} of the path $p$ is the number $k$ of edges traversed by the path. The vertices $o(e_1)$ and $t(e_k)$ are called the \emph{initial} and \emph{terminal} vertices for a path \emph{from $o(e_1)$ to $t(e_k)$}. Two paths are said to be \emph{coterminal} if they have the same initial and terminal vertices. 
The \defin{inverse} of a path $(e_1,\ldots,e_k)$ is the path $(e_k^{-1},\ldots,e_1^{-1})$. Conventionally, we consider that there exists a \emph{trivial path} (of length 0) at every vertex of $\gr{A}$. The \emph{concatenation} of two consecutive paths is defined as follows: if $(e_1,\dots, e_k)$ and $(e_{k+1},\dots, e_\ell)$ are two paths with $t(e_k) = o(e_{k+1})$, then the \emph{concatenation} is the path $(e_1,\dots, e_k, e_{k+1},\dots, e_\ell)$, of length $k+\ell$. Trivial paths are neutral with respect to concatenation in the natural way. 

A \emph{backtracking} in $\gr{A}$ is a path consisting of two successive edges inverse of each other; that is, a (sub)path of the form $(e,e^{-1})$ is a \emph{backtracking at $t(e)$} for $e \in E(\gr{A})$. A path is said to be \emph{reduced} if it does not contain any backtracking. Two coterminal paths in a graph $\gr{A}$ are (\emph{homotopically}) \emph{equivalent} if one can be obtained from the other through a finite sequence of backtracking insertions and removals. Each equivalence class of paths contains a unique reduced path. Equivalence is compatible with concatenation, endowing the set of reduced paths in $\gr{A}$ with the structure of a groupoid, called the \emph{fundamental groupoid} of the graph $\gr{A}$ and denoted by $\pi_1(\gr{A})$.

A path $p = (e_1,\dots, e_k)$ is said to be \emph{closed} (or a \emph{circuit}) at $u\in V (\gr{A})$, if its initial and terminal vertices are equal to $u$, that is, if $o(e_1) = t(e_k) = u$. For any arbitrary vertex $u \in V(\gr{A})$, the restriction of $\pi_1(\gr{A})$ to reduced $u$-circuits is a group called the \defin{fundamental group} of $\gr{A}$ at $u$, denoted by $\pi_1(\gr{A},u)$.

\begin{defn}\label{def: core graphs}
Let $\gr{A}$ be a graph and let $p = (e_1,\dots, e_k)$ be an circuit in $\gr{A}$.
    
\begin{itemize}
    \item[(1)] $\gr{A}$ is said to be \emph{core with respect to a vertex} $u \in V(\gr{A})$ if every vertex in $\gr{A}$ appears in some reduced $u$-circuit.

    \item[(2)] The circuit $p = (e_1,\dots, e_k)$ is said to be \emph{cyclically reduced} if it is reduced and $e_k \neq e_1^{-1}$.

    \item[(3)] $\gr{A}$ is said to be \emph{core} if every vertex in $\gr{A}$ appears in some non-trivial cyclically reduced circuit.
\end{itemize}

\end{defn}

\begin{rem}
A graph that is core with respect to a vertex is always connected, while a core graph is not necessarily connected. The \emph{core of a graph $\gr{A}$ with respect to a vertex $u$} is the union of all reduced circuits based at $u$. The \emph{core of $\gr{A}$} is the union of all cyclically reduced circuits.
\end{rem}

All the maps between graphs considered in this paper are morphisms of graphs, that is, they map vertices to vertices and edges to edges, and they preserve  the adjacency relations.

We will use the following construction. Let $\gr A$, $\gr B$ and $\gr C$ be graphs and let $\mu^B\colon \gr B\to\gr A$ and $\mu^C\colon \gr C \to \gr A$ be graph morphisms. The \emph{fibered product}, or \emph{pullback}, of $\mu^B$ and $\mu^C$, written $\gr B \times_{\gr A} \gr C$, is the graph whose vertex and edge sets are 
$$\{(v,w) \in V(\gr B) \times V(\gr C) \mid \mu^B(v) = \mu^C(w)\} \quad\textrm{and}\quad\{(f,g) \in E(\gr B) \times E(\gr C) \mid \mu^B(f) = \mu^C(g)\},$$
respectively, with incidence maps given by $o(f,g) = (o(f),o(g))$ and $t(f,g) = (t(f),t(g))$.
The first and second component projections from $\gr B \times_{\gr A} \gr C$ to $\gr B$ and $\gr C$, respectively, are denoted by $\rho^B$ and $\rho^C$.

In the sequel, we use \texttt{sans-serif} typeface ($\gr{A},\gr{B},\gr{C}, \ldots$) to denote graphs.

\subsection{On groups}\label{ssec: on groups}
Most of the general notation and terminology for groups used in this paper is standard. However, we outline below some specific conventions adopted.

For any group $G$ and any subset $S$ of $G$, we denote by $\langle S\rangle$ the subgroup of $G$ generated by $S$. If $G = \langle S\rangle$ for some finite set $S$, we say that $G$ is \emph{finitely generated}. For a group $G$, we write $\rank(G)$ to denote the \emph{rank} of $G$, i.e., the minimal cardinal of a generating set for $G$.

We write $H \leqslant G$ to denote that $H$ is a subgroup of $G$.
If $H \leqslant G$, a \emph{left coset} (\resp \emph{right coset}) of $H$ is a set of the form $g H = \{gh \st h \in H \}$ (\resp $H g = \{hg \st h \in H \}$), where $g\in G$. We denote by $G/H$ (\resp $H\backslash G$) the set of left (right) cosets of $H$ in $G$.

Similarly, if $H,K \leqslant G$, an \emph{$(H,K)$-double coset} is a set of the form $H\,g\,K
=
\{ hgk \st h \in H \text{ and } k \in K\}$, where $g\in G$; and we denote by $\dblcoset{H}{G}{K}$ the set of $(H,K)$-double cosets in $G$. 

Given a group $G$ and elements $x,g\in G$, we write $x^g$ for the conjugate $g\inv xg$, and we denote by $\gamma_g$ the inner automorphism $x \mapsto x^g$ of $G$.

We also introduce the notion of \emph{fibered product}, or \emph{pullback}, for groups. If $A$, $B$ and $C$ are groups and $\mu^B\colon B\to  A$ and $\mu^C\colon C \to A$ are group morphisms, the \emph{fibered product} of $\mu^B$ and $\mu^C$ is the group $B \times_A C = \{(b,c) \in B \times C \mid \mu^B(b) = \mu^C(c)\}$, with the natural product (inherited from $B \times C$). The first and second component projection, from $B\times_AC$ to $B$ and $C$ are denoted by $\rho^B$ and $\rho^C$, respectively.

For convenience, we introduce the following notation. If $\mu^B$ and $\mu^C$ are as above, and if $a \in A$, we denote by $\twpb BaAC$ the pullback of $\gamma_a\circ\mu^B$ and $\mu^C$, that is, the set of pairs $(b,c) \in B\times C$ such that $a\inv\,\mu^B(b)\, a = \mu^C(c)$. 

\subsection{Properties related to intersections}\label{ssec: intersection}

Below, we summarise the main properties related to intersections of subgroups and cosets used in this paper (more specific or technical ones are provided in context in the corresponding sections).
\begin{defn}\label{defn: fgip}
    Let $G$ be a group and let $X \subseteq G$. The group $G$ is said to satisfy:
    \begin{enumerate}[(1)]
        \item the \emph{finitely generated intersection property (\fgip)}
        if the intersection of two finitely generated subgroups of $G$ is always finitely generated.
        \item the \emph{finitely generated intersection property relative to a subgroup} $H$ of $G$
        if, for every finitely generated subgroup $K$ of $G$, the intersection $K\cap H$ is finitely generated.
        \item the \emph{$X$-\fgip} if it satisfies the \fgip\ and for all finitely generated subgroups $H, K$ of $G$, there are finitely many double cosets $HgK$ that meet $X$ such that $H^g\cap K\neq 1$. Note that the \fgip\ is the same as the $\{1\}$-\fgip
        \item the \defin{strong finitely generated intersection property (\sfgip)} if it satisfies the $G$-\fgip, that is, $G$ satisfies the \fgip\ and for all finitely generated subgroups $H, K$ of $G$, there are finitely many double cosets $H\,g\,K$ such that $H^g\cap K\neq 1$.
    \end{enumerate}
Moreover, if $\mathcal{A}$ is a collection of subgroups of $G$, we say that
\begin{enumerate}[(1)]
\setcounter{enumi}{4}
    \item $(G,\mathcal{A})$ satisfies the \emph{$X$-\fgip}
    (or \sfgip\ if $X = G$) if $G$ satisfies the \fgip\ and, for all finitely generated subgroups $H, K$ of $G$, there exist only finitely many double cosets $HgK$ intersecting $X$ such that $g\in X$ and $H^g\cap K$ does not lie in a subgroup belonging to the family $\mathcal{A}$. Note that if $\mathcal{A}$ consists only of the trivial subgroup, we recover the previous notion. That is, $(G, \{1\})$ satisfies the $X$-\fgip\ precisely when $G$ has the $X$-\fgip
\end{enumerate}

\end{defn}
Note that $G$ has the \sfgip\ precisely if for all finitely generated subgroups $H,K$ of $G$, we have 
\[
\sum_{HgK \in H\backslash G /K} \rank(H^g \cap K) < \infty
\]
and $G$ has the $X$-\fgip\ precisely if for all finitely generated subgroups $H, K$ of $G$, we have
\[
\sum_{\substack{HgK \in H\backslash G /K\\ HgK\cap X\neq\emptyset}} \rank(H^g \cap K) < \infty
\]

\begin{exm} 
    Free groups have the \sfgip\ \cite{neu57,neu90}, see Remark~\ref{rem: uniform sfgip} below for more details.
    Every group trivially has the \fgip\ relative to its Noetherian subgroups.
\end{exm}

Hyperbolic groups provide many examples of groups with the \fgip\ and the \sfgip\ We refer the reader to Section \ref{sec: hyperbolic} for the definition of hyperbolic groups and quasi-convex subgroups. 

\begin{prop}
\label{prop: hyperbolic_sfgip_2}
If $G$ is a locally quasi-convex hyperbolic group, then:
\begin{enumerate}[(1)]
    \item $G$ has the \fgip\
    \item $(G, \mathcal{F})$ has the \sfgip, where $\mathcal{F}$ denotes the collection of finite subgroups of $G$.
    \item $G$ has the \sfgip\ if and only if $G$ does not contain any subgroup isomorphic to $\Z/k\Z\times\Z$ for any $k\geqslant 2$.
\end{enumerate}
\end{prop}

\begin{proof}
The first two statements follow directly from Short's Theorem \cite{sho91} and a result of Kharlampovich--Myasnikov--Weil \cite[Propositions 6.7 \& 6.9]{kmw17}. Now we establish the third. Suppose for a contradiction that there are finitely generated subgroups $H, K\leqslant G$ such that for infinitely many distinct double cosets $HgK$ the intersection $H^g\cap K$ is non-trivial. By \cite[Propositions 6.7 \& 6.9]{kmw17}, at most finitely many of these double cosets result in infinite intersections. Since $H$ and $K$ contain finitely many conjugacy classes of finite subgroups, by the pigeonhole principle, there are infinitely many distinct double cosets $Hg_1K, Hg_2K, \ldots$ so that $H^g_i\cap K$ lies in the same $H$-conjugacy class of finite subgroup for all $i$ and so that $H\cap K^{g_i^{-1}}$ lies in the same $K$-conjugacy class for all $i$. This implies that the finite subgroup $Q = H^{g_1}\cap K$ has infinite normaliser and hence that there is an infinite order element $g$ so that $Q^g = Q$. In particular, since $Q$ is finite, for $n$ large enough we have $q^{g^n} = q$ for each $q\in Q$. Thus, $\Z/k\Z\times\Z\leqslant G$ for some $k\geqslant 2$. Since $\Z/k\Z\times\Z$ does not have the \sfgip\ for any $k\geqslant 2$ (see Example \ref{ex: sfgip example} below), the proof is complete.
\end{proof}

\begin{exm}\label{ex: sfgip example}
In Proposition \ref{prop: hyperbolic_sfgip_2}, we cannot replace $(G, \mathcal{F})$ with $G$ in the case torsion is present. For example, for each $k\geqslant 2$, consider the (locally quasi-convex hyperbolic) group $\Z\times\Z/k\Z = \langle a, b \mid [a, b], b^k\rangle$. The double cosets $\langle b\rangle\, a^i\langle\, b\rangle$ are distinct for all $i\in \Z$, but $\langle b\rangle^{a^i}\cap \langle b\rangle = \langle b\rangle\neq 1$ for all $i\in \Z$ and so $\Z\times\Z/k\Z$ does not have the \sfgip\ On the other hand, $(\Z\times\Z/k\Z, \{\Z/k\Z\})$ has the \sfgip\
\end{exm}

\begin{rem}\label{rem: uniform sfgip}
We point the readers' attention to the fact that another notion of \emph{strongly Howson property}, occurs in \cite{ass15} with a different meaning involving a uniform bound on the rank of intersections. We would call it the \emph{uniform \fgip}: a group $G$ is said to have the \emph{uniform \fgip} if there exists a function $f$ on pairs of integers such that, if $H$ and $K$ are subgroups of finite rank at most $m$ and $n$, respectively, then the rank of $H\cap K$ is at most $f(m,n)$. Similarly, $G$ is said to have the \emph{uniform \sfgip} if there exists a function $f$ on pairs of integers, such that, if $H$ and $K$ are non-trivial subgroups of finite rank $m$ and $n$, respectively, then the sum of the ranks of the $H^g \cap K$, taken over the double cosets $H\,g\,K$ is at most $f(m,n)$. Some of our results  can be refined using this notion. We chose not to develop it, to keep this paper within reasonable bounds.

It is worth noting however that every free group $F$ has the uniform \fgip\ and the uniform \sfgip\ Indeed, Hanna Neumann proved~\cite{neu57} that, if $H$ and $K$ are non-trivial subgroups of $F$, then
$$\rank(H \cap K) - 1 \leq 2 (\rank(H) - 1)(\rank(K) - 1)$$
and she conjectured (the so-called Hanna Neumann Conjecture) that the factor 2 can be eliminated. This gave rise to an abundant literature, dealing with improved bounds and special cases. Walter Neumann then proved~\cite{neu90} that
\begin{equation}\label{equ: SHNC}
\sum\limits_{HgK\in H\backslash F_n/K} \hspace{-15pt} (\rank(H^g\cap K) -1) \leq 2 (\rank(H) - 1)(\rank(K) - 1),
\end{equation}
where the sum is restricted to the double cosets such that $H^g \cap K \ne 1$. This translates to a bound on the number of intersections $H^g \cap K$ which are neither trivial nor infinite cyclic. Walter Neumann also conjectured that the factor 2 could be dispensed with. This Strengthened Hanna Neumann Conjecture was established independently by Friedman~\cite{fri15} and by Mineyev~\cite{min12}. Finally, in \cite{Li21} it was shown that free groups have the uniform \sfgip, with the function $f(m, n)$ quadratic in both variables. It is unknown whether $f(m, n)$ can be taken to be linear in each variable.
\end{rem}

\section{Graphs of groups and their fundamental groups}\label{sec: graphs of groups}

In this section we shall provide an overview of basic facts about graphs of groups, following \cite{dllrw_pullback}, which itself is based on work of Bass \cite{bas93} and Serre \cite{ser80}. We begin by defining graphs of groups, $\AA$-paths and the fundamental group of a graph of groups. Then, in Section \ref{ssec: morphisms}, we introduce morphisms of graphs of groups, focusing on immersions. We conclude by stating Theorem \ref{thm: bijection subgroups immersions covers} which allows us to canonically realise any subgroup of the fundamental group of a graph of groups as an immersion of graphs of groups.

\subsection{Graphs of groups}\label{sec: define gog}

A \emph{graph of groups} $\AA = (\gr{A}, \{A_u\}, \{A_e\}, \{\alpha_e, \omega_e\})$ consists of
\begin{itemize}
\item[(1)] an underlying graph $\gr{A}$;
\item[(2)] a collection of groups $\{A_u\}$ indexed by $V(\gr{A})$, called the \emph{vertex groups};
\item[(3)] a collection of groups $\{A_e\}$ indexed by $E(\gr{A})$, such that $A_e = A_{e^{-1}}$, called the \emph{edge groups};
\item[(4)] and monomorphisms $\alpha_e\colon A_e\to A_{o(e)}$, $\omega_e\colon A_e\to A_{t(e)}$ called the \emph{edge maps}, satisfying $\alpha_{e} = \omega_{e^{-1}}$ and $\omega_e = \alpha_{e^{-1}}$. 
\end{itemize}
If $u$ is a vertex of $\gr A$, we say that $(\AA,u)$ is a \emph{pointed graph of groups}. The vertex $u$ is referred to as the \emph{basepoint}.

To facilitate the reading of this paper, we will systematically denote by $\gr A, \gr B, \gr C, \gr D$, etc, the underlying graphs of the graphs of groups $\AA, \BB, \CC, \DD$, etc.

\subsubsection{$\AA$-paths and core graphs of groups}

An \emph{$\AA$-path from $u$ to $u'$} (where $u, u'\in V(\gr A)$) is a sequence $p = (a_0, e_1, a_1, \dots, e_k, a_k)$, where $(e_1, e_2, \dots, e_k)$ is a path from $u$ to $u'$ in $\gr A$, each $a_{i-1}$ is in $A_{o(e_i)}$ and $a_k\in A_{t(e_k)}$ for all $i\in [1,k]$. We say that the $\AA$-path $p$ is closed if its \emph{underlying path} $(e_1,\dots,e_k)$ is closed, that two $\AA$-paths are coterminal if their underlying paths are coterminal, and that $p$ has length $n$ if its underlying path has length $n$. If $u\in V(\gr A)$, we say that the $\AA$-path $(1_{A_u})$, of length 0, is the \emph{trivial $\AA$-path at $u$}.

The concatenation of $\AA$-paths is defined as follows: if $p = (a_0, e_1, a_1, \dots, e_k, a_k)$ is an $\AA$-path from $u$ to $u'$ and $p' = (a'_0, e'_1, a'_1, \dots, e'_\ell, a'_\ell)$ is an $\AA$-path from $u'$ to $u''$, the concatenation $p\,p'$ is the $\AA$-path $(a_0, e_1, a_1, \dots, a_{k-1}, e_k, (a_ka'_0), e'_1, a'_1, \dots, e'_\ell, a'_\ell)$ from $u$ to $u''$.

The $\AA$-path $p$ is \emph{reduced} if, for all $1\le i < k$, we have either $e_{i+1} \ne e_i\inv$, or $e_{i+1} = e_i\inv$ and $a_i \notin \omega_{e_i}(A_{e_i})$. If $p$ is a closed $\AA$-path, we say that it is \emph{cyclically reduced} if it has positive length and $p^2$ is reduced.

If $u_0\in V(\gr A)$, the \emph{core of $\AA$} at $u_0$, written $\core(\AA,u_0)$ is the subgraph of groups of $\AA$ where we drop every vertex and edge that do not sit on a reduced closed $\AA$-path at $u_0$. The pointed graph of groups $(\AA,u_0)$ is said to be \emph{core} if it coincides with $\core(\AA,u_0)$.

Similarly, the \emph{core of $\AA$}, written $\core(\AA)$, is the subgraph of groups of $\AA$ where we drop every vertex and edge that do not sit on a cyclically reduced closed $\AA$-path. If $\AA = \core(\AA)$, we say that $\AA$ is \emph{core}.

\subsubsection{The fundamental group of a graph of groups}

We define the following two relations on $\AA$-paths:
\begin{itemize}
    \item The relation $\sim_\AA$ is the congruence\footnote{The set of $\AA$-paths under concatenation is a category, and we mean here that $\sim_\AA$ is a congruence for this algebraic structure.} generated by the pairs $\left((\alpha_e(x),e,\omega_e(x)\inv),\ (1,e,1)\right)$ ($e\in E(\gr A)$, $x\in A_e$). We note that $\sim_\AA$-equivalent $\AA$-paths must be coterminal and have the same length.

    \item The relation $=_\AA$ is the congruence generated by the pairs $\left((1,e,\omega_e(x),e\inv,1),\ (\alpha_e(x))\right)$ ($e\in E(\gr A)$, $x\in A_e$). Here, $=_\AA$-equivalent $\AA$-paths are coterminal, and their lengths differ by an even number.
\end{itemize}

The following is elementary (see \cite{bas93}): if $p \sim_A q$, then $p =_\AA q$. The converse holds if $p$ and $q$ are reduced. In addition, every $\AA$-path is $=_\AA$-equivalent to a reduced one. It is also immediately verified that, if $p = (a_0,e_1,\dots, a_k,e_k)$ is an $\AA$-path and $p\inv$ is the $\AA$-path $(a_k\inv, e_k\inv, \dots, e_1\inv, a_0\inv)$, then $p\,p\inv$ (resp. $p\inv\,p$) is $\AA$-equivalent to the trivial $\AA$-path at $o(p)$ (resp. $t(p)$).

As a result, the quotient of the set of $\AA$-paths by the relation $=_\AA$ is a groupoid, called the \emph{fundamental groupoid of $\AA$}, written $\pi_1(\AA)$. If $u_0 \in V(\AA)$, the set of closed $\AA$-paths at $u_0$ is a group, called the \emph{fundamental group of $\AA$ at $u_0$} and written $\pi_1(\AA,u_0)$.

The following statement will be used repeatedly.

\begin{prop}\label{prop: finite generation}
Let $(\AA, u_0)$ be a pointed graph of groups.
\begin{enumerate}[(1)]
    \item\label{item: generators of pi1} $\pi_1(\AA,u_0)$ is isomorphic to a quotient of the free product $F\ast(\bigast_{v\in V(\gr A)}A_v)$, where $F$ is the free group on $E(\gr A)$.
    \item\label{item: finite_gen_implies_finite_graph} If $\pi_1(\AA, u_0)$ is finitely generated, then $\core(\AA,u_0)$ has finite underlying graph.
    \item\label{item: finite generation} Suppose that the edge groups of $\AA$ are finitely generated. Then $\pi_1(\AA, u_0)$ is finitely generated if and only if the underlying graph of $\core(\AA,u_0)$ is finite and its vertex groups are finitely generated.
\end{enumerate}
\end{prop}

\begin{proof}
The first statement is elementary. For the second statement, we note that every element of $\pi_1(\AA,u_0)$ is $=_\AA$-equivalent to the concatenation of elements of a finite collection of reduced $\AA$-circuits at $u_0$: those representing the generators of $\pi_1(\AA,u_0)$ and their inverses. Thus any reduced $\AA$-circuit at $u_0$ is $=_\AA$ to such a concatenation, and hence uses vertices and edges from a finite subgraph of $\gr A$.

The third statement follows from \cite[Theorem 1.3]{hw21} (which the authors attribute to Dicks-Dunwoody \cite{dd89}), which states that, if a graph of groups has finite underlying graph, finitely generated edge groups and finitely generated fundamental group, then it also has finitely generated vertex groups.
\end{proof}

Finally, we say that an edge $e$ in a graph of groups is \emph{reduced} if it is a loop or $\alpha_e\colon A_e \to A_{o(e)}$ is not an isomorphism. A graph of groups $\AA$ is said to be \emph{reduced} if it is connected and every one of its edges is reduced. The following technical observation is well-known. The reader can find details in \cite[Section 2.2 \& Proposition 2.1]{df05}. 

\begin{prop}\label{prop: reduced gog}
If $\AA$ is a graph of groups with finitely many non-reduced edges and $v\in V(\gr A)$, then there is a pointed reduced graph of groups $(\AA',v')$ such that $\pi_1(\AA,v)$ is isomorphic to $\pi_1(\AA',v')$.
\end{prop}

\begin{rem}\label{rk: reduction is computable}
If the underlying graph $\gr A$ is finite, the procedure to go from $(\AA,v)$ to $(\AA',v')$ in Proposition~\ref{prop: reduced gog} is quite simple. It consists in repeatedly applying the following transformation: if $e_0$ is a non-reduced edge of $\gr A$, from $u$ to $u'$, delete the vertex $u$ and the edge $e_0$ and redirect every other edge $e \ne e_0, e_0\inv$ of $\gr A$ with target $u$, so that $t(e) = u'$ and set $\omega_{e} = \omega_{e_0}\circ \alpha_{e_0}\inv \circ \omega_e$. If $v\ne u$, let $v' = v$. If $v = u$, let $v' = u'$. Then $(\AA', v')$ is a pointed graph of groups with $\pi_1(\AA, v)\cong \pi_1(\AA', v')$. The procedure terminates after finitely many steps since $(\AA', v')$ is a graph of groups with strictly fewer non-reduced edges than $(\AA, v)$.

Suppose that the graph of groups $\AA$ is given effectively, in the following sense: the underlying graph $\gr A$ is finite, the vertex and edge groups are given explicitly, and one can decide whether a given morphism from an edge group to a vertex group is an isomorphism and, in that case, compute its inverse. Then there is an algorithm computing $(\AA',v')$.
\end{rem}

\begin{rem}\label{rk: reduced more general}
A more general statement than Proposition~\ref{prop: reduced gog} holds without the finiteness hypothesis, assuming only that $\AA$ does not have an infinite ray consisting only of non-reduced edges (traversed in their positive direction). We will not need this stronger statement in this paper.
\end{rem}

\subsection{Morphisms between graphs of groups, immersions and coverings}\label{ssec: morphisms}

We refer the reader to \cite{bas93} and to \cite[Sections 1.3 \& 4.2]{dllrw_pullback} for a more detailed discussion and for the proofs of the statements listed here.

\subsubsection{Morphisms}\label{sec: def morphisms}

Let $\AA = (\gr{A}, \{A_u\}, \{A_e\}, \{\alpha_e, \omega_e\})$ and $\BB = (\gr{B}, \{B_v\}, \{B_f\}, \{\alpha_f, \omega_f\})$ be graphs of groups. A \emph{morphism of graphs of groups} $\mu\colon\BB\to \AA$ consists of
\begin{itemize}
\item a morphism of underlying graphs $[.]\colon \gr{B}\to \gr{A}$ (sometimes also written $\mu$);
\item for every $v\in V(\gr{B})$, a monomorphism $\mu_v\colon B_v\to A_{[v]}$;
\item for every $f\in E(\gr{B})$, a monomorphism $\mu_f\colon B_f\to A_{[f]}$; and
\item for every edge $f\in E(\gr{B})$ from $v$ to $v'$, elements $f_{\alpha}\in A_{[v]}$, $f_{\omega}\in A_{[v']}$, called the \emph{twisting elements for $f$}, such that $(f\inv)_{\alpha} = f_{\omega}$ and $(f\inv)_\omega = f_\alpha$,
\end{itemize}
satisfying the following property: for each edge $e\in E(\gr{A})$, each edge $f\in E(\gr{B})$ from $v$ to $v'$ such that $[f] = e$, we have
\begin{align} 
\label{eq: twisted commutation alpha}
\alpha_e\circ\mu_{f} &= \inn{f_{\alpha}}\! \circ \mu_{v} \circ \alpha_{f}\\
\label{eq: twisted commutation omega}
\omega_e\circ\mu_f &= \inn{f_{w}} \! \circ \mu_{v'} \circ \omega_{f}.
\end{align}
If $\mu\colon \BB\to\AA$ is a morphism and $v\in V(\gr B)$, we say that $\mu\colon(\BB,v)\to(\AA,[v])$ is a \emph{morphism of pointed graphs of groups}. 

The morphism $\mu\colon \BB\to\AA$ naturally extends to a map from $\BB$-paths to $\AA$-paths: if
\[
p = (b_0, f_1, b_1, \dots, f_k, b_k)
\]
is a $\BB$-path, with $f_i$ an edge from $v_{i-1}$ to $v_i$ and $e_i = \mu(f_i) \in E(\gr{A})$, then $\mu(p)$ is the $\AA$-path
\[
\mu(p) = \left(\mu_{v_0}(b_0)\, (f_1)_\alpha,\, e_1,\, (f_1)\inv_\omega\,\mu_{v_1}(b_1)\, (f_2)_\alpha, \dots,\,  e_k,\,  (f_k)_\omega\inv\, \mu_{v_k}(b_k)\right).
\]
This map sends $=_\BB$-equivalent paths to $\AA$-equivalent paths, that is, it induces a morphism of groupoids $\mu_*\colon \pi_1(\BB) \to \pi_1(\AA)$. Moreover, if $v \in V(\gr{B})$ and $u = \mu(v)$, the restriction of $\mu_*$ to $\pi_1(\BB,v)$ is a group morphism $\mu_*\colon \pi_1(\BB,v) \to \pi_1(\AA,u)$.

\subsubsection{Immersions and coverings}\label{sec: immersions coverings}

Consider the following conditions, for a morphism of graphs of groups $\mu\colon \BB \to \AA$.
\begin{enumerate}[(1)]
    \item\label{item: immersion 1} for all edges $f, f' \in E(\gr{B})$ with the same image $e = \mu(f) = \mu(f') \in E(\gr{A})$ and the same initial vertex $v = o(f) = o(f')$, then $f = f'$ if and only if $\mu_v(B_v)\, f_{\alpha}\, \alpha_e(A_e) = \mu_v(B_v)\, f_{\alpha}'\, \alpha_e(A_e)$;

    \item\label{item: immersion 2} for every edge $f \in V(\gr B)$, if $e = \mu(f)$ and $v = o(f)$, then $(\alpha_e\circ\mu_f)(B_f) = \mu_v(B_v)^{f_{\alpha}}\cap \alpha_e(A_e)$.

    \item\label{item: covering 0} for every vertex $v\in V(\gr B)$, every edge $e\in E(\gr A)$ such that $o(e) = \mu(v)$ and every double coset $\mu_v(B_v)\,a\,\alpha_e(A_e) \in \dblcoset{\mu_v(B_v)}{A_e}{\alpha_e(A_e)}$, there exists $f\in \gr B$ such that $o(f) = v$, $\mu(f) = e$ and $\mu_v(B_v)\, f_{\alpha}\, \alpha_e(A_e) = \mu_v(B_v)\, a\, \alpha_e(A_e)$.    
\end{enumerate}

\begin{defn}\label{def: folded}
Let $\mu\colon \BB\to\AA$ be a morphism of graphs of groups. If Conditions~\eqref{item: immersion 1} and~\eqref{item: immersion 2} hold, we say that $\mu$ is an \emph{immersion}. If all three conditions hold, then we say that $\mu$ is a \emph{covering}.
\end{defn}

\begin{rem}
\label{rem: folded}
    The definition of immersions and coverings above is, in fact, equivalent to that of
    Bass \cite{bas93}, see the discussion in \cite[Remarks 1.12 \& 4.7]{dllrw_pullback}.
\end{rem}

Immersions and coverings have the following important properties, see \cite{bas93}.

\begin{prop}\label{folded}\label{prop: folded morphisms}
    A morphism $\mu\colon \BB\to \AA$ of graphs of groups is an immersion if and only if it sends reduced $\BB$-paths to reduced $\AA$-paths.

    The composition of two immersions is an immersion.

    If $\mu\colon \BB \to \AA$ is an immersion, $v\in V(\gr B)$ and $u = \mu(v)$, then the morphism $\mu_*\colon \pi_1(\BB,v) \to \pi_1(\AA,u)$ is injective.
\end{prop}

\subsubsection{Realising subgroups as immersions of graphs of groups}

The following statement is a combination of classical results from \cite{bas93} and of \cite[Corollaries 4.14 \& 4.16]{dllrw_pullback}. The uniqueness statements in~\eqref{item: unique covering} and~\eqref{item: unique immersion} below are translations of category-theoretic uniqueness statements, in the category of connected pointed graphs of groups investigated in \cite{dllrw_pullback}. Statement~\eqref{item: fg vertex and edge groups}, due to Haglund and Wise \cite{hw21}, generalises results by Karrass and Solitar \cite{ks70,ks71}.

\begin{thm}\label{thm: bijection subgroups immersions covers}
    Let $A = \pi_1(\AA,u_0)$ be the fundamental group of a core pointed graph of groups $(\AA,u_0)$ and let $B$ be a subgroup of $A$.
    \begin{enumerate}[(1)]
        \item\label{item: covering} There exists a covering $\mu^B\colon (\BB,v_0) \to (\AA,u_0)$ such that $\gr B$ is connected and $\mu^B_*(\pi_1(\BB,v_0)) = B$.
        \item\label{item: realisation} There exists an immersion $\nu^B\colon (\BB',v'_0) \to (\AA,u_0)$ such that $(\BB',v'_0)$ is core and such that $\nu^B_*(\pi_1(\BB',v'_0)) = B$.
        \item\label{item: unique covering} The covering $\mu^B\colon (\BB,v_0) \to (\AA,u_0)$ is unique up to isomorphism in the following sense: if $\mu^C\colon (\CC,w_0) \to (\AA,u_0)$ is a covering such that $\gr C$ is connected and $\mu^C_*(\pi_1(\CC,w_0)) = B$, then there exist morphisms of graphs of groups $\sigma\colon (\BB,v_0) \to (\CC,w_0)$ and $\tau\colon (\CC,w_0) \to (\BB,v_0)$ which are mutually reciprocal isomorphims on the underlying graph and on each vertex and edge group, and which satisfy $\mu^B_* = \mu^C_* \circ \sigma_*$ and $\mu^C_* = \mu^B_* \circ \tau_*$.
        \item\label{item: unique immersion} The pointed graph of groups $(\BB',v'_0)$ can be taken to be $\core(\BB,v_0)$, with $\nu^B$ the restriction of $\mu^B$ to $\core(\BB,v_0)$. In particular, $\nu^B$ satisfies a uniqueness statement analogous to that for $\mu^B$.
        \item\label{item: fg vertex and edge groups} If $B$ is finitely generated, then $\core(\BB)$ is a finite graph of groups. If, additionally, for each edge $e\in E(\gr A)$ and for each reduced $\AA$-path $p$ from $u_0$ to $o(e)$, the intersection $B \cap (p\, \alpha_e(A_e)\, p\inv)$ is finitely generated, then all vertex and edge groups of $\BB$ are finitely generated.
    \end{enumerate}
\end{thm}

\begin{rem}
    The covering $\mu^B\colon (\BB,v_0) \to (\AA,u_0)$ whose existence is asserted in Theorem~\ref{thm: bijection subgroups immersions covers}~\eqref{item: covering}, is called the \emph{$B$-cover}. It is unique up to isomorphism, as explained in Theorem~\ref{thm: bijection subgroups immersions covers}~\eqref{item: unique covering}.
    Bass \cite{bas93} gives an explicit construction of $(\BB,v_0)$ and $\mu^B$, see also \cite[Section 4.2]{dllrw_pullback}.
\end{rem}

Finally, we shall also need the following slight improvement to Theorem \ref{thm: bijection subgroups immersions covers}~\eqref{item: realisation}.

\begin{lem}
\label{lem: conjugate subgroup}
    Let $(\AA, u_0)$ be a pointed graph of finitely generated groups so that each vertex group of $\AA$ has the \fgip\ relative to its adjacent edge groups. Let $B_1, B_2\leqslant \pi_1(\AA, u_0)$ be conjugate subgroups and let $\BB_1', \BB_2'$ be the (unique) core graphs of groups from Theorem \ref{thm: bijection subgroups immersions covers}. Then $\BB_1'$ is a finite graph of finitely generated groups if and only if $\BB_2'$ is.
\end{lem}

\begin{proof}
    Let us assume that $\BB_1'$ is a finite graph of finitely generated groups. We need to show that $\BB_2'$ is a finite graph of finitely generated groups and we will be done by symmetry. We have two cases to consider, the first is when $B_1$ contains only elements which conjugate into vertex groups (locally elliptic case) and the second is when $B_1$ contains an element which does not conjugate into a vertex group (hyperbolic case). 
    
    In the first case, since $\BB_1'$ is a finite graph of finitely generated groups, both groups $B_1, B_2$ are finitely generated and so must conjugate entirely into a single vertex group by a result of Bass \cite[Corollary 7.3]{bas93}. Thus, $\BB_1'$ and $\BB_2'$, being pointed core, must both have underlying graphs isomorphic to a line leading out from the basepoint to vertices $v_1', v_2'$ respectively and such that each edge $f$ leading away from the basepoint has $\alpha_f$ an isomorphism. It follows that $B_1$, $B_{v_1'}$, $B_{v_2'}$ and $B_2$ are isomorphic and, hence, the vertex groups $B_{v_1'}$ and $B_{v_2'}$ are finitely generated. If $\BB_2'$ has no edges, then we are done. Otherwise, let $f\in E(\gr{B}_2)$ be the edge with $t(f) = v_2'$. By Property \eqref{item: immersion 2} in the definition of an immersion of graphs of groups, the groups $B_{f}$ and $\omega_{[f]}(A_{[f]})\cap B_{v_2'}$ are isomorphic. Since vertex groups of $\AA$ have the \fgip\ relative to their adjacent edge groups by assumption, $B_{f}$ is finitely generated. Since $B_{o(f)} = \alpha_f(B_f)$, we see that $B_{o(f)}$ is finitely generated. Repeating this argument we see that $\BB_2'$ is also a finite graph of finitely generated groups.

    In the second case, we use \cite[Corollary 4.18]{dllrw_pullback} to conclude that $\core(\BB_1')$ and $\core(\BB_2')$ are non-empty and isomorphic.
    Since $\core(\BB_1')$ is a subgraph of groups of $\BB'_1$, it follows that $\core(\BB_2')$ is a finite graph of finitely generated groups. By \cite[Remark 4.4]{dllrw_pullback},
     $\BB_2' = \mathbb{L}\cup \core(\BB_2')$ where $\mathbb{L}$ has underlying graph a finite line connecting the basepoint with $\core(\BB_2')$, such that for each edge $f\in E(\gr{L})$ pointing away from the basepoint, the map $\alpha_f$ is an isomorphism. Now we may conclude exactly as before that $\mathbb{L}$ must also be a graph of finitely generated groups. It follows that $\BB_2'$ is a finite graph of finitely generated groups, as claimed.
\end{proof}

\section{Products of graphs of groups and intersections}\label{sec: pullbacks}
\def\calC{\mathcal{C}}

Let $\AA$, $\BB$ and $\CC$ be graphs of groups and let $\mu^B\colon \BB \to \AA$ and $\mu^C\colon \CC \to \AA$ be immersions.
The  \emph{$\AA$-product} of $\mu^B$ and $\mu^C$, written $\BB\wtimes_\AA\CC$, was introduced in \cite{dllrw_pullback}. Its definition is rather technical and we highlight here only the notions that will be important for our study of the \fgip\ The reader should have in mind that this graph of groups, together with natural morphisms $\rho^B, \rho^C\colon \BB\wtimes_{\AA}\CC\to \BB, \CC$, can be thought of as the pullback.

For each vertex $(v,w)$ and edge $(f,g)$ of $\gr B\times_{\gr A}\gr C$, if $u = \mu^B(v) = \mu^C(w) \in V(\gr A)$ and $e = \mu^B(f) = \mu^C(g) \in E(\gr A)$, we let
\begin{align*}
	V_{v,w}(\gr D) &= \dblcoset{B_v}{A_u}{C_w} = \{ \mu_v^B(B_v)\,a\,\mu_w^C(C_w) \mid a\in A_u \},\\
	 E_{f,g}(\gr D) &=  \dblcoset{B_f}{A_e}{C_g} = \{ \mu_f^B(B_f)\,a\,\mu_g^C(C_g) \mid a\in A_e \}.
\end{align*}
The underlying graph $\gr D$ of $\BB\wtimes_\AA\CC$ has vertex set the disjoint union of the $V_{v,w}(\gr D)$ and edge set the disjoint union of the $E_{f,g}(\gr D)$. The incidence maps are the following. If $e \in E(\gr{A})$ and $(f,g) \in E(\gr{B} \times_\gr{A} \gr{C})$ is an edge from $(v,w)$ to $(v',w')$ such that $[f] = [g] = e$ and $a\in A_e$, then the origin and terminal vertices of the $(f,g)$-edge $\mu_f^B(B_f)\,a\,\mu_g^C(C_g)$ are given by
\begin{equation} \label{eq: incidence in D}
\begin{aligned} 
o\left(\mu_f^B(B_f)\,a\,\mu_g^C(C_g)\right) &= \mu_v^B(B_v)\,(f_{\alpha}\,\alpha_e(a)\,g_{\alpha}^{-1})\,\mu_w^C(C_w),\\
t\left(\mu_f^B(B_f)\,a\,\mu_g^C(C_g)\right) &= \mu_{v'}^B(B_{v'})\,(f_{\omega}\,\omega_e(a)\,g_\omega^{-1})\,\mu_{w'}^C(C_{w'}).
\end{aligned}
\end{equation}
This definition seems to depend on the choice of the representative $a$ of the double coset defining the edge; it is verified in \cite[Def. 3.1]{dllrw_pullback} (or directly, using the definition of morphisms) that it does not.

For each vertex $x\in V_{v,w}(\gr D)$ and edge $h_{f,g}\in E(\gr D)$, we choose representatives $\tilde x$ and $\tilde h$ of the corresponding double cosets. We then let
$$D_x = \twpb{B_v}{\tilde x}{A_u}{C_w}\quad\textrm{and}\quad D_h = \twpb{B_f}{\tilde h}{A_e}{C_g},$$
where $u$ is the image of $v$ and $w$ in $V(\gr A)$ and $e$ is the image of $f$ and $g$ in $E(\gr A)$. The projection morphisms $\rho^B\colon \DD \to \BB$ and $\rho^C\colon \DD\to \CC$ on the vertex and edge groups are given by the projection morphisms associated with the fibered products defining the vertex and edge groups of $\DD$. This definition yields the following isomorphisms
$$D_x \simeq \mu^B_v(B_v)^{\tilde x} \cap \mu^C_w(C_w)\textrm{ and }D_h \simeq \mu^B_f(B_f)^{\tilde h} \cap \mu^C_g(C_g).$$

Since we are not going to use them, we refer the reader to \cite[Definition 3.4]{dllrw_pullback} for the definition of the edge maps from edge groups to vertex groups, as well as for the twisting elements of $\rho^B$ and $\rho^C$.

\begin{rem}
The technical definition of $\AA$-products in \cite[Section 3.1]{dllrw_pullback} depends on the choice of several parameters. However, if $\DD$ and $\DD'$ are $\AA$-products which result from different choices (with projection morphisms $\rho^B, \rho^C$ and $\rho'^B, \rho'^C$, respectively), $\DD$ and $\DD'$ have the same underlying graph $\gr D$ and they satisfy the following property \cite[Lemma 3.5]{dllrw_pullback}: if $x \in V_{v,w}(\gr D)$ and $u = \mu^B(v) = \mu^C(w)$, then there exist immersions $\sigma\colon (\DD,x) \to (\DD',x)$ and $\sigma'\colon (\DD',x) \to (\DD,x)$ such that $\sigma_*$ and $\sigma'_*$ are mutually inverse isomorphisms between $\pi_1(\DD,x)$ and $\pi_1(\DD',x)$, $(\rho^B)_* = (\rho'^B)_*\circ\sigma_*$ and $(\rho^C)_* = (\rho'^C)_*\circ\sigma_*$.
\end{rem}

If $\AA$, $\BB$ and $\CC$ are equipped with base vertices $u_0$, $v_0$ and $w_0$, such that $\mu^B(v_0) = \mu^C(w_0) = u_0$, we let $x_0 = \mu_{v_0}^B(B_{v_0})\,\mu_{w_0}^C(C_{w_0}) \in V_{v_0,w_0}(\gr D)$ be the associated base vertex. We refer to $(\BB\wtimes_\AA\CC,x_0)$ as the \emph{pointed $\AA$-product} of $\mu^B$ and $\mu^C$ (now seen as morphisms of pointed graphs of groups).

The main properties of $\AA$-products which we will use are summarised in the following statements. The first one is a combination of Lemma 3.8 and Corollary 3.17(2)(3) in \cite{dllrw_pullback}.

\begin{thm}\label{thm: product of immersions}
Let $\mu^B\colon (\BB,v_0) \to (\AA,u_0)$ and $\mu^C\colon (\CC,w_0) \to (\AA,u_0)$ be immersions between connected pointed graphs of groups, let $(\DD,x_0)$ be their pointed $\AA$-product, with projection morphisms $\rho^B \colon \DD\to \BB$ and $\rho^C\colon \DD\to \CC$. Then,
\begin{enumerate}[(1)]
\item $\rho^B$ and $\rho^C$ are immersions.
\item\label{item: intersection of subgroups} $(\mu^B\circ\rho^B)_*(\pi_1(\DD, x_0)) = (\mu^C\circ\rho^C)_*(\pi_1(\DD, x_0)) = \mu^B_*(\pi_1(\BB, v_0))\cap \mu^C_*(\pi_1(\CC, w_0))$.
\item \label{item: localy_elliptic_double_cosets} Let $x\in V(\DD)$, $v = \rho^B(x)$ and $w = \rho^C(x)$. Let also $p_v$ be a reduced $\BB$-path from $v_0$ to $v$ and $q_w$ be a reduced $\CC$-path from $w_0$ to $w$. Then we have 
    \[
    \mu^C_*\circ\rho^C_*(\pi_1(\DD, x))^{\mu^C(q_w)^{-1}} = \mu^B_*(\pi_1(\BB, v_0))^{\left(\mu^B(p_v)\tilde{x}\mu^C(q_w)^{-1}\right)}\cap \mu^C_*(\pi_1(\CC, w_0))
    \]
\end{enumerate}
\end{thm}

\def\C{\mathcal{C}}

Before giving the next statement, we need a few definitions. Let $B$ be a subgroup of the fundamental group $\pi_1(\AA,u_0)$, for some pointed graph of groups $(\AA,u_0)$. We say that $B$ is \emph{locally elliptic} if each of its elements conjugates to a vertex group element, that is: if, for each $b\in B$, there is an $\AA$-path $p$ starting at $u_0$ such that $p\inv\, b\, p \in A_{t(p)}$.

Let $\mu^B\colon (\BB,v_0) \to (\AA,u_0)$ and $\mu^C\colon (\CC,w_0) \to (\AA,u_0)$ be immersions between connected pointed graphs of groups, and let $(\DD,x_0)$ be their pointed $\AA$-product, with projection morphisms $\rho^B \colon \DD\to \BB$ and $\rho^C\colon \DD\to \CC$. Let $A = \pi_1(\AA,u_0)$, and $B = \mu^B_*(\pi_1(\BB,v_0)) = \mu^C_*(\pi_1(\CC,w_0))$. For each vertex $v\in V(\gr B)$, let $p_v$ be a fixed $\BB$-path from $v_0$ to $v$. Similarly, for each $w\in V(\gr C)$, fix a $\CC$-path $q_w$ from $w_0$ to $w$. If $(v,w)$ is a vertex of $\gr B \times_{\gr A}\gr C$, we let $\C_{v,w}$ be the following map:
\begin{equation} \label{eq: Cvw}
\begin{aligned}
\C_{v,w}\colon\enspace V_{v,w}(\gr D) &\enspace\longrightarrow\enspace \dblcoset BAC \\
\mu^B(B_v)\, a\, \mu^C(C_w) &\enspace\longmapsto\enspace B\, (\mu^B(p_v)\, a\, \mu^C(q_w)^{-1})\, C.
\end{aligned}
\end{equation}
It is shown in \cite[Section 4.3]{dllrw_pullback} that $\C_{v,w}$ does not depend on the choice of the $\BB$-paths $p_v$ and the $\CC$-paths $q_w$. Using the fact that $V(\gr D)$ is the disjoint union of the $V_{v,w}(\gr D)$, we then define
\begin{equation}\label{eq: C}
\begin{aligned}
\C \colon V(\gr D) &\enspace\longrightarrow\enspace \dblcoset BAC\\
            x &\enspace\longmapsto\enspace \C_{\rho^B(x), \rho^C(x)}(x).
\end{aligned}
\end{equation}

The following statement is a combination of Lemmas 4.19 and 4.20, of Theorems 4.21 and 4.22 and of Proposition 4.23 in \cite{dllrw_pullback}.

\begin{thm}\label{thm: summary immersions}
Let $\mu^B\colon (\BB,v_0) \to (\AA,u_0)$ and $\mu^C\colon (\CC,w_0) \to (\AA,u_0)$ be immersions between connected pointed graphs of groups, and let $A = \pi_1(\AA,u_0)$, $B = \mu^B_*(\pi_1(\BB,v_0)) = \mu^C_*(\pi_1(\CC,w_0))$.  Let also $\DD = \BB\wtimes_\AA\CC$ be the $\AA$-product of $\mu^B$ and $\mu^C$, with projection morphisms $\rho^B \colon \DD\to \BB$ and $\rho^C\colon \DD\to \CC$. Denote by $\pi_0(\DD)$ the set of connected components of $\DD$.
\begin{enumerate}[(1)]
\item \label{item: same_double_coset} Let vertices $x,x'$ be vertices in $V(\gr D)$. Then $x$ and $x'$ are in the same connected component if and only if $\C(x) = \C(x')$. In particular, $\C$ induces an injective map $\C_0\colon \pi_0(\DD) \to \dblcoset BAC$.
\item \label{item: components_bijection}$\C_0$ induces a bijection from $\pi_0(\core(\DD))$ to the set of double cosets $B\,g\,C$ ($g \in A$) such that $B^g\cap C$ is not locally elliptic.
\item \label{item: conjugate_into_edge_group}If $B\,g\,C\notin \calC(V(\BB\wtimes_{\AA}\CC))$, then $B^g\cap C$ conjugates into an edge group of $\AA$.
\item If $\mu^B$ is a covering, then the map $\C_0$ is a bijection. Moreover, $\C_0$ also establishes a bijection from the set of connected components of $\DD$ with non-trivial fundamental group and the set of double cosets $B\,g\,C$ such that $B^g \cap C \ne 1$.
\item For each $\DD'\in \pi_0(\DD)$, $x\in V(\gr{D}')$ and each $g\in \mathcal{C}_0(\DD')$, we have $\pi_1(\DD', x')\cong B^g\cap C$.
\end{enumerate}
\end{thm}

Our last statement concerns the case of acylindrical graphs of groups. Let $k\ge 1$. A graph of groups $\AA$ is \emph{$k$-acylindrical} if, for every reduced $\AA$-path $p$ of length $k$ and every non-trivial element $a\in A_{t(p)}$, the $\AA$-path $pap^{-1}$ does not reduce to an element of $A_{o(p)}$. It is \emph{acylindrical} if it is $k$-acylindrical for some $k$. This is equivalent to the classical definition in terms of acylindrical actions on trees, see \cite[Lemma 4.24]{dllrw_pullback}. The statement below is a direct consequence of \cite[Lemma 4.24, Proposition 4.26 \& Theorem 4.27]{dllrw_pullback}.

\begin{thm}
\label{thm: acylindrical}
    Let $(\AA,u_0)$ be a pointed graph of groups, let $B\leqslant \pi_1(\AA, u_0)$ be a finitely generated subgroup and let $\mu\colon (\BB, v_0)\to (\AA, u_0)$ be the corresponding $B$-cover. If the following conditions hold:
    \begin{itemize}
        \item $\BB$ is acylindrical and $B_e$ is finitely generated for every $e\in E(\core(\BB))$,
        \item for each $e\in E(\AA)$, $A_{o(e)}$ has the \sfgip\ relative to $\alpha_e(A_e)$,
    \end{itemize}
    then there exists a pointed graph of groups $(\BB',v'_0)$ such that $\BB'$ has non-empty core, $\gr B'$ is finite and all the vertex and edge groups of $\BB'$ are finitely generated, and there exists an immersion $\mu^B\colon (\BB',v'_0) \to (\AA,u_0)$ such that, if $\mu^C\colon (\CC,w_0)\to (\AA,u_0)$ is an immersion where $\CC$ has non-empty core, $C = \mu^C_*(\pi_1(\CC,w_0))$ and $\DD$ is the union of the components of $\BB'\wtimes_{\AA}\CC$ with non-trivial fundamental group, then the map $\C_0$ establishes a bijection between $\pi_0(\DD)$ and the set of double coset $B\,g\,C$ such that $B^g \cap C \ne 1$.
\end{thm}

\section{The finite factor rank intersection property}\label{sec: ffrip}

We start with an investigation of the \emph{finite factor rank intersection property} (\ffrip), which evaluates the maximum rank of a freely indecomposable free factor in a group $A$, and of a strong version of that property. We show that a graph of groups where each vertex group has the strong \ffrip\ and with trivial or infinite cyclic edge groups also has the strong \ffrip

Recall that a group $A$ is said to \emph{split as a proper free product} if $A$ has non-trivial subgroups $B$ and $C$ such that $A = B * C$. The group $A$ is said to be \defin{freely indecomposable} otherwise.  A \emph{Grushko decomposition} of $A$ is a free product splitting of the form $A = F*(\bigast_{\alpha}A_{\alpha})$, where $F$ is a free group and each of the groups $A_{\alpha}$ is non-trivial, freely indecomposable, and not infinite cyclic. If such a decomposition exists, it is always unique up to permutation of the factors and conjugation. In particular, the sequence of isomorphism types of the factors is unique up to permutation (see, for instance, \cite{sta65} for the finitely generated case and \cite{imr84} for the infinitely generated case). The so-called \emph{Grushko decomposition theorem} asserts that any non-trivial finitely generated group admits such a (finite) decomposition.

\begin{defn}
    Let $A$ be a group which admits a Grushko decomposition. The \defin{factor rank} of $A$, denoted by $\frank(A)$, is the supremum of the ranks of the non-free factors in a Grushko decomposition of $A$ or, equivalently,
\[
\frank(A)
\,=\,
\sup\left\{ \,
\rank(B) \mid A=B*C \text{, $B\ne \Z$ and $B$ is freely indecomposable}\,
\right\}
\in \mathbb{N}\cup\{+\infty\}.
\]
Note that $\frank(A) \le \rank(A)$, and the factor rank of a group admitting a Grushko decomposition is $0$ if and only if it is free (maybe trivial). Moreover, if $A \ne \Z$ is freely indecomposable, then $\frank(A) = \rank(A)$.

We say that $A$ satisfies the  \emph{finite factor rank intersection property} (\ffrip) if, for every pair of finitely generated subgroups $B$ and $C$ of $A$,
$\frank(B\cap C)$ is well-defined and finite.
Similarly, $A$ is said to satisfy the \defin{strong finite factor rank intersection property} (strong \ffrip) if, for every pair of finitely generated subgroups $B$ and $C$ of $A$ and every $a \in A$,
$\frank(B^a\cap C)$ is well-defined and $\sum\frank(B^a\cap C) <\infty$, where the sum runs over all double cosets $B\,a\,C$.
\end{defn}

In order to investigate the inheritability of these properties from the vertex groups to the fundamental group of a graph of groups, we need a few technical propositions concerning free product decompositions.
The first of these results is elementary.

\begin{prop}\label{prop: splitting over trivial edge groups}
Let $\AA$ be a connected graph of groups, $u$ a vertex of $\gr{A}$ and $e$ an edge of $\gr A$ such that $A_e = 1$.
\begin{enumerate}[(1)]
\item If $\gr{A} - e$ is connected, denote by
$\AA_{\bar{e}}$  the subgraph of groups of $\AA$ whose underlying graph is $\gr{A}-e$. Then $\pi_1(\AA, u)$ is isomorphic to $\pi_1(\AA_{\bar{e}}, u)*\ZZ$, where the $\ZZ$ factor is generated by a $u$-circuit
in $\AA$ traversing $e$.
\item\label{item: split pi1} If $\gr{A}-e$ is not connected, denote by $\AA_{\bar{e}, 1}$ and $\AA_{\bar{e}, 2}$ the subgraphs of groups of $\AA$ whose underlying graphs are the two components of $\gr{A}-e$. Then $\pi_1(\AA, u)$ is isomorphic to the free product $\pi_1(\AA_{\bar{e}, 1}, u_1)*\pi_1(\AA_{\bar{e}, 2}, u_2)$, where $u_1$ and $u_2$ are arbitrary vertices in the corresponding connected components. 
\item\label{item: split pi1 many trivial edges}  If $E_{1}\subseteq E(\gr{A})$ is a subset of edges with $A_e = 1$ for all $e\in E_1$ and if $\{\AA_{\alpha}\}$ is the collection of subgraphs of groups of $\AA$
with underlying graphs the connected components of $\gr{A} - E_{1}$, then $\pi_1(\AA, u)$ is isomorphic to
\[
F*\left(\bigast_{\alpha}\pi_1(\AA_{\alpha}, u_{\alpha})\right),
\]
where $F$ is a free group and $u_{\alpha}$ is any choice of a basepoint in $\AA_{\alpha}$.
\end{enumerate}
\end{prop}

\begin{prop}\label{prop: split DDi}
Let $(\AA,u_0)$ be a connected pointed graph of groups. Suppose that, for each $u\in V(\gr A)$, the vertex group $A_u$ splits as a free product $A_u = A'_u * A''_u$ in such a way that $A'_u$ contains $\alpha_e(A_e)$ for each edge $e\in E(\gr A)$ originating at $u$. Let $\AA'$ be obtained from $\AA$ by replacing each vertex group $A_u$ by $A'_u$.
Then $\pi_1(\AA,u_0)$ is isomorphic to the free product $\pi_1(\AA',u_0) *\bigast_{u\in V(\gr A)} A''_u$.
\end{prop}

\begin{proof}
Let $\hat\AA$ be the graph of groups obtained from $\AA$ by adding new vertices $\hat u$ and new edges $e_u$ from $u$ to $\hat u$ for each $u\in V(\gr A)$, and taking vertex and edge groups as follows: $\hat A_u = A'_u$, $\hat A_{\hat u} = A''_u$ and $\hat A_{e_u} = 1$ for each $u \in V(\gr A)$, and $\hat A_e = A_e$ for each $e\in E(\gr A)$. By \cite[Proposition 2.1]{df05} (or \cite[Section 2]{bf91}), the fundamental groups $\pi_1(\AA,u_0)$ and $\pi_1(\hat\AA,u_0)$ are isomorphic. Moreover, Proposition~\ref{prop: splitting over trivial edge groups}~\eqref{item: split pi1} shows that $\pi_1(\hat\AA,u_0)$ splits as the free product of $\pi_1(\hat\AA',u_0)$ and the $A''_u$ ($u\in V(\gr A))$.
\end{proof}

The main result in this section is the following.

\begin{thm}\label{thm: on ffrip}
Suppose that a group $A$ splits as the fundamental group of a graph of groups $\AA$ whose vertex groups satisfy the strong \ffrip\ and whose edge groups are trivial or infinite cyclic. Then $A$ satisfies the strong \ffrip~as well.
\end{thm}

\begin{proof}
Let $(\AA, u_0)$ be a connected pointed graph of groups whose vertex groups satisfy the strong \ffrip\ and whose edge groups are trivial or infinite-cyclic and let $A = \pi_1(\AA, u_0)$. Let $B$ and $C$ be arbitrary finitely generated subgroups of $A$. We want to show that the sum $\sum \frank(B^a \cap C)$ is well-defined and finite, where the sum runs over the double cosets $B\, a\, C$.
    
By Proposition~\ref{thm: bijection subgroups immersions covers}~\eqref{item: fg vertex and edge groups}, there exist core pointed graphs of groups $(\BB,v_0)$ and $(\CC,w_0)$, with finite underlying graphs and finitely generated vertex and edge groups, and there exist immersions $\mu^B\colon (\BB,v_0)\to (\AA,u_0)$ and $\mu^C\colon (\CC,w_0)\to (\AA,u_0)$ such that $B = \mu^B_*(\BB,v_0)$ and $C = \mu^C_*(\CC,w_0)$.
Let $\DD = \BB\wtimes_\AA\CC$. We proceed with two claims.

\begin{claim}\label{claim: vertex groups have bounded factor rank}
The sum $\sum_x\frank(D_x)$, where $x$ runs over the vertices $x\in V(\gr D)$, is finite. In particular, $D_x$ is free for all but a finite number of vertices of $\gr D$.
\end{claim}

\begin{proof}
Let $x \in V(\gr D)$. Suppose that $x\in V_{v,w}(\gr D)$ with $v\in V(\gr B)$ and $w\in V(\gr C)$, and let $\mu^B(v) = \mu^C(w) = u \in V(\gr A)$. Then $D_x = \mu^B_v(B_v)^{\tilde x} \cap \mu^C_w(C_w) \le A_u$, where $\tilde x \in A_u$. Since $B_v$ and $C_w$ are finitely generated, the fact that $A_u$ has the strong \ffrip\ implies that $\sum\frank(D_x)$ is finite, where $x$ runs over the $(v,w)$-vertices of $\DD$. The result follows since $V(\gr B)$ and $V(\gr C)$ are finite.
\end{proof}

\begin{claim}\label{claim: edges with non-trivial edge group}
$\DD$ has only finitely many edges with non-trivial edge groups.
\end{claim}

\begin{proof}
Let $h$ be an edge in $\DD$, say, $h = \mu^B_f(B_f)\,\tilde h\,\mu^C_g(C_g)$ for some $f\in E(\gr B)$, $g\in E(\gr C)$ and $\tilde h \in A_e$, where $e = \mu^B(f) = \mu^C(g)$. The corresponding edge group $D_h$ is equal to $\mu^B_f(B_f)^{\tilde h}\cap \mu^C_g(C_g)$. If $D_h$ is non-trivial, then $\mu^B_f(B_f)$ and $\mu^C_g(C_g)$ are non-trivial subgroups of $A_e$. In view of our hypothesis on the edge groups of $\AA$, the quotients $\mu^B_f(B_f)\backslash A_e$ and $A_e/\mu^C_g(C_g)$ are finite, and hence $E_{f,g}(\gr D) = \dblcoset{B_f}{A_e}{C_g}$ is finite. Again, the result follows since $E(\gr B)$ and $E(\gr C)$ are finite.
\end{proof}

\def\TT{\mathbb{T}}

Let $B\,a\,C$ ($a\in A$) be a double coset. If $B^a\cap C = 1$, then $\frank(B^a\cap C) = 0$ by definition of $\frank$. 

Now suppose that $B^a\cap C \ne 1$. If $B^a\cap C$ is not of the form $\calC(x)$ for a vertex $x$ of $\DD$ (where~$\calC$ is the map introduced in \eqref{eq: C}), then $B^a\cap C$ is isomorphic to a subgroup of an edge group of $\AA$ by Theorem~\ref{thm: summary immersions}~\eqref{item: conjugate_into_edge_group}, and since we assumed these groups to be trivial or infinite cyclic, we have that $\frank(B^a\cap C) = 0$.

Finally, if $B^a\cap C = \calC(x)$ for a vertex $x$ of $\DD$, then $B^a\cap C$ is isomorphic to $\pi_1(\DD,x)$ by Theorem~\ref{thm: product of immersions}~\eqref{item: localy_elliptic_double_cosets}. Moreover Theorem~\ref{thm: summary immersions}~\eqref{item: same_double_coset} shows that different $(B,C)$-double cosets are $\calC$-images of vertices in different connected components of $\DD$. Hence, we need to show that $\sum_{x\in R} \frank(\pi_1(\DD, x))<\infty$, where $R$ is a set of vertices, one in each connected component of $\DD$.

Let $(\DD_j)_{j\in J}$ be the (maybe infinite) collection of the connected components of the graph of groups obtained from $\DD$ by removing every edge with associated trivial edge group. 
For each $j$, fix $x_j \in V(\gr D_j)$. For each $x\in R$, let $J_x$ be the set of indices $j$ such that $\DD_j$ is in the connected component of $\DD$ containing $x$. By Proposition~\ref{prop: splitting over trivial edge groups}~\eqref{item: split pi1 many trivial edges}, we have
\begin{align*}
\frank(\pi_1(\DD, x)) &= \sup_{j \in J_x} \frank(\pi_1(\DD_j, x_j))\quad\textrm{and hence} \\
\sum_{x\in R}\frank(\pi_1(\DD, x)) &= \sum_{x\in R} \left(\sup_{j\in J_x} \frank(\pi_1(\DD_j, x_j)\right) \le \sum_{x\in R,\,j \in J_x} \frank(\pi_1(\DD_j, x_j)). 
\end{align*}
    
If $\DD_j$ does not contain any edge, then $\pi_1(\DD_j,x_j) = D_{x_j}$. By Claim~\ref{claim: vertex groups have bounded factor rank}, it follows that
$$\sum_{x\in R,\,j \in J'_x}\frank(\pi_1(\DD_j, x_j)) < \infty,$$
where $J'_x$ is the set of indices $j\in J_x$ such that $\DD_j$ has no edge.

Consider now the other components $\DD_j$, which contain at least one edge. Claim~\ref{claim: edges with non-trivial edge group} shows that there are finitely many such components, say $\DD_1, \ldots, \DD_k$, and each $\gr D_i$ ($i\in [1,k]$) is finite.

Now fix $i\in [1,k]$. From \ref{claim: vertex groups have bounded factor rank}, we see that, if $x \in V(\gr D_i)$, then $D_x$ is isomorphic to the free product of a free group $F_x$ and a free product of freely indecomposable groups $D_{x,\beta}$, each of which has rank at most some constant $N$. If $h \in E(\gr D_i)$ and $o(h) = x$, then the edge map $\alpha_h$ sends $D_h$ injectively to $D_x$.
Since $D_h$ is infinite cyclic, its image is either in one of the $D_{x,\beta}$ or in $F_x$. In the latter case, $F_x$ splits as a free product of a finite rank free group containing $\alpha_h(D_h)$ and another free group. Since $E(\gr D_i)$ is finite, it follows that $D_x$ splits as $D'_x * D''_x$, where $D'_x$ is the free product of a finite rank free factor $F'_x$ of $F_x$ and of finitely many of the $D_{x,\beta}$, so that $D'_x$ contains $\alpha_h(D_h)$ for every edge $h$ such that $o(h) = x$, and where $D''_x$ is the free product of the free complement of $F'_x$ in $F_x$ and of the $D_{x,\beta}$ which meet trivially each $\alpha_h(D_h)$ originating at $x$.  In particular, $D'_x$ has finite rank and $D''_x$ has finite factor rank.

By Proposition~\ref{prop: split DDi}, it follows that $\pi_1(\DD_i,x_i)$ splits as the free product of the $D''_x$ ($x \in V(\gr D_i)$) and of the fundamental group of a finite graph of groups whose vertex groups are the $D'_x$ ($x \in V(\gr D_i)$). By Proposition~\ref{prop: finite generation}~\eqref{item: generators of pi1}, that fundamental group has finite rank, and hence finite factor rank. It follows that $\frank(\pi_1(\DD_i,x_i))$ is finite, and this concludes the proof.
\end{proof}

\begin{rem}
    The definition of $\frank$ could be generalised to be the maximal rank of a vertex group in a maximal splitting over finite groups (as opposed to the maximal rank of a vertex group in a maximal splitting over $1$). The definition of \ffrip\ could similarly be generalised and a version of Theorem \ref{thm: on ffrip} would still remain true, replacing trivial and infinite cyclic with finite and virtually infinite cyclic. However, one would need extra assumptions in Theorem \ref{thm: on ffrip} since finitely generated groups are not necessarily accessible over finite groups by a result of Dunwoody \cite{Du93}.
\end{rem}

\section{The \fcip\ and general criteria for the \fgip}\label{sec: fcip}

In this section we develop tools to determine whether a given (fundamental group of a) graph of groups has the \fgip\ or \sfgip\ Our tools are primarily based on a property we call the finite coset interaction property (\fcip), which we introduce in Section \ref{sec: fcip2} along with some variations. This is a property which allows us to deduce properties on the structure of arbitrary $\AA$-products. For example, we prove that given a graph of groups $\AA$ in which edge groups have the \fcip\ in their adjacent vertex groups, any $\AA$-product of finite graphs of finitely generated groups will have the property that each component has finite underlying graph. In Section \ref{sec: fgip fcip interplay} we use the properties of $\AA$-products derived using the \fcip, in conjunction with Theorem \ref{thm: bijection subgroups immersions covers}, to establish concrete criteria for the \fgip\ in graphs of groups. 

As noted in the introduction, our results proved here recover all known criteria for the \fgip\ for fundamental groups of graphs of groups. We will indicate precisely how known results can be recovered in the appropriate sections.

\subsection{The finite coset interaction property and variations}\label{sec: fcip2}

In this section we introduce several variants of a \emph{coset interaction property} which will allow us to constrain the structure of the underlying graphs of $\AA$-products. These constraints will then play a key role in developing criteria for a graph of groups to have the finitely generated intersection property (\fgip).

Let $A, B, C$ be finitely generated subgroups of a group $G$ and let $f, g\in G$. Let $q^A_{f,g}$ be the following map:
\begin{align*}
q^A_{f,g}\colon \dblcoset{(A\cap B^f)}A{(A\cap C^g)} & \longrightarrow \dblcoset BGC \\
(A\cap B^f)\,a\,(A\cap C^g) & \longmapsto B\,(f\,a\,g\inv)\,C.
\end{align*}
This map is well-defined. Indeed, if $a, a'\in A$ and $a'\in (A\cap B^f)\,a\,(A\cap C^g)$, then $a' = f\inv\,b\,(f\,a\,g\inv)\,c\,g$ for some $b\in B$ and $c\in C$. Therefore $f\,a'\,g\inv = b\,(f\,a\,g\inv)\,c \in B\,(f\,a\,g\inv)\,C$.

\begin{exm}
\label{exm: map}
    When the group $G$ above is abelian, the description of the map $q_{f, g}^A$ can be greatly simplified since double cosets become cosets and cosets are simply elements of a quotient group. For instance, let $G\cong \Z$, let $A = i\Z, B = j\Z, C = k\Z$ and let $f = m, g = n$. Then the map $q_{f, g}^A$ can be described explicitly:
    \begin{align*}
        q_{m, n}^{i\Z}\colon i\Z/\gcd(\lcm(i, j), \lcm(i, k))&\to \Z/\gcd(j, k)\\
                            it + \gcd(\lcm(i, j), \lcm(i, k))\Z &\mapsto it + (m-n) + \gcd(j, k)\Z.
    \end{align*}
    Note that if $f = g$, then $q_{f, g}^A$ is actually a group homomorphism.
    
    Letting $i = 2, j = 3, k = 6, m = 7, n = 8$, the map becomes:
    \begin{align*}
        q_{7, 8}^{2\Z}\colon 2\Z/6\Z &\to \Z/3\Z\\
                                2t + 6\Z &\mapsto 2t - 1 + 3\Z.
    \end{align*}
    In this case, $q_{7, 8}^{2\Z}$ is a bijection.
\end{exm}

\begin{defn}\label{defn: finite coset interaction2}
Let $B, C$ be finitely generated subgroups of a group $G$ and let $\mathcal{A}$ be a collection of finitely generated subgroups of $G$. We say that the pair $(B, C)$ has \emph{finite coset interaction relative to $\mathcal{A}$} if, for each $A\in \mathcal{A}$ and any finite collections $\mathcal{F}_A$ and $\mathcal{G}_A$ of elements of $G$, lying in pairwise distinct $(B,A)$- and $(C,A)$-double cosets, respectively, the following holds:
\[
\sum_{B\,h\,C\ \in\ \dblcoset BGC}\max{\left\{0, \left(\sum_{A\in \mathcal{A}}\sum_{\substack{f\in \mathcal{F}_A\\ g\in \mathcal{G}_A}} \left|(q^A_{f, g})^{-1}(B\,h\,C)\right|\right) - 1\right\}}<\infty.
\]
In other words, there are finitely many $(B,C)$-double cosets which occur at least twice in the union of the ranges of the $q^A_{f,g}$ ($f\in \mathcal{F}_A$, $g\in \mathcal{G}_A$) --- either because they occur in the ranges of $q^A_{f,g}$ and $q^{A'}_{f',g'}$ for distinct triples $A, f,g$ and $A', f',g'$, or because they have at least two pre-images under some $q^A_{f,g}$.
\end{defn}

\begin{rem}\label{rem: translate fcip2}
The motivation for Definition~\ref{defn: finite coset interaction2} is the following. Let $\mu^B\colon \BB\to\AA$ and $\mu^C\colon\CC\to\AA$ be immersions of graphs of groups. Let $f\in E(\gr B)$ and $g \in E(\gr C)$ be edges with the same image $e\in E(\gr A)$, and let $v = o(f)$, $w = o(g)$ and suppose that $u\in V(\gr A)$ the common image of $v$ and $w$. Note that, by definition of immersions (see Definition~\ref{def: folded}), the double coset $\mu^B_v(B_v)\,f_\alpha\,\alpha_e(A_e)$ uniquely determines $f$ and the double coset $\mu^C_w(C_w)\,g_\alpha\,\alpha_e(A_e)$ uniquely determines $g$.
\smallskip
Finally, let $\DD = \BB\wtimes_\AA\CC$ and take $h \in E_{f,g}(\gr D)$, say, $h = \mu^B_f(B_f)\,a\,\mu^C_g(C_g)$ with $a\in A_e$. Then we have (recall from Section~\ref{sec: pullbacks})
$$o(h) = \mu^B_v(B_v)\,(f_\alpha\,\alpha_e(a)\,g_\alpha\inv)\,\mu^C_w(C_w).$$
Moreover, $\alpha_e(A_e) \cap \mu^B_v(B_v)^{f_\alpha} = \alpha_e(\mu^B_f(B_f))$ and $\alpha_e(A_e) \cap \mu^C_w(C_w)^{g_\alpha} = \alpha_e(\mu^C_g(C_g))$ since $\mu^B$ and $\mu^C$ are immersions. Therefore,
$$o(h) = q^{A_e}_{f_\alpha,g_\alpha}\left(\alpha_e\left(\mu^B_f(B_f)\,a\,\mu^C_g(C_g) \right)\right) = q_{f_\alpha,g_\alpha}(\alpha_e(h)).$$
In particular, if $G = A_u$, $B = \mu^B_v(B_v)$, $C = \mu^C_w(C_w)$, and if $(B,C)$ has finite coset interaction relative to $\{\alpha_e(A_e)\}$, letting $\mathcal F$ (resp. $\mathcal G$) be the set of all $f_\alpha$ (resp. $g_\alpha$), where $f \in E(\gr B)$, $\mu^B(f) = e$ and $o(f) = v$ (resp. $g\in E(\gr C)$,  $\mu^C(g) = e$ and $o(g) = w$), then every $(v,w)$-vertex is the origin of a finite number of edges which map to $e$, and only finitely many $(v,w)$-vertices are the origin of two or more such edges.
\end{rem}

\begin{exm}
\label{exm: abelian}
    Let $G$ be a finitely generated abelian group, let $B, C\leqslant G$ and let $\mathcal{A} = \{A\}$ be a single subgroup of $G$. Consider the homomorphism
    \[
    q_{0, 0}^A\colon A/((A\cap B) + (A\cap C))\to G/(B + C).
    \]
    If any one of the following holds, then $(B, C)$ does not have finite coset interaction relative to $A$:
    \begin{enumerate}[(1)]
        \item the kernel of $q_{0, 0}^A$ is infinite;
        \item the image of $q_{0, 0}^A$ is infinite and its kernel is non-trivial;
        \item the image of $q_{0, 0}^A$ is infinite and $B+A\neq G$, $C+A\neq G$. Indeed, there is $f\in G - ((B+A)\cup (C+A))$ (a group cannot be equal to the union of two proper subgroups) and so we see that each element in the image of $q_{0, 0}^A$ also has a pre-image in $q_{f, f}^A = q_{0, 0}^A$.
        \item the image of $q_{0, 0}^A$ is infinite and $B+A = G$ but $C+A\neq G$. Indeed, since $B+A = G = B+A - g + C$, each map $q_{0, g}^A$ is surjective.
    \end{enumerate}
    Thus $(B, C)$ has finite coset interaction with respect to $A$ if and only if either the image and kernel of $q_{0, 0}^A$ are finite (an explicit example of this situation can be found in Example \ref{exm: map}), or $B+A = C+A = G$ and the kernel of $q_{0, 0}^A$ is trivial.
\end{exm}

We will also use the following variants and extensions of Definition~\ref{defn: finite coset interaction2}.

\begin{defn}\label{defn: variants of fcip2}
Let $B, C$ be finitely generated subgroups of $G$, let $\mathcal{A}$ be a collection of finitely generated subgroups of $G$ and let $k \ge 0$ be an integer. We say that
\begin{itemize}
\item  the pair $(B, C)$ has \emph{$k$-finite coset interaction relative to $\mathcal{A}$} if, for each $A\in \mathcal{A}$ and any finite collection of elements $\mathcal{F}_A, \mathcal{G}_A\subset G$ lying in pairwise distinct $(B, A)$- and $(C, A)$-double cosets, respectively, the following holds:
\[
\sum_{B\,h\,C\ \in\ \dblcoset BGC}\max\left\{0, \left(\sum_{A\in \mathcal{A}}\sum_{\substack{f\in \mathcal{F}_A,\\ g\in \mathcal{G}_A}}\left|(q^A_{f, g})^{-1}(B\,h\,C)\right|\right) - k\right\}<\infty.
\]
(For $k=1$, this coincides with finite coset interaction as in Definition~\ref{defn: finite coset interaction2}.)

\item The pair $(B, C)$ has \emph{weak finite coset interaction relative to $\mathcal{A}$} if for any collection of elements $\{f_A, g_A\}_{A\in \mathcal{A}}$ of $G$, the following holds:
\[
\left|\bigcap_{A\in \mathcal{A}}\Ima(q^{A}_{f_A, g_A})\right|<\infty.
\]

\item The pair $(B, C)$ has \emph{local finite coset interaction relative to $\mathcal{A}$} if, for each $A\in \mathcal{A}$ and any $f, g, h\in G$ we have:
\[
\left|(q^A)_{f, g}^{-1}(B\,h\,C)\right|<\infty.
\]

\item
$G$ has the \emph{$k$-finite coset interaction property}, or \emph{$k$-\fcip}, \emph{relative to $\mathcal{A}$} if every pair of finitely generated subgroups of $G$ has $k$-finite coset interaction relative to $\mathcal{A}$; if $k=1$, we simply say that \emph{$G$ has the \fcip\ relative to $\mathcal{A}$}.

\item $G$ has the \emph{weak \fcip\ relative to $\mathcal{A}$} if every pair of finitely generated subgroups of $G$ have weak finite coset interaction relative to $\mathcal{A}$.

\item $G$ has the \emph{local \fcip\ relative to $\mathcal{A}$} if every pair of finitely generated subgroups of $G$ have local finite coset interaction relative to $\mathcal{A}$.
\end{itemize}
\end{defn}

\begin{rem}
It follows from Definition~\ref{defn: variants of fcip2} that, if $G$ has the $k$-\fcip\ relative to $A$, then it also has the $(k+1)$-\fcip\ and the local \fcip\ relative to $A$.
\end{rem}

In the spirit of Remark~\ref{rem: translate fcip2}, these definitions translate to the following statement.

\begin{prop}\label{prop: translate fcip2}
Let $\AA, \BB, \CC$ be graphs of groups with finite underlying graphs and finitely generated vertex and edge groups, equipped with immersions from $\BB$ and $\CC$ to  $\AA$. Let also $u \in V(\gr A)$, $E_u\subset \Star(u)$ and let $k\ge 0$.
\begin{itemize}

\item If $A_u$ has the $k$-\fcip\ relative to $\{\alpha_e(A_e)\}_{e\in E_u}$, then the $\AA$-product $\BB\wtimes_{\AA}\CC$ has finitely many vertices with strictly more than $k$ outgoing edges mapping edges
in $E_u$, and no vertex is the origin of infinitely many outgoing edges mapping into $E_u$.

\item If $A_u$ has the weak \fcip\ relative to $\{\alpha_e(A_e)\}_{e\in E_u}$, then the $\AA$-product $\BB\wtimes_{\AA}\CC$ has finitely many vertices $x$ such that the image of $\Star(x)$ contains $E_u$.

\item If $A_u$ has the local \fcip\ relative to $\{\alpha_e(A_e)\}_{e\in E_u}$, then every vertex of the $\AA$-product $\BB\wtimes_{\AA}\CC$ has finitely many outgoing edges mapping into $E_u$.
\end{itemize}
\end{prop}

\begin{rem}\label{rem: translate fcip is an equivalence}
The implications in Proposition~\ref{prop: translate fcip2} are in fact logical equivalences in the following sense. With the notation of that proposition, if $A_u$ does not have the $k$-\fcip\ relative to $\alpha_e(A_e)$, then one can find graphs of groups $\BB$ and $\CC$ with finite underlying graphs, and immersions from $\BB$ and $\CC$ to $\AA$, such that $\BB\wtimes_{\AA}\CC$ has infinitely many vertices with at least $k+1$ outgoing edges mapping to $e$, or a vertex with infinitely many outgoing edges mapping to $e$.

A similar statement holds regarding the weak \fcip\ and the local \fcip\
\end{rem}

Proposition~\ref{prop: translate fcip2} immediately implies the following statements. Recall that a graph is \emph{locally finite} if each vertex is the extremity of a finite number of edges.

\begin{cor}\label{cor: k-fcip}
Let $k\geqslant 0$ and let $\AA$ be a finite graph of groups such that each vertex group has the $k$-\fcip\ relative to the collection of adjacent edge groups. Then for any pair of immersions $\mu^B\colon \BB\to \AA$, $\mu^C\colon \CC\to \AA$ of finite graphs of finitely generated groups, the $\AA$-product $\BB\wtimes_{\AA}\CC$ is locally finite and has finitely many vertices with strictly more than $k$ outgoing edges. In particular:
\begin{enumerate}[(1)]
    \item If $k = 0$ then $\BB\wtimes_{\AA}\CC$ has finitely many edges.
    \item If $k = 1$, then each component of $\BB\wtimes_{\AA}\CC$ has finitely many edges.
\end{enumerate} 
\end{cor}

\begin{cor}
Let $\AA$ be a finite graph of groups such that each vertex group has the weak \fcip\ relative to the collection of adjacent edge groups. Then for any pair of immersions $\mu^B\colon \BB\to \AA$, $\mu^C\colon \CC\to \AA$ of finite graphs of finitely generated groups, the $\AA$-product $\BB\wtimes_{\AA}\CC$ has finitely many vertices at which the map $\BB\wtimes_{\AA}\CC\to\AA$ is not locally injective.
\end{cor}

\begin{cor}\label{locally_finite_criterion}
Let $\AA$ be a locally finite graph of groups such that each vertex group has the local \fcip\ relative to each adjacent edge group. Then for any pair of immersions $\mu^B\colon \BB\to \AA$, $\mu^C\colon \CC\to \AA$ of locally finite graphs of finitely generated groups, the $\AA$-product $\BB\wtimes_{\AA}\CC$ is locally finite.
\end{cor}

There are three more situations in which we may use the \fcip\ properties to constrain the structure of $\AA$-products.

\begin{prop}
\label{prop: 2-fcip}
    Let $\AA$ be a finite graph of groups such that each vertex group has the $2$-\fcip\ relative to the collection of adjacent edge groups. Then for any pair of immersions $\mu^B\colon \BB\to \AA$, $\mu^C\colon \CC\to \AA$ of finite graphs of finitely generated groups, there is a decomposition 
    \[
    \BB\wtimes_{\AA}\CC = \DD_1\cup\DD_2\cup \DD_3
    \]
    so that the following holds:
    \begin{enumerate}[(1)]
        \item $\gr{D}_1$ is a finite graph.
        \item $\gr{D}_2$ is a disjoint union of finitely many lines that each intersect $\gr{D}_1$ only at a single vertex of degree one in $\gr{D}_2$.
        \item $\gr{D}_3$ is a disjoint union of lines and cycles that do not intersect $\gr{D}_1\cup\gr{D}_2$.
    \end{enumerate}
\end{prop}

\begin{proof}
    Letting $\DD = \BB\wtimes_{\AA}\CC$, by Proposition \ref{prop: translate fcip2}, all but finitely many vertices $x\in V(\gr{D})$ have $\deg(x)\leqslant 2$. Let $\gr{D}_1\subseteq\gr{D}$ be the subgraph consisting of all the stars of all vertices of $\gr{D}$ of degree at least 3 and the union of all the images of reduced paths connecting pairs of such vertices. Since there are finitely many such vertices and each vertex has finite degree, $\gr{D}_1$ is a finite subgraph. Let $\gr{D}_2$ be the subgraph on the edges not in $\gr{D}_1$, but that lie in the same component in $\gr{D}$ as a component of $\gr{D}_1$. By definition of $\gr{D}_1$, every vertex of $\gr{D}_2$ must have degree at most 2, except possibly at the intersection points with $\gr{D}_1$, which must have degree one (in $\gr{D}_2$). A component of $\gr{D}_2$ cannot have two intersection points with $\gr{D}_1$ as otherwise the path in this component connecting the two intersection points would lie in $\gr{D}_1$ by definition. Finally, all other components of $\gr{D}$, which form $\gr{D}_3$, consists only of vertices of degree at most 2. Hence $\gr{D}_3$ is a disjoint union of lines and cycles as required.
\end{proof}

\begin{prop}
\label{prop: fcip on orientation}
    Let $\AA$ be a finite graph of groups and let $E^+(\gr{A})\subset E(\gr{A})$ be an orientation such that, for each vertex $u\in V(\gr A)$, $A_u$ has the \fcip\ relative to the collection $\{\alpha_e(A_e) \mid e\in E^+(\gr A),\ o(e) = u\}$. Then for any pair of immersions $\mu^B\colon \BB\to \AA$, $\mu^C\colon \CC\to \AA$ of finite graphs of finitely generated groups, there is a decomposition
    \[
    \BB\wtimes_{\AA}\CC =\DD_1\cup\DD_2
    \]
    such that the following holds:
    \begin{enumerate}[(1)]
        \item $\gr{D}_1$ is a finite graph containing the core of each of its components.
        \item $\gr{D}_2$ is a finite union of finitely many graphs that intersect $\gr{D}_1$ in at most one vertex.
        \item Each vertex in $\gr{D}_2$ has at most one outgoing edge mapping to an edge in $E^+(\gr{A})$.
    \end{enumerate}
\end{prop}

\begin{proof}
    Letting $\DD = \BB\wtimes_{\AA}\CC$, if $x\in V(\gr{D})$, denote by $\deg_+(x)$ the number of edges originating at $x$, mapping to an edge in $E^+(\gr{A})$. Let $\gr{D}_1\subseteq\gr{D}$ be the subgraph consisting of the stars of all vertices $x$ of $\gr{D}$ with $\deg_+(x)\geqslant 2$ and the union of all the images of reduced paths connecting pairs of such vertices (and that only traverse vertices of $\deg_+$ at most one). Denoting by $V\subset V(\gr{D})$ the vertices with outdegree at least 2, if $\gr{D}_1'$ is any connected finite subgraph of $\gr D_1$, we have
        \[
        \chi(\gr{D}_1') = \sum_{x\in V(\gr{D}_1')}1 - \deg_+(x)\geqslant \sum_{x\in V}1 - \deg_+(x).
        \]
    This implies that the core of each component of $\gr{D}_1$ is a finite graph and hence that each component of $\gr{D}_1$ itself is a finite graph. Since there are finitely many vertices $x$ with $\deg_+(x)\geqslant 2$, $\gr{D}_1$ has finitely many components and so $\gr{D}_1$ is a finite graph. Let $\gr{D}_2$ be the subgraph on the edges that do not lie in $\gr{D}_1$. By definition of $\gr{D}_1$, each vertex of $\gr{D}_2$ must have outdegree at most 1 and each component of $\gr{D}_2$ must intersect $\gr{D}_1$ in at most a single vertex (this is because every pair of points in a component of $\gr{D}_2$ can be connected by a reduced path which only traverses vertices of outdegree 1). The fact that each component of $\gr{D}_1$ contains the core of the corresponding component of $\gr{D}$ is also clear from the definition.
\end{proof}

\begin{prop}
\label{prop: combination_components}
    Let $\AA$ be a finite graph of groups, let $\{\AA_{\alpha}\}_{\alpha}$ be a collection of subgraphs of groups so that $\gr{A}_{\alpha}\cap\gr{A}_{\beta}\subset V(\gr{A})$ for each $\alpha\neq\beta$ and so that $\gr{A} = \bigcup_{\alpha}\gr{A}_{\alpha}$. Suppose also that for each vertex $u\in V(\gr{A})$ and each pair of edges $e, f\in \Star(u)$ that do not both lie in the same $\gr{A}_{\alpha}$, we have that $A_u$ has the weak \fcip\ relative to $\{\alpha_e(A_e), \alpha_f(A_f)\}$. Then for any pair of immersions $\mu^B\colon \BB\to \AA$, $\mu^C\colon \CC\to \AA$ of finite graphs of finitely generated groups, there is a decomposition
    \[
    \BB\wtimes_{\AA}\CC =\bigcup_{\alpha}\DD_{\alpha}
    \]
    such that the following holds:
    \begin{enumerate}[(1)]
        \item $\DD_{\alpha}$ is a union of components of $\BB_{\alpha}\wtimes_{\AA_{\alpha}}\CC_{\alpha}$, where $\BB_{\alpha} = (\mu^B)^{-1}(\AA_{\alpha})$ and $\CC_{\alpha} = (\mu^C)^{-1}(\AA_{\alpha})$.
        \item $\gr{D}_{\alpha}\cap \gr{D}_{\beta}$ consists of finitely many vertices for each $\alpha\neq\beta$.
    \end{enumerate}
\end{prop}

\begin{proof}
    Firstly note that by definition of $\BB\wtimes_{\AA}\CC$ we have that 
    \[
    \BB\wtimes_{\AA}\CC = \bigcup_{\alpha}\BB_{\alpha}\wtimes_{\AA_{\alpha}}\CC_{\alpha}.
    \]
    For each $\alpha$, let $\DD_{\alpha}$ be the union of all components of $\BB_{\alpha}\wtimes_{A_{\alpha}}\CC_{\alpha}$ that contain at least one edge. Then the union of $\DD_{\alpha}$ is all of $\DD$ except for possibly some isolated vertices. By making some choices, we may add each remaining isolated vertex to some $\DD_{\alpha}$ so that $\BB\wtimes_{\AA}\CC = \bigcup_{\alpha}\DD_{\alpha}$. 
    
    By Proposition \ref{prop: translate fcip2} there are finitely many vertices in $\DD = \BB\wtimes_{\AA}\CC$ that have outgoing edges $e, f$ that do not both map into some $\AA_{\alpha}$. Thus, $\gr{D}_{\alpha}\cap \gr{D}_{\beta}$ consists of finitely many vertices for each $\alpha\neq\beta$ as required.
\end{proof}

\subsubsection*{Finite index subgroups and the $k$-\fcip\ property}

Proposition~\ref{index_k} below motivates the definition of the $k$-\fcip\ when $k > 1$. It allows us to provide more examples of groups with the $k$-finite coset interaction relative to a collection of subgroups and will be used in Section~\ref{sec: fcip for qcv subgroups}. Its proof requires the following group-theoretic lemma.

\begin{lem}
\label{double_coset_lemma}
    Let $G', H, K$ be subgroups of a group $G$. If $G'$ has index $k$ in $G$, then the map
    \[
    (H\cap G')\backslash G'/(K\cap G') \to H\backslash G/K
    \]
    is at most $k$ to one.
\end{lem}

\begin{proof}
The subgroup $H \cap G'$ has index at most $k$ in $H$, and the same holds for $K\cap G'$ in $K$. Let $T_H \subset H$ be a right transversal for $H \cap G'$ in $H$, and $T_K \subset K$ a left transversal for $K\cap G'$ in $K$, so that $|T_H|, |T_K| \le k$. We may assume that $1\in T_H, T_K$.

Suppose that $g_1,\dots, g_n \in G'$ sit in pairwise distinct $(H\cap G', K\cap G')$-double cosets, while they sit in the same $(H,K)$-double coset.

Then, for each $i\in [2,n]$, we have $g_i \in H\,g_1\,K$, so that $g_i = h_i s_i g_1 t_i k_i$ for some $h_i\in H\cap G'$, $s_i \in T_H$, $t_i \in T_K$ and $k_i \in K\cap G'$. This is extended to $i = 1$ by letting $h_1 = s_1 = t_1 = k_1 = 1$.
In particular, if $i, j \in [1,n]$, then
$$g_j = h_j\,s_j\,s_i\inv\,h_i\inv\,g_i\,k_i\inv\,t_i\inv\,t_j\,k_j.$$
Suppose that, for some $i\ne j$, we have $s_i = s_j$. Then $g_j = h_j\,h_i\inv\,g_i\,k_i\inv\,(t_i\inv\,t_j)\,k_j$. It follows that $t_i\inv\,t_j \in G'$, and hence $t_i\inv\,t_j \in K\cap G'$. By definition of $T_K$, it follows that $t_i  = t_j$, and this implies that $g_j = h_j\,h_i\inv\,g_i\,k_i\inv\,k_j \in (H\cap G')\,g_i\,(K\cap G')$, a contradiction.

Therefore the $s_i$ are pairwise distinct, and hence $n \le k$.
\end{proof}

\begin{prop}
\label{index_k}
    Let $G$ be a group and let $B, C$ be finitely generated subgroups of $G$ such that $(B, C)$ has finite coset interaction with respect to a collection of finitely generated subgroups $A_1, \ldots, A_n\leqslant G$. For each $i$, let $A_i'$ be a subgroup of $A_i$ of finite index $k_i$, and let $k = \max{\{k_i\}}$. Then $(B, C)$ has $k$-finite coset interaction with respect to $A_1', \ldots, A_n'$.
\end{prop}

\begin{proof}
Let $\mathcal F_i$ (resp. $\mathcal G_i$) be a finite set of elements lying in pairwise distinct $(B,A_i')$-double cosets (resp. $(C,A_i')$-double cosets). Since $(B,C)$ has finite coset interaction relative to $A_1, \ldots, A_n$, there exists finite subsets $X_i$ of $A_i$ with the following property: if $a \in A_i - X_i$ and $a'\in A_j$ satisfy $B\,(f\,a\,g\inv)\,C = B\,(f'\,a'\,g'^{-1})\,C$, then $i = j$, $f = f'$, $g = g'$ and $(A_i\cap B^f)\,a\,(A_i\cap C^g) = (A_i\cap B^f)\,a'\,(A_i\cap C^g)$.
    
Lemma~\ref{double_coset_lemma} (applied to $G = A_i$, $G' = A_i'$, $H = A_i\cap B^f$ and $K = A_i \cap C^g$) shows that, for each $a\in A_i'$, the double coset $(A_i\cap B^f)\,a\,(A_i\cap C^g)$ contains at most $k$ double cosets of the form $(A_i'\cap B^f)\,a'\,(A_i'\cap C^g)$ with $a'\in A_i'$.

It follows that, if $a\in A_i' - (A_i' \cap X_i)$, there are at most $k$ double cosets of the form $(A_i'\cap B^f)\,a'\,(A_i'\cap C^g)$ (with $a'\in A_i'$) such that $B\,(fa'g\inv)\,C = B\,(fag\inv)\,C$, and this concludes the proof.
\end{proof}

\begin{exm}
    It is clear that a group $G$ has the \fcip\ with respect to itself. Proposition \ref{index_k} implies that $G$ has the $k$-\fcip\ with respect to any index $k$ subgroup of $G$. More explicitly, $\Z$ has the $k$-\fcip\ with respect to $\{k\Z\}$. One should contrast this with Example \ref{exm: abelian}.
\end{exm}

\subsection{General criteria for the \fgip\ and the \sfgip}\label{sec: fgip fcip interplay}

In this section we shall establish general criteria for the \fgip\  of graphs of groups, based on results on $\AA$-products from \cite{dllrw_pullback} and various \fcip\  from Section \ref{sec: fcip2}. We shall repeatedly use the proposition below, which follows from the description of the vertex and edge groups of the $\AA$-product provided in Section \ref{sec: pullbacks}.

\begin{prop}
    \label{prop: fgip_summary}
    Let $\mu^B\colon \BB\to \AA$ and $\mu^C\colon \CC\to \AA$ be immersions of graphs of groups and denote by $\DD = \BB\wtimes_{\AA}\CC$ the $\AA$-product. Suppose that each vertex group (respectively, edge group) of $\AA, \BB, \CC$ is finitely generated. For each vertex $u\in V(\gr{A})$ (respectively, edge $e\in E(\gr{A})$), let $\mathcal{G}_u$ (respectively, $\mathcal{G}_e$) be a collection of subgroups of $A_u$ (respectively, $A_e$). Then:
    \begin{itemize}
        \item If each vertex group (respectively, edge group) of $\AA$ has the \fgip, then each vertex group (respectively, edge group) of $\DD$ is finitely generated.
        \item If each vertex group (respectively, edge group) of $\AA$ has the \sfgip\ and $\gr{A}, \gr{B}, \gr{C}$ are finite, then $\DD$ has finitely many non-trivial vertex groups (respectively, edge groups).
        \item If for each $u\in V(\gr{A})$, the pair $(A_u, \mathcal{G}_u)$ (respectively, for each $e\in E(\gr{A})$, the pair $(A_e, \mathcal{G}_e)$) has the \sfgip\ and $\gr{A}, \gr{B}, \gr{C}$ are finite, then $\DD$ has finitely many vertex groups not contained in $\bigcup_{u\in V(\gr{A})}\mathcal{G}_u$ (respectively, edge groups not contained in $\bigcup_{e\in E(\gr{A})}\mathcal{G}_e$).
        \item If $e\in E(\gr{A})$ is such that $(A_{o(e)}, \alpha_{e}(A_e))$ has the $\alpha_e(A_e)$-\fgip, then all but finitely many edges in $\DD$ that map to $e$ are non-reduced.
    \end{itemize}
\end{prop}

\begin{rem}\label{rem: wlog A is finite}
    Within this section we shall always assume that the graph
    $\gr A$ is finite. However, most of our results hold without this assumption for the following reason. If $(\BB, v)\to (\AA, u)$ and $(\CC, w)\to (\AA, u)$ are morphisms of graphs of groups with $\gr B$ and $\gr C$ finite, then there is a finite subgraph $\gr A' \subset \gr A$
    which supports the images of $\gr B$ and $\gr C$
    , and we let $\AA'$ be the subgraph of groups induced by $\gr A'$.
    Since all our results involve intersections of the images of $\pi_1(\BB, v), \pi_1(\CC, w)$ which lie in $\pi_1(\AA', u)\leqslant \pi_1(\AA, u)$, it will usually suffice to only consider the finite subgraph of groups $\AA'\subset \AA$.
    We will also always assume that the vertex groups of $\AA$ are finitely generated but, for the same reason, our results hold also without this hypothesis.
\end{rem}

A few particularly important collections of subgroups of a graph of groups will arise as in Proposition \ref{prop: fgip_summary}. One is the collection of \emph{finite subgroups} which we shall always denote by $\mathcal{F}$. When we want to specify the finite subgroups from a given vertex group $A_u$ (or edge group $A_e)$, then we shall write $\mathcal{F}_u$ (or $\mathcal{F}_e$). Other important collections of subgroups are $\mathcal{V}$ and $\mathcal{E}$, defined as follows when $\AA$ has a basepoint $u_0$:  $\mathcal{V}$ is the collection of all subgroups of the form $p^{-1}A_u\,p$ as $u$ ranges over vertices in $\gr{A}$ and $p$ ranges over $\AA$-paths connecting $u$ with the basepoint; and $\mathcal{E}$ is the collection of all edge subgroups of $\AA$.

\subsubsection{The $0$-\fcip\ and graphs of groups with finite edge groups}\label{sec: 0-fcip and finite edge groups}

Cohen proved in \cite[Theorem 7]{coh74} that a group which splits as a graph of groups where each vertex group has the \fgip\ and each edge group is finite, also has the \fgip\ We reprove this statement using our \fcip, and extend the result to the strong \fgip\ We first need a lemma.

\begin{lem}
\label{0-fcip_lemma}
If $G$ is a group and $A_1, \ldots, A_n$ is a collection of finitely generated subgroups of $G$, then $G$ has the $0$-\fcip\ relative to $A_1, \ldots, A_n$ if and only if each $A_i$ is finite.
\end{lem}

\begin{proof}
If each $A_i$ is finite, then for any choices of $f, g\in G$ and $B, C\leqslant G$, the map $q^{A_i}_{f, g}$ has finite domain. This implies that the sum in the definition of $0$-finite coset interaction is always finite and so $G$ has the $0$-fcip relative to the collection $A_1, \ldots, A_n$.

Conversely, suppose that $G$ has the $0$-fcip relative to the collection $A_1, \ldots, A_n$. Letting $B = C = 1$ and $\mathcal{F} = \mathcal{G} = \{1\}$, we see that the sum in the definition of $0$-finite coset interaction is at least the cardinality of $A_i$. Thus, the $A_i$ must all be finite.
\end{proof}

\begin{thm}\label{thm: fgip when 0-fcip}
Let $(\AA,u_0)$ be a pointed graph of groups with finite underlying graph, in which every vertex group has the \fgip\  and every edge group is finite. Then
\begin{enumerate}[(1)]
\item\label{item: 0-fcip fgip} $\pi_1(\AA, u_0)$ has the \fgip
\item\label{item: 0-fcip relative sfgip} $\left(\pi_1(\AA, u_0), \mathcal{V}\right)$ has the \sfgip\
\item\label{item: 0-fcip sfgip} If each vertex group has the \sfgip, then $(\pi_1(\AA, u_0), \mathcal{F})$ has the \sfgip
\end{enumerate}
\end{thm}

\begin{proof}
Let $B$ and $C$ be finitely generated subgroups of $\pi_1(\AA,u_0)$. By Theorem~\ref{thm: bijection subgroups immersions covers} there are immersions of pointed graphs of groups $\mu^B\colon (\BB,v_0) \to (\AA,u_0)$ and $\mu^C\colon (\CC,w_0) \to (\AA,u_0)$, where $\BB$ and $\CC$ have finite underlying graphs, finite edge groups, finitely generated vertex groups and $\mu^B_*(\pi_1(\BB, v_0)) = B$, $\mu^C_*(\pi_1(\CC, w_0)) = C$.

We noted in Lemma~\ref{0-fcip_lemma} that a group always has the 0-\fcip\ relative to any finite collection of finite subgroups, hence each vertex group has the $0$-\fcip\ relative to the images of all of its adjacent edge groups. It follows, by Corollary \ref{cor: k-fcip}, that the $\AA$-product $\DD' = \BB\wtimes_\AA\CC$ has finitely many edges. In particular, the connected components of the underlying graph $\gr{D}'$ are all finite, and finitely many of them contain at least one edge. The other connected components of $\gr{D}'$, possibly infinitely many of them, consist of a single vertex and no edge. In particular, $\DD = \core(\DD')$ has finite underlying graph. By Proposition \ref{prop: fgip_summary}, every vertex and edge group of $\DD'$ is finitely generated and so $\pi_1(\DD, x_0)$ is finitely generated. Hence, by Theorem~\ref{thm: product of immersions}~\eqref{item: intersection of subgroups} $B\cap C$ is finitely generated and so $\pi_1(\AA, u_0)$ has the \fgip, establishing Statement~\eqref{item: 0-fcip fgip}.

By Theorem~\ref{thm: summary immersions}~\eqref{item: components_bijection}, only finitely many double cosets $B\,g\,C$ are such that $B^g\cap C$ contains at least one non-elliptic element. For all other double cosets $B\,g\,C$, the subgroup $B^g\cap C$ contains only elliptic elements. Hence, since $B^g\cap C$ is finitely generated, if it only contains elliptic elements then it is an elliptic subgroup by \cite[Corollary 7.3]{bas93}. In other words, it conjugates into a vertex group or, equivalently, it is contained in a group in $\mathcal V$. This implies that $\left(\pi_1(\AA), \mathcal{V}\right)$ has the \sfgip\ as claimed in Statement~\eqref{item: 0-fcip relative sfgip}.

Suppose now that the vertex groups of $\AA$ have the \sfgip\ Then by Proposition \ref{prop: fgip_summary} all but finitely many vertex groups of $\DD'$ are trivial. Thus the number of double cosets $B\, g\, C$ such that $B^g \cap C\neq 1$, where $B\,g\,C$ runs over the $\calC$-images of vertices of $\DD'$ (see Equation~\eqref{eq: C} in Section~\ref{sec: pullbacks} for the definition of the map $\C$) is finite by Theorem~\ref{thm: summary immersions}~\eqref{item: components_bijection}.  Theorem~\ref{thm: summary immersions}~\eqref{item: conjugate_into_edge_group} shows that all other double cosets $B\,g\,C$ are such that $B^g\cap C$ conjugates into an edge group of $\AA$ (that is, a subgroup in $\mathcal{E}$), and hence is finite. Thus $(\pi_1(\AA, u_0), \mathcal{F})$ has the \sfgip\ as claimed in Statement~\eqref{item: 0-fcip sfgip}.
\end{proof}

Theorem \ref{thm: fgip when 0-fcip}~\eqref{item: 0-fcip sfgip} cannot be upgraded to show that the vertex groups having the \sfgip\ suffices to show that $\pi_1(\AA,u_0)$ has the \sfgip, as Example~\ref{ex: no sfgip} shows. However, Theorem~\ref{thm: charact sfgip finite edge groups} below characterises the situation where this is the case.

\begin{exm}\label{ex: no sfgip}
The abelian group $\Z\times\Z/k\Z = \langle a, b\mid [a, b] = 1, b^k\rangle$ splits as an HNN extension with finite vertex group $\langle b\rangle$ and with edge group mapping isomorphically to $\langle b\rangle$. In particular, the vertex group satisfies the \sfgip\ However, we already saw in Example \ref{ex: sfgip example} that $\Z\times\Z/k\Z$ does not have the \sfgip.
\end{exm}

\begin{thm}\label{thm: charact sfgip finite edge groups}
    Let $\AA$ be a finite graph of groups with finite edge groups. The following are equivalent:
    \begin{enumerate}[(1)]
        \item\label{item: just sfgip} $\pi_1(\AA, u_0)$ has the \sfgip
        \item\label{item: no obstruction} Each vertex group of $\AA$ has the \sfgip\ and $\pi_1(\AA, u_0)$ does not contain any subgroup isomorphic to $\Z/k\Z\times\Z$ for some $k\geqslant 2$.
        \item\label{item: sfgip and acylindricity} Each vertex group of $\AA$ has the \sfgip\ and $\AA$ is acylindrical.
    \end{enumerate}
\end{thm}

\begin{proof}
    It is immediate that if a group has the \sfgip, so do its subgroups. In view of Example~\ref{ex: no sfgip}, this establishes that \eqref{item: just sfgip} implies \eqref{item: no obstruction}.

    Suppose now that $\AA$ is acylindrical. Let $B$, $C$, $\BB$, $\CC$, $\mu^B$ and $\mu^C$ be as in the proof of Theorem~\ref{thm: fgip when 0-fcip}. In view of the definition of acylindricity (see the end of Section~\ref{sec: pullbacks}) and since $\mu^B$ is an immersion (and therefore preserves reducedness), $\BB$ is acylindrical as well. Theorem~\ref{thm: acylindrical} shows that the map $\calC$ is surjective onto the set of double cosets $B\,g\,C$ such that $B^g\cap C \ne 1$. The reasoning applied to establish Theorem~\ref{thm: fgip when 0-fcip}~\eqref{item: 0-fcip sfgip} now does not need to involve the finite subgroups of $A$ and it follows that $\pi_1(\AA,u_0)$ has the \sfgip\ Thus \eqref{item: sfgip and acylindricity} implies \eqref{item: just sfgip}.
    
    Before we prove that \eqref{item: no obstruction} implies \eqref{item: sfgip and acylindricity}, we observe the following. If $p = (a',e,a)$ is an $\AA$-path of length 1, $x\in A_{o(e)}$ and $p\inv\,x\,p \in A_{t(e)}$, then $a'^{-1} \, x\, a' \in \alpha_e(A_e)$ and, in that case, $p\inv\,x\,p = a\inv\,\omega_e\alpha_e\inv(a'^{-1} \, x\, a')\,a$.

    Suppose now that $p$ is a longer $\AA$-path with last edge $e$, say, $p = p'\, (1, e, a)$. Then, similarly, if $x\in A_{o(e)}$ and $p\inv\,x\,p \in A_{t(e)}$, then $p'^{-1} \, x\, p' \in \alpha_e(A_e)$ and, in that case, $p\inv\,x\,p = a\inv\,\omega_e\alpha_e\inv(p'^{-1} \, x\, p')\,a$. It follows in particular that
    \begin{enumerate}[(a)]
        \item if the last vertex group element of $p$ is trivial, then $p\inv\, A_{o(p)}\,p \cap A_{t(p)}$ is a subgroup of $\omega_e(A_e)$, and
        \item if $p = p_1\,p_2$, where $p_1$ and $p_2$ are non-trivial $\AA$-paths, and if $x\in A_{o(p)}$ is such that $p\inv\, x\, p$ has length 0, then so does $p_1\inv\,x\,p_1$.
    \end{enumerate}

    Now suppose that $\AA$ is not acylindrical. Let $n$ be an integer larger than $(K+1)\,|E(\gr A)|$, where $K$ is greater than the total number of subgroups of the (finitely many) edge groups of $\AA$. Since $\AA$ is not acylindrical, there exists a reduced $\AA$-path \J{$p$} of length $n$ such that $p\inv A_{t(p)} p \cap A_{o(p)}$ is non-trivial. By the pigeonhole principle, $p$ visits some edge $e$ at least $K+1$ times, and $p$ factors as $p = p_0 p_1 \dots p_{K+1}$ such that $p_0, p_1, \dots, p_K$ all end with $e$ and all have trivial last vertex group element. In particular $p_1, \dots, p_K$ are circuits at $u = t(e)$. Let $q_i = p_1\dots p_i$.
    
    As observed above, for each $1\le i\le K$, the intersection $q_i\inv A_u q_i \cap A_u$ is a non-trivial subgroup of $\omega_e(A_e)$. As a result, there exist $1 \le i < j \le K$ such that $q_i\inv A_u q_i \cap A_u$ and $q_j\inv A_u q_j \cap A_u$ are equal.

    Let $F = q_i\inv A_u q_i \cap A_u = q_j\inv A_u q_j \cap A_u \le \omega_e(A_e)$ and let $t = p_{i+1}\dots p_j \in \pi_1(\AA,u)$. Let $x\in F$. By definition of $F$, there exists $y\in A_u$ such that $x = q_j\inv\, y\, q_j = t\inv\,q_i\inv\,y\,q_i\,t$. It follows that $t\,x\,t\inv = q_i\inv\,y\,q_i \in A_u$ (by Item (b) above), and hence $t\,x\,t\inv \in q_i\inv\,A_u\,q_i \cap A_u = F$. In particular, $t\,F\,t\inv = F$ and $t\inv\, F\, t = F$.

    The $\AA$-circuit $t$ is certainly reduced, since it is a factor of the reduced $\AA$-path $p$. We observe that $t$ is cyclically reduced: this is clearly the case if the first edge of $t$ (of $p_{i+1}$) is not $e\inv$. If instead the first edge of $p_{i+1}$ is $e\inv$, we let $a \in A_u$ be the first vertex group element of $p_{i+1}$. Then $a$ is also a vertex group element in $p$, and since $p$ is reduced, $a\not\in \omega_e(A_e)$ and hence $t$ is cyclically reduced. This implies, in turn, that every power of $t$ has positive length and, in particular, that $t$ has infinite order in $\pi_1(\AA,u)$.
    
    Now $t$ acts on $F$ by conjugation, and hence the subgroup $\langle F, t\rangle$ of $A_u$ is isomorphic to $F\rtimes \Z$. Let $\ell$ be the order of the group of automorphisms of the finite group $F$. Then $\langle F, t^\ell\rangle$ is isomorphic to $F\times \Z$, and hence it contains a subgroup isomorphic to $\Z \times \Z/k\Z$, where $k$ is the order of any non-trivial element of $F$.
    \end{proof}

\subsubsection{The $1$-\fcip}\label{sec: the 1-fcip}

Our second criterion for the \fgip\ uses the 1-\fcip\ (or \fcip) instead.

\begin{prop}
\label{prop: finite_pullback_1}
Let $(\AA,u_0)$ be a pointed graph of groups. Suppose that for each vertex $u\in V(\gr A)$, $A_u$ has the \fcip\ relative to the collection $\{\alpha_e(A_e)\}_{e\in \Star(u)}$. Let $\mu^B\colon (\BB, v_0)\to (\AA, u_0)$ and $\mu^C\colon (\CC, w_0)\to (\AA, u_0)$ be immersions of graphs of groups, where $\AA$, $\BB$ and $\CC$ have finite underlying graphs and finitely generated vertex and edge groups. Let $B = \mu^B_*(\pi_1(\BB, v_0))$ and $C = \mu^C_*(\pi_1(\CC, w_0))$. Then
\begin{enumerate}[(1)]
    \item\label{1st of finite_pullback_1} If each vertex group of $\AA$ has the \fgip, then for any choice of vertex $x$, $\core(\BB\wtimes_{\AA}\CC, x)$ has finite underlying graph and has finitely generated vertex and edge groups. In particular, $B^g\cap C$ is finitely generated for all $g\in \pi_1(\AA, u_0)$.
    \item\label{2nd of finite_pullback_1} If each vertex group of $\AA$ has the \sfgip, then $\core(\BB\wtimes_{\AA}\CC)$ has finite underlying graph and finitely generated vertex and edge groups. In particular, there are finitely many double cosets $B\, g\, C$ such that $B^g\cap C$ is not elliptic.
    \item\label{3rd of finite_pullback_1} If each vertex group of $\AA$ has the \sfgip\ and $\BB$ is acylindrical, then there are finitely many double cosets $B\, g\, C$ such that $B^g\cap C\neq 1$. 
\end{enumerate}
\end{prop}

\begin{proof}
Statement~\eqref{1st of finite_pullback_1} follows from Lemma~\ref{lem: conjugate subgroup}, Theorem \ref{thm: product of immersions}, Corollary \ref{cor: k-fcip} and Proposition \ref{prop: fgip_summary}. 

To verify Statement~\eqref{2nd of finite_pullback_1}, we may use Corollary \ref{cor: k-fcip} and Proposition \ref{prop: fgip_summary} to obtain that $\core(\BB\wtimes_{\AA}\CC)$ has finite underlying graph and has finitely generated vertex and edge groups. Then Theorem \ref{thm: summary immersions} implies that there are finitely many double cosets $B\,g\,C$ such that $B^g\cap C$ is not locally elliptic. Finally, since $B^g\cap C$ is finitely generated (by Statement~\eqref{1st of finite_pullback_1}), if it is locally elliptic, then it is elliptic by \cite[Corollary 7.3]{bas93}. 

Statement~\eqref{3rd of finite_pullback_1} follows in much the same way, using Theorem \ref{thm: acylindrical} instead of Theorem \ref{thm: summary immersions}.
\end{proof}

Theorem \ref{fgip_criterion_1} follows by combining Theorem~\ref{thm: bijection subgroups immersions covers} with Proposition \ref{prop: finite_pullback_1}.

\begin{thm}
\label{fgip_criterion_1}
Let $(\AA,u_0)$ be a core pointed graph of groups with finite underlying graph and finitely generated vertex and edge groups. Suppose that for each vertex $u\in V(\gr A)$, $A_u$ has the \fcip\ relative to the collection $\{\alpha_e(A_e)\}_{e\in \Star(u)}$ and that $\pi_1(\AA,u_0)$ has the \fgip\ relative to each subgroup in $\mathcal{E}$. Then:
\begin{enumerate}[(1)]
    \item If each vertex group of $\AA$ has the \fgip, then $\pi_1(\AA, u_0)$ has the \fgip\ 
    \item If each vertex group of $\AA$ has the \sfgip, then the pair $(\pi_1(\AA, u_0), \mathcal{E})$ has the \sfgip\
    \item If each vertex group of $\AA$ has the \sfgip\ and $\AA$ is acylindrical, then $\pi_1(\AA, u_0)$ has the \sfgip\
\end{enumerate}
\end{thm}

\subsubsection{The $2$-\fcip}

Our third criterion for the \fgip\ uses the 2-\fcip\

\begin{prop}
\label{prop: finite_pullback_2}
Let $(\AA,u_0)$ be a pointed graph of finitely generated groups such that for each $u\in V(\gr{A})$ and $e\in \Star(u)$ we have:
\begin{itemize}
    \item $A_u$ has the $2$-\fcip\ relative to the collection $\{\alpha_e(A_e)\}_{e\in \Star(u)}$.
    \item $(A_{u}, \alpha_e(A_e))$ has the $\alpha_e(A_e)$-\fgip
\end{itemize}

Let $\mu^B\colon (\BB, v_0)\to (\AA, u_0)$ and $\mu^C\colon (\CC, w_0)\to (\AA, u_0)$ be immersions of graphs of groups with finite underlying graphs and finitely generated vertex and edge groups. Denote by $B = \mu^B_*(\pi_1(\BB, v_0))$ and $C = \mu^C_*(\pi_1(\CC, w_0))$. Then
\begin{enumerate}
    \item If each vertex group of $\AA$ has the \fgip, then for any choice of vertex $x$, $\core(\BB\wtimes_{\AA}\CC, x)$ is a finite graph of groups with finitely generated vertex and edge groups. In particular, $B^g\cap C$ is finitely generated for all $g\in \pi_1(\AA, u_0)$.
    \item If each vertex group of $\AA$ has the \sfgip, then $\core(\BB\wtimes_{\AA}\CC) = \DD_1\sqcup \DD_2$ where $\DD_1$ has finite underlying graph and finitely generated vertex and edge groups and where $\DD_2$ has underlying graph a union of cycles with trivial vertex groups. In particular, there are finitely many double cosets $B\, g\, C$ such that $B^g\cap C$ is not elliptic or infinite cyclic.
    \item If each vertex group of $\AA$ has the \sfgip\ and if $\BB$ is also acylindrical, then there are finitely many double cosets $B\, g\, C$ such that $B^g\cap C$ is not infinite cyclic. 
\end{enumerate}
\end{prop}

\begin{proof}
The proof is identical to that of Proposition \ref{prop: finite_pullback_1}, using Proposition \ref{prop: 2-fcip} in place of Corollary \ref{cor: k-fcip}.
\end{proof}

Theorem \ref{fgip_criterion_2} follows in much the same way as Theorem \ref{fgip_criterion_1}, by combining Theorem~\ref{thm: bijection subgroups immersions covers} with Proposition \ref{prop: finite_pullback_2}.

\begin{thm}
\label{fgip_criterion_2}
Let $(\AA,u_0)$ be a pointed graph of finitely generated groups such that for each $u\in V(\gr{A})$ and $e\in \Star(u)$ we have:
\begin{itemize}
    \item $A_u$ has the $2$-\fcip\ relative to the collection $\{\alpha_e(A_e)\}_{e\in \Star(u)}$.
    \item $(A_{o(e)}, \alpha_e(A_e))$ has the $\alpha_e(A_e)$-\fgip\
\end{itemize}
If each vertex group of $\AA$ has the \fgip, then $\pi_1(\AA,u_0)$ has the \fgip\ if and only if it has the \fgip\ relative to each subgroup in $\mathcal{E}$.
\end{thm}

\begin{rem}\label{rem: no sfgip upgrade} 
A version of Theorem \ref{fgip_criterion_1} unfortunately does not hold when the \fcip\ condition is replaced with the 2-\fcip\ condition as the following example demonstrates. Consider the Klein bottle group $K = \langle a, b\mid a^2 = b^2\rangle$. This splits as an amalgam with infinite cyclic vertex groups $\langle a\rangle$ and $\langle b\rangle$ and with edge group mapping to index two subgroups on either side. This graph of groups certainly satisfies the hypotheses of Theorem \ref{fgip_criterion_2}. Moreover, each vertex group has the \sfgip\ However, the infinitely many double cosets $\langle [a, b]\rangle a^{2i}\langle [a, b]\rangle$ (as $i$ varies over the integers) all yield infinite intersections $\langle [a, b]\rangle^{a^{2i}}\cap \langle [a, b]\rangle = \langle [a, b]\rangle$, none of which are elliptic.
\end{rem}

The following corollary follows from Theorem \ref{fgip_criterion_2} combined with the following observation: if $C_1, C_2\leqslant A$ are subgroups so that $A$ has the \fcip\ relative to $C_1$ and relative to $C_2$, then $A$ has the 2-\fcip\ relative to $\{C_1, C_2\}$.

\begin{cor}
\label{cor: single_edge_2}
Let $G = A*_{\varphi}$ be a HNN extension, where $\varphi\colon C_1\to C_2$ identifies two subgroups of $A$. Suppose $A$ has the \fgip, $A$ has the \fcip\ relative to $C_i$ and $(A, C_i)$ has the $C_i$-\fgip\ for $i = 1, 2$. Then $G$ has the \fgip\ if and only if $G$ has the \fgip\ relative to $C$.
\end{cor}

\begin{rem}
\label{rem: recover_Burns_HNN}
    In 1973, Burns proved a sufficient condition for an HNN extension to have the \fgip\ \cite[Theorem 1.2]{bur73}. This condition was expressed in terms of so-called \emph{Burns subgroups} (there called AMFI subgroups): if $H*_{\phi}$ is a HNN extensions, where $\phi\colon A\to B$ is the identifying isomorphism, so that $H$ has the \fgip\ and $A, B$ are Burns subgroups of $H$, then $H*_{\phi}$ has the \fgip\ when it has the \fgip\ relative to $A$ and $B$. The definition of Burns subgroups is somewhat technical and is not given here, we refer the reader to the appendix. However, we remark that by Proposition \ref{prop: Burns collection} and Lemma \ref{lem:Burns_A_fgip}, if $A\leqslant H$ is Burns, then $H$ has the \fcip\ relative to $A$ and $(H, A)$ has the $A$-\fcip\ Hence, Corollary \ref{cor: single_edge_2} implies Burns' result. One may also use Proposition \ref{prop: finite_pullback_2} to derive more technical generalisations.
\end{rem}

\subsubsection{The $1$-\fcip\ on an orientation}
\label{sec: 1-fcip orientation}

Our fourth criterion for the \fgip\ uses the 1-\fcip, but only on certain edge groups induced by an orientation on the edges. The assumptions that are needed, as well as the proofs, are much more involved than 
for our previous criteria. We say that a group $A$ has the \emph{limited coset intersection property} (\lcip) \emph{relative to a collection of subgroups $\{A_i\}_i$} if $A$ has the \fcip\ relative to the collection $\{A_i\}_i$ and if, for every $i$, $(A,A_i)$ has the $A_i$-\fgip

\begin{prop}
\label{prop: finite_pullback_3}
Let $(\AA,u_0)$ be a pointed graph of groups in which each vertex group has the \fgip\ Suppose that there exists an orientation $E^+\subset E(\gr{A})$ on the edges such that
\begin{itemize}
    \item for each vertex $u\in V(\gr A)$, $A_u$ has the \lcip\ relative to the collection $\{\alpha_e(A_e) \mid e\in E^+,\ o(e) = u\}$;
    \item for each edge $e \in E^+(\gr A)$, either $A_e$ has the \sfgip\ or $(A_{e}, \mathcal{F}_{e})$ has the \sfgip\ (where $\mathcal F_e$ is the collection of finite subgroups of $A_e$) and there is a bound on the orders of the elements of $\mathcal{F}_e$.
\end{itemize}
Let $\mu^B\colon (\BB, v_0)\to (\AA, u_0)$ and $\mu^C\colon (\CC, w_0)\to (\AA, u_0)$ be immersions of graphs of groups. Suppose that $\AA$, $\BB$ and $\CC$ have finite underlying graphs and finitely generated vertex and edge groups, and denote by $B = \mu^B_*(\pi_1(\BB, v_0))$ and $C = \mu^C_*(\pi_1(\CC, w_0))$. Then
\begin{enumerate}[(1)]
    \item For each vertex $x$, $\core(\BB\wtimes_{\AA}\CC, x)$ is a finite graph of groups with finitely generated vertex and edge groups. In particular, $B^g\cap C$ is finitely generated for all $g\in \pi_1(\AA, u)$.
    \item $\core(\BB\wtimes_{\AA}\CC) = \DD_1\sqcup \DD_2$ where $\DD_1$ has finite underlying graph and finitely generated vertex and edge groups and where $\DD_2$ has underlying graph a union of cycles with equal finite vertex groups. In particular, there are finitely many double cosets $B\, g\, C$ such that $B^g\cap C$ is not elliptic or virtually infinite cyclic.
\end{enumerate}
\end{prop}

\begin{proof}
    Let $\DD$ be the component of $\BB\wtimes_{\AA}\CC$ containing $x$. By Proposition \ref{prop: fcip on orientation},
    $\DD$ is a union $\DD = \DD_1\cup \DD_2$ such that 
    \begin{itemize}
        \item $\gr D_1$ is finite, connected and contains $x$,
        \item $\gr D_2$ is a disjoint union of finitely many trees, each of which intersects $\gr D_1$ in a single vertex,
        \item  each vertex of $\gr D_2$ has at most one outgoing edge mapping to $E^+$.
    \end{itemize}
For each $e\in E^+$, $(A_e, \mathcal{F}_e)$ has the \sfgip\ or $A_e$ has the \sfgip\ Proposition \ref{prop: fgip_summary} then implies that every vertex and edge group of $\DD$ is finitely generated and all but finitely many of the edge groups are finite (or trivial for those mapping to an edge $e$ so that $A_e$ has the \sfgip). Moreover, since $(A_u, \alpha_e(A_e))$ has the $\alpha_e(A_e)$-\fgip\ for each $e\in E^+$ with $o(e) = u$, Proposition \ref{prop: fgip_summary} also implies that, for all but finitely many edges $h\in E(\gr{D})$ that map to $E^+$, the morphism $\alpha_h$ is an isomorphism. Thus, possibly after enlarging $\DD_1$, we may assume that every vertex group of $\DD_2$ is finite and that for every edge $h\in E(\gr{D}_2)$ that maps to $E^+$, $\alpha_h$ is an isomorphism. We now analyse the (possibly infinite) subgraph of groups $\DD_2$.

    If $y\in V(\gr{D}_2)$ is a vertex, denote by $S_y\subset \gr{D}_2$ the image of the unique path leading out of $y$ following the orientation of the edges induced by $E^+$. Note that $S_y$ is either finite or an infinite ray. Furthermore, we claim that for any pair of vertices $y, y'$ in the same component, we have $S_{y}\cap S_{y'}\neq\emptyset$. Indeed, we note that a reduced path cannot have a negative edge followed by a positive edge, since each vertex has a unique positive outgoing edge. As a result, any reduced path, and in particular a geodesic path $p$ in $\gr{D}_2$ from $y$ to $y'$, decomposes as $p = p_1p_2$ where $p_1, p_2^{-1}$ follow the orientation on the edges. As a result    
    $p_1\subset S_y$, $p_2\subset S_{y'}$ and $t(p_1) = o(p_2)\in S_y\cap S_{y'}$. Since $\alpha_h$ is an isomorphism for every $h\in E(\gr{D}_2)$ mapping to an edge in $E^+$, it follows that all reduced $\DD$-loops at $x$ do not reach any vertex in $\DD_2$ that does not lie in some $S_y$ for some $y\in \gr{D}_1\cap \gr{D}_2$.
    Finally, since each vertex group in $\DD_2$ is finite of bounded order (being contained in an adjacent edge group with associated edge mapping to $E^+$), each subgraph of groups of $\DD_2$ on $S_y$ for any $y\in V(\DD_2)$ is a directed ray in which the edge inclusions $\omega_h$ eventually become isomorphisms once one goes sufficiently far towards infinity.
    But this means that every reduced $\DD$-loop at $x$ can only visit finitely many vertices from $\DD_2$ and hence also $\DD$.
    In particular, $\core(\DD, x)$ is a finite graph of groups with finitely generated vertex and edge groups. 
    This implies that $B\cap C$ is finitely generated by Theorem~\ref{thm: product of immersions}~\eqref{item: intersection of subgroups}. To see that $B^g\cap C$ is finitely generated for any $g$, we may use Lemma~\ref{lem: conjugate subgroup} to see that there is a pointed immersion of graphs of groups $\mu^{B'}\colon(\BB', v_0')\to(\AA, u_0)$ such that $\BB'$ has finite underlying graph of groups, finite vertex and edge groups and $(\mu^{B'})_*(\pi_1(\BB', v_0')) = B^g$. This establishes the first fact.

    The above implies that every locally elliptic intersection $B^g\cap C$ must actually be elliptic since it is finitely generated (as before, using \cite[Corollary 7.2]{bas93}). Now Proposition \ref{prop: fcip on orientation} together with the observations above implies that $\core(\BB\wtimes_{\AA}\CC)$ has the property that all but finitely many components have underlying graph a circle, each vertex group is finite and each edge group inclusion an isomorphism. The fundamental group of such a graph of groups is a semidirect product $N\rtimes \Z$ with $N$ a finite group (isomorphic to each of the vertex groups). Thus, for all but finitely many double cosets $B\, g\, C$, $B^g\cap C$ is either elliptic or virtually $\Z$ as claimed.
\end{proof}

Theorem \ref{fgip_criterion_3} follows in much the same way as Theorem \ref{fgip_criterion_2}, by combining Theorem~\ref{thm: bijection subgroups immersions covers} with Proposition \ref{prop: finite_pullback_3}.

\begin{thm}
\label{fgip_criterion_3}
Let $(\AA,u_0)$ be a pointed graph of groups in which each vertex group has the \fgip, and let $E^+\subset E(\gr{A})$ be an orientation on the edges satisfying the same properties as in Proposition~\ref{prop: finite_pullback_3}.
Then $\pi_1(\AA,u_0)$ has the \fgip\ if and only if it has the \fgip\ relative to each subgroup in $\mathcal{E}$.
\end{thm}

In the case that $\AA$ is a single edge graph of groups, we may state Theorem \ref{fgip_criterion_3} in a simplified form. 

\begin{cor}
\label{cor: single_edge_3}
Let $G = A*_{\varphi}$ be a HNN extension, where $\varphi\colon C\to \varphi(C)$ identifies two subgroups of $A$. Suppose $A$ has the \fcip\ relative to $C$, $(A, C)$ has the $C$-\fgip\ and $C$ has the \sfgip\ Then $G$ has the \fgip\ if and only if $G$ has the \fgip\ relative to $C$.
\end{cor}

\begin{rem}
    The inspiration for Corollary \ref{cor: single_edge_3}, and hence for Theorem \ref{fgip_criterion_3} is the case of solvable Baumslag--Solitar groups $\bs(1, n)$ handled by Moldavanskiĭ in \cite{mol68}. Indeed, Corollary \ref{cor: single_edge_3} applies directly to the usual single vertex and single edge HNN extension decomposition of $\bs(1, n)$, recovering Moldavanskiĭ's result. More generally, we will use Theorem \ref{fgip_criterion_3} in Section \ref{sec: vcyclic} when we characterise graphs of virtually (infinite) cyclic groups with the \fgip
\end{rem}

\begin{rem}
\label{rem: no_fgip}
    The \sfgip\ hypothesis in Corollary \ref{cor: single_edge_3} is needed as the following example demonstrates. Consider the graph of groups $\AA$ with a single vertex $u$ with $A_{u} = \langle a_1, a_2\rangle \cong \Z^2$ and a single edge $e$ with $A_{e} = \langle b_1, b_2\rangle\cong \Z^2$ so that $\alpha_e(b_i) = a_i$ and $\omega_e(b_i) = a_i^2$ for $i = 1, 2$. Let $\hat{e} = (1, e, 1)$. Consider the finitely generated subgroups $B = \langle a_1, \hat{e}a_2\rangle\cong \bs(1, 2)$ and $C = \langle a_1, \hat{e}\rangle \cong \bs(1, 2)$ and let $(\BB, v), (\CC, w)\to (\AA, u)$ be the corresponding immersions of graphs of groups. It can be verified that $\DD = \core(\BB\wtimes_{\AA}\CC, B_vC_w)$ has the following underlying graph
    \[
    \begin{tikzcd}
    B_{v}C_{w} \arrow[r, "B_fC_g"] &[4em] B_{v}a_2C_{w} \arrow[r, "B_fb_2C_g"] &[4em] B_{v}a_2^3C_{w} \arrow[r, "B_fb_2^3C_g"] &[4em] B_{v}a_2^7C_{w} \arrow[r, "B_fb_2^7C_w"] &[4em] \ldots
    \end{tikzcd}
    \]
    and each vertex group is $\langle a_1\rangle$ and each edge group is $\langle b_1\rangle$. More explicitly, we have
    \[
    B\cap C = \langle \hat{e}^ia_1\hat{e}^{-i} \mid i\geqslant 0\rangle \cong \Z\left[\frac{1}{2}\right]
    \]
    which is not finitely generated. The reader should compare this ascending union of groups with the ascending union of groups obtained in the second paragraph of the proof of Proposition \ref{prop: finite_pullback_3}.
\end{rem}

\begin{rem}
    In \cite{bur73}, Burns mentioned that it was unknown whether, in his main theorem on HNN extensions (see Remark \ref{rem: recover_Burns_HNN}), the assumption that the associated group is a Burns subgroup (called AMFI subgroups there) with respect to both inclusions can be replaced with the assumption that it is Burns with respect to only one of the inclusions. The example in Remark \ref{rem: no_fgip} shows that this is not possible. However, Corollary \ref{cor: single_edge_3} shows that the assumption can be dropped if one adds the assumption that the associated group has the \sfgip
\end{rem}

\subsubsection{The 1-\fcip\ on an acyclic orientation}\label{sec: 1-fcip acyclic orientation}

We may improve the result from Section \ref{sec: 1-fcip orientation} in the case the orientation gives rise to an acyclic graph. The reader should have the amalgamated free product case in mind.

\begin{prop}
\label{prop: finite_pullback_4}
Let $(\AA,u_0)$ be a pointed graph of groups in which each vertex group has the \fgip\ Let $E^+\subset E(\gr{A})$ be an orientation on the edges and let $\gr A^+$ be the resulting directed graph. Suppose that $\gr A^+$ is acyclic and that, for each vertex $u\in V(\gr A)$, $A_u$ has the \lcip\ relative to the collection $\{\alpha_e(A_e) \mid e\in E^+,\ o(e) = u\}$.

Let $\mu^B\colon (\BB, v_0)\to (\AA, u_0)$ and $\mu^C\colon (\CC, w_0)\to (\AA, u_0)$ be immersions of graphs of groups, where $\AA$, $\BB$ and $\CC$ have finite underlying graphs and finitely generated vertex and edge groups. Denote by $B = \mu^B_*(\pi_1(\BB, v_0))$ and $C = \mu^C_*(\pi_1(\CC, w_0))$. Then:
\begin{enumerate}
    \item For any choice of vertex $x$, $\core(\BB\wtimes_{\AA}\CC, x)$ is a finite graph of groups with finitely generated vertex and edge groups. In particular, $B^g\cap C$ is finitely generated for all $g\in \pi_1(\AA, u_0)$.
    \item $\core(\BB\wtimes_{\AA}\CC)$ is a finite graph of finitely generated groups. In particular, there are finitely many double cosets $B\, g\, C$ such that $B^g\cap C$ is not elliptic.
    \item If each vertex group of $\AA$ has the \sfgip\ and $\BB$ is also acylindrical, then there are finitely many double cosets $B\, g\, C$ so that $B^g\cap C\neq 1$.
\end{enumerate}
\end{prop}

\begin{proof}
    Since $\gr{A}^+$ is acyclic, the graph of groups $\BB\wtimes_{\AA}\CC$ is also acyclic with the induced orientation. Since $\BB$, $\CC$ have finite underlying graphs, this implies that there is a constant $\kappa$ (we may take $\kappa = |E\gr(B)|\cdot |E(\gr{C})|$) so that each directed path in $\BB\wtimes_{\AA}\CC$ has length at most $\kappa$. The \fcip\ hypothesis implies that there are finitely many vertices $x$ with distinct edges $h_1, h_2\in \Star(x)$ that map to edges in $E^+$. The fact that $(A_{o(e)}, \alpha_e(A_e))$ has the $\alpha_e(A_e)$-\fgip\ for each $e\in E^+$ implies that all but finitely many edges $h$ in $\BB\wtimes_{\AA}\CC$ that map to an edge in $E^+$ have $\alpha_h$ an isomorphism. Thus, if we remove each (non basepoint) vertex $x$ so that $\Star(x) = \emptyset$ or $\Star(x) = \{h\}$ and $\alpha_h$ is an isomorphism and repeat this $\kappa$ times, then the above implies that we will be left with a finite graph of finitely generated groups. Moreover, the inclusion of $\core(\BB\wtimes_{\AA}\CC)$ (or $\core(\BB\wtimes_{\AA}\CC, x)$) into this new graph of groups necessarily induces an isomorphism on $\pi_1$ on each component. Thus, $\core(\BB\wtimes_{\AA}\CC)$ and $\core(\BB\wtimes_{\AA}\CC, x)$ are finite graphs of finitely generated groups. Finally, the three statements now follow in the same way as in Proposition \ref{prop: finite_pullback_1}.
\end{proof}

Again, Theorem \ref{fgip_criterion_4} follows in much the same way as Theorem \ref{fgip_criterion_2}, by combining Theorem~\ref{thm: bijection subgroups immersions covers} with Proposition \ref{prop: finite_pullback_4}.

\begin{thm}
\label{fgip_criterion_4}
Let $(\AA,u_0)$, $E^+$ and $\gr A^+$ satisfy the same conditions as in Proposition~\ref{prop: finite_pullback_4}. Then $\pi_1(\AA,u_0)$ has the \fgip\ if and only if it has the \fgip\ relative to each subgroup in $\mathcal{E}$. 

If, additionally, each vertex group has the \sfgip, $\AA$ is acylindrical and $\pi_1(\AA, u_0)$ has the \fgip, then $\pi_1(\AA, u_0)$ also has the \sfgip\
\end{thm}

\begin{cor}
\label{cor: single_edge_4}
Let $G = A*_CB$, where $A$ has the \fgip\ and the \lcip\ relative to $C$ and $B$ has the \fgip\ Then $G$ has the \fgip\ if and only if $G$ has the \fgip\ relative to $C$.
\end{cor}

\begin{rem}
    By Proposition \ref{prop: Burns collection} and Lemma \ref{lem:Burns_A_fgip} combined with Corollary \ref{cor: single_edge_4}, we have shown that, if $G$ is an amalgamated product $G=A*_CB$ such that $A$ and $B$ have the \fgip\ and $C$ is a Burns subgroup of $A$, then $G$ has the \fgip\ Thus, we have recovered Cohen's result \cite[Theorem 2]{coh76} (see also the slightly weaker result of Burns \cite[Theorem 2.3]{bur72}).

    Kapovich generalised Cohen's amalgam result in \cite[Proposition 5.4]{kap02} to star graphs of groups in which the edge groups at the central vertex form a Burns collection. This is also recovered by Theorem \ref{fgip_criterion_4} using Proposition \ref{prop: Burns collection} and Lemma \ref{lem:Burns_A_fgip} as above.
\end{rem}

\subsubsection{Combining along subgraphs of groups}

Our final criterion allows us to prove the \fgip\ for graphs of groups, provided we know the \fgip\ holds for subgraphs of groups which do not interact with each other too much.

\begin{prop}
\label{prop: combine_subgraphs}
    Let $(\AA, u_0)$ be a pointed graph of groups, let $\{\AA_{\alpha}\}_{\alpha}$ be a finite collection of subgraphs of groups such that $\gr{A} = \bigcup_{\alpha}\gr{A}_{\alpha}$ and $\gr{A}_{\alpha}\cap\gr{A}_{\beta}\subset V(\gr{A})$ for each $\alpha\neq\beta$. Suppose also that for each vertex $u\in V(\gr{A})$ and each pair of edges $e, f\in \Star(u)$ that do not both lie in the same $\gr{A}_{\alpha}$, we have that $A_u$ has the weak \fcip\ relative to $\{\alpha_e(A_e), \alpha_f(A_f)\}$.
    
    Let $\mu^B\colon \BB\to \AA$ and $\mu^C\colon \CC\to \AA$ be immersions of finite graphs of groups such that $\BB$ and $\CC$ have finitely generated vertex and edge groups and denote by $\DD = \BB\wtimes_{\AA}\CC$ their $\AA$-product. For each $\alpha$, let also $\BB_{\alpha} = (\mu^B)^{-1}(\AA_{\alpha})$ and $\CC_{\alpha} = (\mu^C)^{-1}(\AA_{\alpha})$. If each component of $\BB_{\alpha}\wtimes_{\AA_{\alpha}}\CC_{\alpha}$ has finitely generated fundamental group, then each component of $\DD$ has finitely generated fundamental group.
\end{prop}

\begin{proof}
    By definition, each $\BB_{\alpha}\wtimes_{\AA_{\alpha}}\CC_{\alpha}$ can be naturally considered as a subgraph of groups of $\BB\wtimes_{\AA}\CC$. Moreover, each edge in $\BB\wtimes_{\AA}\CC$ lies in $\BB_{\alpha}\wtimes_{\AA_{\alpha}}\CC_{\alpha}$ for a unique $\alpha$. By Proposition \ref{prop: combination_components}, only finitely many vertices in $\BB\wtimes_{\AA}\CC$ have outgoing edges lying in $\BB_{\alpha}\wtimes_{\AA_{\alpha}}\CC_{\alpha}$ for two distinct values of $\alpha$. Let us now consider an arbitrary component $\DD'\subset \DD$. If $\DD'$ is a
    component of some $\BB_{\alpha}\wtimes_{\AA_{\alpha}}\CC_{\alpha}$, then $\pi_1(\DD', x)$ is finitely generated for any basepoint $x$ by assumption. Suppose now that $\DD'$ contains edges from $\BB_{\alpha}\wtimes_{\AA_{\alpha}}\CC_{\alpha}$ for at least two distinct values of $\alpha$. Let $Y$ be an indexing set with a value for each component of $\sqcup_{\alpha}\BB_{\alpha}\wtimes_{\AA_{\alpha}}\CC_{\alpha}$ that lies in $\DD'$ and denote by $\DD_{\gamma}$ for $\gamma\in Y$ the corresponding component. For each $\gamma\in Y$ choose a vertex $x_{\gamma}\in \DD_{\gamma}$. Let $\Gamma\subset \gr{D}$ be any finite subgraph containing the basepoint $x$ and all of the finitely many vertices $x_{\gamma}$ for $\gamma\in Y$. Choose reduced $\DD$-paths $d_{\gamma}$ connecting $x$ with $x_{\gamma}$ for all $\gamma\in Y$. Now we have
    \[
    \pi_1(\DD', x) = \left\langle \bigcup_{\gamma\in Y}d_{\gamma}\pi_1(\DD_{\gamma}, x_{\gamma})d_{\gamma}^{-1}, \pi_1(\Gamma, x)\right\rangle.
    \]
    Thus, each component of $\DD$ has fundamental group generated by finitely many groups isomorphic to fundamental groups of components of $\BB_{\alpha}\wtimes_{\AA_{\alpha}}\CC_{\alpha}$ for various values of $\alpha$, which are finitely generated, and $\pi_1(\Gamma, x)$ which is also finitely generated. Hence, each component of $\DD$ has finitely generated fundamental group as claimed.
\end{proof}

Once more, we may use Theorem~\ref{thm: bijection subgroups immersions covers} and Proposition \ref{prop: combine_subgraphs} to obtain Theorem \ref{thm: combine_subgraphs}.

\begin{thm}
\label{thm: combine_subgraphs}
     Let $(\AA, u_0)$ be a pointed graph of groups, let $\{\AA_{\alpha}\}_{\alpha}$ be a finite collection of subgraphs of groups so that $\gr{A}_{\alpha}\cap\gr{A}_{\beta}\subset V(\gr{A})$ for each $\alpha\neq\beta$ and so that $\gr{A} = \bigcup_{\alpha}\gr{A}_{\alpha}$. Suppose also that for each vertex $u\in V(\gr{A})$ and each pair of edges $e, f\in \Star(u)$ that do not both lie in the same $\gr{A}_{\alpha}$, we have that $A_u$ has the weak \fcip\ relative to $\{\alpha_e(A_e), \alpha_f(A_f)\}$.
     
     If $\pi_1(\AA_{\alpha}, u_{\alpha})$ has the \fgip\ for each $\alpha$, then $\pi_1(\AA, u_0)$ has the \fgip\ if and only if it has the \fgip\ relative to each subgroup in $\mathcal{E}$.
\end{thm}

\begin{proof}
    Let $B$ and $C$ be finitely generated subgroups of $\pi_1(\AA,u_0)$. Just as in the proof of the other theorems in this subsection, we may use Theorem~\ref{thm: bijection subgroups immersions covers} to obtain immersions of pointed graphs of groups $\mu^B\colon (\BB,v_0) \to (\AA,u_0)$ and $\mu^C\colon (\CC,w_0) \to (\AA,u_0)$, where $\BB$ and $\CC$ have finite underlying graphs, finite edge groups, finitely generated vertex groups and $\mu^B_*(\pi_1(\BB, v_0)) = B$, $\mu^C_*(\pi_1(\CC, w_0)) = C$.

    Since we assumed that each $\pi_1(\AA_{\alpha}, u_{\alpha})$ has the \fgip, using the notation of Proposition \ref{prop: combine_subgraphs}, we see that each component of $\BB_{\alpha}\wtimes_{\AA_{\alpha}}\CC_{\alpha}$ has finitely generated fundamental group by Theorem~\ref{thm: product of immersions}~\eqref{item: localy_elliptic_double_cosets}. Thus, we may apply Proposition \ref{prop: combine_subgraphs} to conclude that each component of $\DD = \BB\wtimes_{\AA}\CC$ has finitely generated fundamental group. Then by Theorem~\ref{thm: product of immersions}~\eqref{item: intersection of subgroups} we see that $B\cap C$ is finitely generated and so $\pi_1(\AA, u_0)$ has the \fgip\ as claimed. 
\end{proof}

\section{The \fgip\ for graphs of virtually infinite cyclic groups}
\label{sec: vcyclic}

The aim of this section is to characterise precisely when a fundamental group of graph of virtually infinite cyclic groups has the \fgip\ We first record the following statement, which follows directly from the definition.

\begin{fact}\label{fact: elementary virtually Z}
Let $G$ be a virtually $\Z$ group. If $H, K$ are infinite subgroups of $G$, then $H$, $K$ and $H\cap K$ are infinite, virtually $\Z$, and they have finite index in $G$.
\end{fact}

We prove the following theorem.

\begin{thm}
\label{thm: vcyclic_fgip}
    Let $\AA$ be a graph of groups with virtually $\Z$ vertex and edge groups and let $u\in V(\gr A)$. The group $\pi_1(\AA, u)$ has the \fgip\ if and only if it does not contain $F_2\times\Z$ as a subgroup. Moreover, if $\AA$ is reduced, this is equivalent to $\AA$ being of one of the following forms:
    \[
    \begin{tikzcd}
A & \textrm{or} & A \arrow[r, "2" near start, "2" near end] & B & \textrm{or} & A  \arrow["1"' near start, "k"' near end, loop, distance=2em, in=35, out=325]
    \end{tikzcd}
    \]
where the integers $1$, $2$ and $k > 0$ on the edges indicate the indices of the edge groups in the adjacent vertex groups.
\end{thm}

\begin{rem}\label{rk: general edge groups}
Theorem \ref{thm: vcyclic_fgip} also holds if we relax the assumptions on the vertex groups. We shall prove this later with Corollary~\ref{cor: main locally qc characterization}, using Theorem \ref{thm: vcyclic_fgip}.
\end{rem}

\begin{rem}\label{rk: decide fgip for virtually Z}
    Let $\AA$ be a graph of groups with virtually $\Z$ vertex and edge groups and let $u\in V(\gr A)$. Suppose that the graph $\gr A$ is finite, the vertex and edge groups and edge morphism between them are given explicitly, and one can decide, given a morphism from an edge group to a vertex group, whether it is an isomorphism and, in that case, compute its inverse. In view of Remark~\ref{rk: reduction is computable}, deciding whether the group $\pi_1(\AA, u)$ has the \fgip\ is decidable.
\end{rem}

We immediately record the following corollary, which follows directly from Theorem~\ref{thm: vcyclic_fgip}. Recall that a \emph{generalised Baumslag--Solitar group} is one that splits as a graph of groups, with infinite cyclic vertex and edge groups. Since an amalgam of two infinite cyclic groups over index two subgroups is isomorphic to $\bs(1, -1)$, we may reduce the three cases appearing in Theorem \ref{thm: vcyclic_fgip} to only two.
\begin{cor}\label{cor: fgip for GBS}
A generalised Baumslag--Solitar group has the \fgip\ if and only if it does not contain $F_2\times \Z$ as a subgroup, if and only if it is isomorphic to $\ZZ$ or $\bs(1,n)$.
\end{cor}

The case of Baumslag--Solitar groups $\bs(m, n)$ (single vertex and single edge case from Corollary \ref{cor: fgip for GBS}) is due to Moldavanski\u\i\ \cite{mol68} if $m = 1$ and to Paramantzoglou \cite{pa12} in the general case.

The proof of Theorem~\ref{thm: vcyclic_fgip}, which will be given at the end of this section, relies on the consideration of ``small'' graphs of groups $\AA$, see Propositions~\ref{prop: vZ_amalgam2} to~\ref{prop: vZ_amalgam_HNN}. Before we consider these special cases, we gather a few technical lemmas that will be used repeatedly (Lemmas~\ref{lm: semidirect product}, \ref{lm: general characteristic} and~\ref{lm: elementary virtually Z}).

\begin{lem}\label{lm: semidirect product}
Let $G$ be a group, let $H,K$ be subgroups of $G$ such that $H$ has a trivial center and $K$ is contained in the centraliser of $H$. Then the subgroup $\langle H,K\rangle$ is isomorphic to $H \times K$.
\end{lem}

\begin{proof}
The map $\phi\colon H\times K \to G$ which sends $(h,k)$ to $hk$ is a morphism since every element of $H$ commutes with every element of $K$. It is injective since $\phi(h,k) = 1$ implies $h\in K$: then $h$ is central in $H$ and hence $h = k = 1$.
\end{proof}

Say that a subgroup $H$ of $G$ is \emph{injectively characteristic} if every injective endomorphism of $G$ maps $H$ into itself. We record the following elementary facts regarding these subgroups.

\begin{lem}\label{lm: general characteristic}
Every subgroup of $\Z$ is injectively characteristic.

If $H$ is injectively characteristic in $G$ and $\phi$ is an automorphism of $G$, then $\phi(H) = H$.

If $K$ is injectively characteristic in $H$ and $H$ is injectively characteristic (resp. normal) in $G$, then $K$ is injectively characteristic (resp. normal) in $G$.
\end{lem}

\begin{lem}\label{lm: elementary virtually Z}\label{fully_char}
Every virtually $\Z$ group admits an infinite cyclic subgroup which is injectively characteristic.
\end{lem}

\begin{proof}
We use the well-known fact (see \cite[Lemma 11.4]{hem76}) that if $G$ is virtually $\Z$, then $G$ is isomorphic to a semidirect product $N\rtimes H$ where $N$ is finite and $H$ is either $\Z$ or the infinite dihedral group $D_{\infty}$. In particular, there is an exact sequence of the form
 \[
\begin{tikzcd} 1 \arrow[r] & N \arrow[r,"\beta"] & G \arrow[r,"\alpha"] & H \arrow[r]  & 1, \end{tikzcd}
 \]
Note that, for every $g, g'\in G$, if $\alpha(g) = \alpha(g')$, then there exists $n\in N$ such that $g = \beta(n)\ g'$.

Since $\beta(N) = \ker\alpha$, $\beta(N)$ is normal in $G$ and $G$ acts on the finite group $\beta(N)$ by conjugation. Let $k = |\aut(N)|$. Then conjugation by any element of the form $g^k$ ($g\in G$) is the identity on $\beta(N)$. In particular, every $k$-th power commutes with $\beta(N)$.

Let $\psi$ be an injective endomorphism of $G$ and suppose that $z$ is an infinite order element of $G$ such that $\alpha(\psi(z))$ is a proper power of $\alpha(z)$, say $\alpha(\psi(z)) = \alpha(z)^\ell$ ($\ell \ne 0$). We claim that $g = z^{k|N|}$ has the announced property. Indeed, $\alpha(\psi(z^k)) = \alpha(z^{k\ell})$, so there exists $n\in N$ such that $\psi(z^k) = \beta(n)\ z^{k\ell}$. It follows that, for any integer $j$, $\psi(z^{kj}) = \beta(n)^j\ z^{k\ell j}$. In particular, for $j = |N|$, $\psi(g) = \psi(z^{k|N|}) = z^{k|N|\ell} = g^\ell$.

Now we only need to establish the existence of such an element $z$ in $G$. If $H = \Z$, we can take $z$ to be a pre-image under $\alpha$ of a generator of $H$, since $\alpha(G) = \langle \alpha(z)\rangle$.

If $H = D_\infty$, recall that $H$ is isomorphic to $\Z/2\Z \ast \Z/2\Z$. We let $a$ and $b$ be the generators of the two copies of $\Z/2\Z$, $a'$ and $b'$ be pre-images under $\alpha$ of $a$ and $b$, and $z = a'b'$. Then $\alpha(z) = ab$ has infinite order, and hence so does $z$. Moreover $\alpha(\psi(z))$ (as well as every element of $H$) is of the form $b^\epsilon(ab)^\ell a^\eta$ for some $\ell\in \N$, $\epsilon, \eta\in \{0,1\}$. Since $\psi$ is injective, $\alpha(\psi(z))$ has infinite order, and hence $\epsilon = \eta$. Thus $\alpha(\psi(z))$ is either a positive power of $ab$, or a positive power of $ba = (ab)\inv$: in either case, $\alpha(\psi(z))$ is a power of $\alpha(z)$ and we are done.
\end{proof}

We now consider small graphs of virtually $\Z$ groups with respect to the \fgip

\begin{prop}
\label{prop: vZ_amalgam2}
    If $G$ splits as the graph of virtually $\ZZ$ groups
    \[
    \begin{tikzcd}
    B & A \arrow[l, "2"' near start, "e"', "2"' near end] \arrow[r, "2" near start, "f", "2" near end] & C
    \end{tikzcd}
    \]
    where the integers on the edges indicate the indices of the edge groups in the adjacent vertex groups, then $G$ contains a copy of $F_2\times \Z$.
\end{prop}

\begin{proof}
Since $\alpha_e(A_e)$ has index 2 in $A$, it is infinite and hence virtually $\Z$. As a result, $A_e$ is virtually $\Z$ and so is $A_f$. By Lemma~\ref{fully_char}, there exist infinite orders elements $d_e \in A_e$ and $d_f\in A_f$ such that $\langle d_e\rangle$ and $\langle d_f\rangle$ are injectively characteristic in $A_e$ and $A_f$, respectively.

Then $\langle \alpha_e(d_e)\rangle$ and $\langle \alpha_f(d_f)\rangle$ are infinite and injectively characteristic in $\alpha_e(A_e)$ and $\alpha_f(A_f)$, respectively. By Fact~\ref{fact: elementary virtually Z} and Lemma~\ref{lm: elementary virtually Z}, the intersection $\langle \alpha_e(d_e)\rangle \cap \langle \alpha_e(d_f)\rangle$ is infinite, and injectively characteristic in both $\alpha_e(A_e)$ and $\alpha_f(A_f)$. Let $d$ be a generator of $\langle \alpha_e(d_e)\rangle \cap \langle \alpha_e(d_f)\rangle$, $d_B = \omega_e(\alpha_e\inv(d))$ and $d_C = \omega_f(\alpha_f\inv(d))$. Then $\langle d_B\rangle$ is infinite and injectively characteristic in $\omega_e(A_e)$, and $\langle d_C\rangle$ is infinite and injectively characteristic in $\omega_f(A_f)$.

Since $\alpha_e(A_e)$ has index 2 in $A$, it is normal and so is $\langle d\rangle$. Similarly, $\langle d_B\rangle$ is normal in $B$ and $\langle d_C\rangle$ is normal in $C$.

This implies that $\langle d\rangle$ is normal in $G$. Indeed, $G$ is generated by $A$, $(1,e,b,e\inv,1)$ ($b\in B$) and $(1,f,c,f\inv,1)$ ($c\in C$). Then, if $b\in B$, we have
$$d^{(1,e,b,e\inv,1)} = (1,e,b\inv,e\inv,d,e,b,e\inv,1) = (1,e,b\inv d_Bb,e\inv,1) = (1,e,d_B^\epsilon,e\inv,1),$$
where $\epsilon = \pm1$, and hence $d^{(1,e,b\inv,e\inv,1)} = d^\epsilon$. It follows that $\langle d\rangle^{(1,e,b\inv,e\inv,1)} = \langle d\rangle$ and, similarly, $\langle d\rangle^{(1,f,c,f\inv,1)} = \langle d\rangle$. Thus $\langle d\rangle$ is normal in $G$.

Let now $a_1$ be an element of the complement of $\alpha_e(A_e)$ in $A$, $a_2$ be in the complement of $\alpha_f(A_f)$ in $A$, $b$ in the complement of $\omega_e(A_e)$ in $B$ and $c$ in the complement of $\omega_f(A_f)$ in $C$. Let also
\begin{align*} 
        p_0 &= (1, e, b, e\inv, a_1),\\
        q_0 &= (1, f, c, f\inv, a_2),
\end{align*}
and $p = p_0^2$ and $q = q_0^2$. We claim that the subgroup $\langle d, p, q\rangle$ is isomorphic to $\Z \times F_2$. First note that $\langle d\rangle^{p_0} = \langle d\rangle$, so that $d^{p_0} \in \{d,d\inv\}$, and hence $d^p = d$. Similarly $d^q = q$.
In view of Lemma~\ref{lm: semidirect product}, we only need to show that $\langle p,q\rangle$ is isomorphic to $F_2$. Let indeed $\chi\colon F(x,y)\to G$ be given by $\chi(x) = p$ and $\chi(y) = q$. If $z$ is a non-trivial freely reduced word in $F(x,y)$, then $\chi(z)$ is a concatenation of copies of $p$, $q$, $p\inv$ and $q\inv$, which is always $\AA$-reduced. Accordingly, $\chi$ is one-to-one and $\langle p,q\rangle$ is isomorphic to $F_2$.
\end{proof}

\begin{prop}
\label{prop: vZ_amalgam}
    Let $G = A*_{C}B$ be an amalgam of two virtually $\Z$ groups over a virtually $\Z$ subgroup. The following are equivalent:
    \begin{enumerate}[(1)]
        \item\label{itm:1} $G$ has the \fgip
        \item\label{itm:2} $G$ does not contain a copy of $F_2\times \Z$ as a subgroup.
        \item\label{itm:3} $C$ is either equal to $A$ or $B$, or $C$ has index two in both $A$ and $B$.
    \end{enumerate}
\end{prop}

\begin{proof}
\eqref{itm:1} implies \eqref{itm:2}, since $F_2\times \Z$ does not have the \fgip

We now show that \eqref{itm:3} implies \eqref{itm:1}. If $C$ is equal to $A$, then $G = B$, so $G$ is virtually $\Z$, and hence~$G$ has the \fgip\
The same holds if $C = B$. Let us now assume that $C$ has index 2 in both $A$ and $B$.

In this situation, $C$ is normal in $A$ and in $B$, and hence in $G$, and the quotient $G/C$ is isomorphic to $\Z/2\Z \ast \Z/2\Z$. The derived subgroup $Z$ of $G/C = \langle r,s \mid r^2 = s^2 = 1\rangle$ is the kernel of the abelianisation morphism $\alpha\colon G/C \to \langle r,s\mid r^2 = s^2 = [r,s] = 1\rangle = \Z/2\Z \times \Z/2\Z$; thus $Z$ is the set of sequences of alternating $r$'s and $s$'s with an even number of $r$'s and $s$'s, that is, $Z = \langle (rs)^2\rangle$, which is infinite cyclic, and has index 4 in $G/C$. Let $K$ be the pre-image of $Z$ in $G$. Then $K$ also has finite index in $G$, and the following short exact sequence
 \[
\begin{tikzcd} 1 \arrow[r] & C \cap K \arrow[r] & K \arrow[r,"\alpha"] & Z \arrow[r]  & 1 \end{tikzcd}
 \]
splits since $Z$ is free. It follows that $K$ is isomorphic to a semidirect product $(C\cap K) \rtimes Z$. If $H$ is a subgroup of $K$, then $H$ is isomorphic to $(C\cap H) \rtimes \alpha(H)$. Since $C$ is virtually $\Z$, all its subgroups are finitely generated, so $H$ is finitely generated. It follows, in turn and since $K$ has finite index in $G$, that every subgroup of $G$ is finitely generated. Therefore $G$ has the \fgip

We finally focus on the proof that \eqref{itm:2} implies \eqref{itm:3}, or rather, its contrapositive. Suppose that $C$ has index at least three in $A$ and index at least 2 in $B$. We first show that $C$ contains an infinite cyclic subgroup which is normal in $A$ and in $B$. Let indeed $a\in A$ and $b\in B$ be infinite order elements such that $\langle a\rangle$ is injectively characteristic in $A$ and $\langle b\rangle$ is injectively characteristic in $B$ (Lemma~\ref{fully_char}). By the same lemma, $C$ has finite index in both $A$ and $B$ and hence $C$ contains a finite power of $a$ and a finite power of $b$. Thus we may assume that $a, b \in C$. Then the intersection $\langle a\rangle \cap \langle b\rangle$ is an infinite cylic subgroup of both $\langle a\rangle$ and $\langle b\rangle$, injectively characteristic in both, and hence normal in $A$ and in $B$. Let $c$ be a generator of $\langle a\rangle \cap \langle b\rangle$.

Let now $a_1, a_2\in A$ be such that $C$, $Ca_1$ and $Ca_2$ are pairwise distinct, and let $b\in B$ such that $C$ and $Cb$ are distinct. Since $\langle c\rangle$ is normal in $A$ and in $B$, we have $\langle c\rangle^b = \langle c\rangle^{a_1} = \langle c\rangle^{a_2} = \langle c\rangle$. It follows that $c^{ba_1}$ is equal to $c$ or $c\inv$, and hence $c^{ba_1ba_1} = c$. Similarly, $c^{ba_2ba_2} = c$. This implies that $\langle ba_1ba_1, ba_2ba_2\rangle$ commutes with $\langle c\rangle$. 

Now we verify that $\langle ba_1ba_1, ba_2ba_2, c\rangle$ is isomorphic to $F_2\times \Z$. It is convenient, at this point, to use the graph of groups notation for $G$: $G = \pi_1(\AA,u)$, where $\gr A$ has two vertices $u$ and $v$ and an edge $e$ from $u$ to $v$, $A_u = A$, $A_v = B$, $A_e = C$ and $\alpha_e$ and $\omega_e$ are the identity map on $A_e$ (since we have, in this proof, considered $C$ as a subgroup of both $A$ and $B$).

Let $\chi\colon F(x,y) \to \langle ba_1ba_1, ba_2ba_2\rangle$ be the morphism which maps $x$ to $ba_1ba_1$ and $y$ to $ba_2ba_2$. If $z$ is a non-trivial reduced word in $F(x,y)$, then $\chi(z)$ is the $=_\AA$-equivalence class of a concatenation of the paths
$$(1,e,b,e\inv,a_1,e,b,e\inv,a_1),\quad (1,e,b,e\inv,a_2,e,b,e\inv,a_2)$$
and their inverses. The $=_\AA$-cancellations between these $\AA$-paths and their inverses are short, and the vertex group elements along this concatenation (once reduced) are equal to $b$, $a_1$, $a_2$, $a_1a_2\inv$, $a_1\inv a_2$ and their inverses. Our choices for $b$, $a_1$ and $a_2$ is such that none of these elements is in $C$. As a result, the $\chi$-image of a non-trivial element $z \in F(x,y)$ is a non-trivial reduced $\AA$-path. Thus $\chi$ is an isomorphism and $H$ is isomorphic to $F_2$. 

Lemma~\ref{lm: semidirect product} then asserts that $\langle ba_1ba_1, ba_2ba_2, c\rangle$ is isomorphic to $F_2\times \Z$.
\end{proof}

\begin{prop}
\label{prop: vZ_HNN}
    Let $G = A*_\psi$ be an HNN extension, where $\psi\colon C\to D$ is an isomorphism between two subgroups of $A$, and $A$, $C$ and $D$ are virtually $\Z$. The following are equivalent:
    \begin{enumerate}
        \item\label{itm:1'} $G$ has the \fgip
        \item\label{itm:2'} $G$ does not contain a copy of $F_2\times \Z$.
        \item\label{itm:3'} $C = A$ or $D = A$.
    \end{enumerate}
\end{prop}

\begin{proof}
As in the previous statement, \eqref{itm:1'} implies \eqref{itm:2'} since $F_2\times \Z$ does not have the \fgip

By definition, $G = \langle A, t \mid c^t = \psi(c), c\in C\rangle$. Note that $G$ is the fundamental group of the graph of groups $\AA$ with a single vertex $u$, a single edge $e$ (and its inverse), $A_u = A$, $A_e = C$, $\alpha_e = \textsf{id}_C$ and $\omega_e = \psi$.

We now show that \eqref{itm:3'} implies \eqref{itm:1'}. By symmetry, we may assume that $C = A$. Let $\mathcal{F}$ be the collection of finite subgroups of $A$. Since $A$ is virtually $\Z$, the pair $(A, \mathcal{F})$ has the \sfgip\ by Proposition \ref{prop: hyperbolic_sfgip_2}. Moreover, there is a bound on the order of finite subgroups of $A$. Indeed, by definition, $A$ admits a subgroup $H$ isomorphic to $\ZZ$, such that $H$ has index $d < \infty$ in $A$. If $B$ is a subgroup of $A$, then $B \cap H$ has index at most $d$ in $B$; thus, if $B$ is finite, then $B \cap H$ is trivial and hence $|B| \le d$. Since every subgroup of $A$ is finitely generated, every finitely generated subgroup of $G$ intersects each conjugate of $A$ in a finitely generated subgroup. Thus, all the hypotheses of Theorem \ref{fgip_criterion_3} are satisfied and we may conclude that $G$ has the \fgip

Finally, we prove (the contrapositive of the fact) that \eqref{itm:2'} implies \eqref{itm:3'}. Suppose that $C$ and $D = C^t$ are proper subgroups of $A$. We note that $A \cap A^t = D$, and $\langle A, A^t\rangle$ is isomorphic to the amalgamated product $A *_D A^t$. If $D$ has index at least three in $A$ and at least 2 in $A^t$ (or vice versa), then $\langle A, A^t\rangle$ (and hence $G$) contains a subgroup isomorphic to $F_2\times \Z$ by Proposition~\ref{prop: vZ_amalgam}. Thus we are reduced to considering the case where $D$ has index two in $A$ and in $A^t$. Conjugation by $t\inv$ shows that $C = D^{t\inv}$ has index two in $A$ as well.

Consider the graph of groups $\BB$
    \[
    \begin{tikzcd}
    A  \arrow[r, "f'", "C"'] & A  \arrow[r, "f", "C"'] & A
    \end{tikzcd}
    \]
rooted at the middle vertex $v_0$, with $\alpha_f = \alpha_{f'} = \textsf{id}_C$ and $\omega_f = \omega_{f'} = \psi$. Then $\BB$ admits a morphism $\mu$ to $\AA$, which maps both $f$ and $f'$ to $e$, is the identity on each vertex and edge group, and has trivial twisting elements. It is directly verified that $\pi_1(\BB,v_0)$ is generated by $A$, $(1,f,1)\,A\,(1,f\inv,1)$ and $(1,{f'}\inv,1)\,A\,(1,f',1)$, so that its image $\mu_*(\pi_1(\BB,v_0))$ is $\langle A, A^t, A^{t\inv}\rangle$. By Proposition~\ref{prop: vZ_amalgam2}, $\pi_1(\BB,v_0)$ contains a copy of $\Z\times F_2$, and since $\mu_*$ is injective, so do $\langle A, A^t, A^{t\inv}\rangle$ and $G$.
\end{proof}

\begin{prop}
\label{prop: vZ_HNN2}
    If $G$ is the fundamental group of the 2-vertex graph of virtually $\Z$ groups below,
    \[
    \begin{tikzcd}
    A_u \arrow[r, "2" near start, "e'", "2" near end, bend left=49] \arrow[r, "2"' near start, "e"', "2"' near end, bend right=49] & A_{u'}
    \end{tikzcd}
    \]
    where the integers indicate the index of the edge groups in the adjacent vertex groups, then $G$ contains a copy of $F_2\times\Z$.
\end{prop}

\begin{proof}
Let $\BB$ be the following graph of groups
    \[
    \begin{tikzcd}[column sep = large]
    B_v \arrow[r, "2" ' near start, "f" , "2" ' near end] & B_{v'}  & B_{v''}  \arrow[l, "2" near start, "f'" ', "2" near end]
    \end{tikzcd}
    \]
where $B_v = B_{v''} = A_u$, $B_{v'} = A_{u'}$, $B_f = A_e$, $B_{f'} = A_{e'}$, and the edge maps $\alpha_f$, $\alpha_{f'}$, $\omega_f$ and $\omega_{f'}$ are as in $\AA$.

Let $\mu\colon \BB \to \AA$ be the morphism which maps $v$ and $v''$ to $u$, $v'$ to $u'$, $f$ to $e$ and $f'$ to $e'$, with $\mu_v$, $\mu_{v'}$, $\mu_{v''}$, $\mu_f$ and $\mu_{f'}$ the identity morphism on the appropriate group, and with trivial twisting elements.

Then $\mu_*\colon \pi_1(\BB,v) \to G = \pi_1(\AA,u)$ is injective. Proposition~\ref{prop: vZ_amalgam2} shows that $\pi_1(\BB,v)$ contains a copy of $F_2\times \Z$, and hence so does $G$.
\end{proof}

\begin{prop}
\label{prop: double_HNN}
    If $G$ is a double HNN extension of virtually $\Z$ groups of the form:
    \[
    \begin{tikzcd}
    A \arrow["1"' near start, "e"', "m"' near end, loop, distance=2em, in=215, out=145] \arrow["1"' near start, "e'"', "n"' near end, loop, distance=2em, in=35, out=325]
    \end{tikzcd}
    \]
    where $m, n\geqslant 1$ and the integers indicate the index of the image of the edge groups in the corresponding vertex groups, then $G$ contains a copy of $F_2\times \Z$.
\end{prop}

\begin{proof}
We distinguish three cases, where $m = n = 1$, $m,n \ne 1$ and exactly one of $m, n$ equals 1.

    Let us first suppose that $m = n = 1$, that is, $\omega_e$ and $\omega_{e'}$ are both isomorphisms. Let $a\in A$ be an element generating an infinite injectively characteristic subgroup, which exists by Lemma \ref{fully_char}. Then $\omega_e\circ\alpha_e^{-1}$ and $\omega_{e'}\circ\alpha_{e'}^{-1}$ induce automorphisms of $\langle a\rangle$. In particular, their squares fix $a$ itself. In other words, if $p = (1, e, 1, e, 1)$ and $q = (1, e', 1, e', 1)$, then $a^p = a$ and $a^q = a$. Now consider the morphism $\chi\colon F(x,y) \to \langle p,q\rangle$ which maps $x$ to $p$ and $y$ to $q$. It is immediate that a reduced word in $F(x,y)$ maps to an $\AA$-path which is already reduced, and hence $\chi$ is an isomorphism. As in the previous lemmas, Lemma~\ref{lm: semidirect product} then shows that $\langle a,p,q\rangle$ is isomorphic to $\Z\times F_2$.
    
Now suppose that $m,n\ne 1$, so that neither $\omega_e$ nor $\omega_{e'}$ is an isomorphism. Let $\BB$ be the following graph of groups
\[
\begin{tikzcd}
B_v \arrow[r, "1" near start, "f'", "m" near end, bend left=49] \arrow[r, "1"' near start, "f"', "n"' near end, bend right=49] & B_{v'}
\end{tikzcd}
\]
where $B_v = B_{v'} = A$, $B_f = A_e$, $B_{f'} = A_{e'}$ and the edge maps are as in $\AA$. Let $\mu\colon \BB\to \AA$ be the morphism which maps $v$ and $v'$ to $u$, $f$ to $e$ and $f'$ to $e'$ with $\mu_v, \mu_{v'}, \mu_f, \mu_{f'}$ the identity morphism on the appropriate group, and with trivial twisting elements. Then $\mu_*\colon \pi_1(\BB, v)\to \pi_1(\AA, u)$ is injective. Now let $C = \omega_e(A_e)$, $D = \omega_{e'}(A_{e'})$ and $\psi = \omega_{e'} \circ \alpha_{e'}\inv \circ \alpha_e \circ \omega_e\inv \colon C \to D$. Then $A\ast_\psi$ contains a copy of $\Z\times F_2$ by Lemma~\ref{prop: vZ_HNN}. Since $A\ast_{\psi}\cong \pi_1(\BB, v)$, it follows that $\pi_1(\AA, u)$ contains a copy of $F_2\times \Z$.

Finally suppose that exactly one of $m$ and $n$ is equal to 1, say, $\omega_{e'}$ is an isomorphism and $\omega_e$ is not. Let $\BB$ be the following graph of groups
\[
\begin{tikzcd}
B_v \arrow[r, "1"' near start, "f'", "1"' near end] \arrow["1"' near start, "f", "m"' near end, loop, distance=2em, in=125, out=55] & B_{v'} \arrow["1"' near start, "f''", "m"' near end, loop, distance=2em, in=125, out=55]
\end{tikzcd}
\]
where $B_v = B_{v'} = A$, $B_f = B_{f''} = B_e$, $B_{f'} = A_{e'}$ and the edge maps are as in $\AA$. Let $\mu\colon \BB\to \AA$ be the morphism which maps $v$ and $v'$ to $u$, $f, f''$ to $e$, $f'$ to $e'$ with $\mu_v, \mu_{v'}, \mu_f, \mu_{f'}, \mu_{f''}$ identity morphisms on the appropriate group, and with trivial twisting elements. Since $\pi_1(\BB, v)$ is isomorphic to the fundamental group of the graph of groups obtained from $\BB$ by collapsing $f'$ to a point, we are reduced to the case in which $m, n\neq 1$ which we already handled.
\end{proof}

\begin{prop}
\label{prop: vZ_amalgam_HNN}
    If $G$ is the fundamental group of the graph of virtually $\Z$ groups below,
    \[
    \begin{tikzcd}
    A \arrow[r, "2" near start, "e", "2" near end] & B \arrow["1"' near start, "f"', "k"' near end, loop, distance=2em, in=35, out=325]
    \end{tikzcd}
    \]
    where $k\geqslant 1$ and the integers indicate the index of the image of each edge group in the corresponding vertex group, then $G$ contains a copy of $F_2\times \Z$.
\end{prop}

\begin{proof}
Let $c$ be in the complement of $\alpha_e(A_e)$ in $A$ and let $d$ be in the complement of $\omega_e(A_e)$ in $B$.

By Lemma \ref{fully_char}, there exist elements $a_0\in A_e$ and $a_1\in B$, which generate infinite, injectively characteristic subgroups of $A_e$ and $B$, respectively.  Then the intersection $\langle \omega_e(a_0)\rangle \cap \langle a_1\rangle$ is infinite and injectively characteristic in $B$ (see Fact~\ref{fact: elementary virtually Z} and Lemma~\ref{lm: elementary virtually Z}). Let $a_2$ be a generator of that intersection and let $a = \alpha_e(\omega_e\inv(a_2)) \in A$.

We note the following:
\begin{itemize}
\item $\langle a\rangle$ is injectively characteristic in $A$, $\langle a_2\rangle$ is injectively characteristic in $B$, and hence $a^c = a^\epsilon$ and $a_2^d = a_2^\eta$ for some $\epsilon, \eta \in \{1,-1\}$.

\item $\omega_f\circ\alpha_f\inv$ is an isomorphism from $B$ to $\omega_f(A_f)$; let $p$ be the index of $\langle a_2\rangle$ in $B$. Then $\langle \omega_f(\alpha_f\inv(a_2))\rangle$ has index $p$ in $\omega_f(A_f)$, and hence index $pk$ in $B$. Moreover, there exists $m\ne 0$ such that $\omega_f(\alpha_f\inv(a_2)) = a^m$: then $\langle \omega_f(\alpha_f\inv(a_2))\rangle$ has index $m$ in $\langle a_2\rangle$ and hence index $p|m|$ in $B$. It follows that $|m| = k$, that is, $m = \zeta k$ for some $\zeta = \pm1$.
\end{itemize}

Let $t_e = (1,e,1)$ and $t_f = (1,f,1)$. Then $t_e\inv at_e = a_2$, $t_f\inv a_2t_f = a_2^{\zeta k}$. Let now
\begin{align*}
        x &= t_e\,(d)\,t_f\,t_e\inv\, (c)\, t_e\, t_f\inv\,(d)\, t_e\inv\\
        y &= (c)\, t_e\,t_f\,t_e\inv\, (c)\, t_e\,(d)\, t_f\inv\, t_e\inv.
\end{align*}
First, we observe as in earlier proofs that $\langle x,y\rangle \le A$ is freely generated by $x$ and $y$, and hence is isomorphic to $F_2$.

Next it is directly verified that $a^x = x\inv ax = a^{\zeta\epsilon}$, and that $a^y = a^{\eta\epsilon}$, so that both are in $\{a,a\inv\}$. Therefore $a^{x^2} = a^{y^2} = a$. It follows, by Lemma~\ref{lm: semidirect product}, that $\langle a,x^2,y^2\rangle$ is isomorphic to $\Z\times F_2$, which concludes the proof.
\end{proof}

We can finally prove Theorem~\ref{thm: vcyclic_fgip}.

\begin{proof}[Proof of Theorem~\ref{thm: vcyclic_fgip}]
Proposition~\ref{prop: finite generation}~\eqref{item: finite_gen_implies_finite_graph} and Theorem~\ref{thm: bijection subgroups immersions covers} show that, when considering the intersection of two given finitely generated subgroups of $\pi_1(\AA, u)$, only a finite fragment of $\gr A$ needs to be considered: thus we may assume that $\AA$ is finite. Proposition~\ref{prop: reduced gog} then shows that we can assume $\AA$ to be reduced.

Let $e$ be an edge of $\AA$: depending on whether $e$ is a loop, we are in one of the two configurations in Figure~\ref{fig: single edge}, where $k, \ell, m, n$ are the indices of $\alpha_e(A_e)$ and $\omega_e(A_e)$ in the adjacent vertex groups.

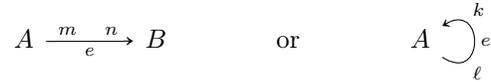
\begin{figure}[htb]
\centering
    \begin{tikzcd}
A \arrow[r, "m" near start, "e" ', "n" near end] & B & \textrm{or} & A  \arrow["\ell"' near start, "e"', "k"' near end, loop, distance=2em, in=35, out=325]
    \end{tikzcd}
\caption{A single edge in $\AA$}
\label{fig: single edge}
\end{figure}

If $e$ is not a loop and the indices $m$ and $n$ do not satisfy $m = 1$, or $n = 1$, or $m = n =2$, then Proposition~\ref{prop: vZ_amalgam} shows that $\pi_1(\AA,u)$ contains a copy of $F_2\times \Z$ and hence does not have the \fgip

Similarly, if $e$ is a loop and if $k, \ell \ne 1$, then Proposition~\ref{prop: vZ_HNN} shows that $\pi_1(\AA,u)$ contains a copy of $F_2\times \Z$ and hence does not have the \fgip

Therefore, if $\pi_1(\AA,u)$ has the \fgip, then every non-loop edge $e$ of $\AA$ is such that the edge group $A_e$ has index 2 in both adjacent vertex groups, and every loop $e$ is such that $\alpha_e$ or $\omega_e$ is an isomorphism. If $\AA$ contains at least two edges (up to orientation), then $\AA$ contains a subgraph of groups covered by one of Propositions~\ref{prop: vZ_amalgam2}, \ref{prop: vZ_HNN2}, \ref{prop: double_HNN}  and \ref{prop: vZ_amalgam_HNN}. It follows that $\pi_1(\AA, u)$ contains a copy of $F_2\times \Z$, and hence it does not have the \fgip

Thus $\AA$ has a single edge, and we conclude by Propositions~\ref{prop: vZ_amalgam} and~\ref{prop: vZ_HNN}.
\end{proof}

\section{The \fgip\ for graphs of locally quasi-convex hyperbolic groups}\label{sec: hyperbolic}

In this final section, our aim will be to establish the main application of our \fgip\ criteria: a characterisation of the fundamental groups of locally quasi-convex hyperbolic groups with virtually $\ZZ$ edge groups which have the \fgip, Corollary~\ref{cor: main locally qc characterization}, and a proof of the decidability of this property, Corollary~\ref{cor: decidability for locally qcv hyperbolic}. In order to prove this result, we begin in Section \ref{sec: fcip for qcv subgroups} by attempting to understand the \fcip\ property for subgroups of hyperbolic groups. The main result in Section \ref{sec: fcip for qcv subgroups} states that pairs $(B, C)$ of quasi-convex subgroups of hyperbolic groups have finite coset interaction relative to any almost malnormal quasi-convex subgroup. In Section \ref{sec: hyperbolic_new} we use this, together with well-known facts about hyperbolic groups, to obtain the announced results. We shall also combine the main result of Section \ref{sec: fcip for qcv subgroups} with our \fgip\ criteria to prove the weaker, but more general, Theorems~\ref{hyperbolic_amalgam} and~\ref{hyperbolic_HNN}.

\subsection{The \fcip\ for quasi-convex subgroups of hyperbolic groups}\label{sec: fcip for qcv subgroups}

The aim of this section is to prove that quasi-convex subgroups of hyperbolic groups have the \fcip\ relative to almost malnormal quasi-convex subgroups. We refer the reader to \cite{gh90,bh99} for the basic definitions and results regarding hyperbolic groups. Let us first set up some notation.

Let $G$ be a group, $S$ be a finite generating set of $G$ and $\Gamma = \Gamma(G, S)$ be the corresponding Cayley graph. If $g\in G$, we denote by $|g|_S$ the word length (also: the \emph{$S$-length}) of $g\in G$, that is, $|g|_S = \min\{n \mid g = s_1\ldots s_n, s_i\in S\}$. For each $g\in G$ with length $n$, we fix an arbitrary geodesic path $p_g$ in $\Gamma$, starting at vertex $1$ and ending at vertex $g$. We say that a factorisation $g = g_1g_2\dots g_n$ is ($S$-)\emph{geodesic} if $|g|_S = \sum_i|g_i|_S$.

Let $A$ be a subgroup of $G$, generated by a finite set $S_A$, and let $a \in A$. If $|a|_{S_A} = n$, we fix an $S_A$-geodesic factorisation $a = a_1\cdots a_n$, with $a_1,\dots, a_n\in S_A$ and we denote by $q_a$
the concatenation of the paths $p_{a_1}, a_1\cdot p_{a_2}, \dots, (a_1\dots a_{n-1})\cdot p_{a_n}$ in $\Gamma$, a path from 1 to $a$. We will use the following facts.

\begin{prop}\label{prop: standard hyperbolic}
Let $G$ be a group and $S$ a finite generating set for $G$. Suppose that $G$ is hyperbolic with hyperbolicity constant $\delta$ (relative to $S$). Let $A$ be a quasi-convex subgroup of $G$ with finite generating set $S_A$.
\begin{enumerate}[(1)]
\item \label{item: lambda}There exists $\lambda \ge 1$ such that $A$ is $\lambda$-quasi-convex and, for each $a\in A$, $q_a$ is a $\lambda$-quasi-geodesic in $\Gamma$. In particular
$$\frac 1\lambda\,|a|_S - \lambda \le |a|_{S_A} \le \lambda\,|a|_S + \lambda,\quad\text{and hence}\quad\frac 1\lambda\,(|a|_{S_A} - \lambda) \le |a|_S \le \lambda\,(|a|_{S_A} + \lambda).$$
\item \label{item: Morse}(Generalised Morse Lemma). Let $\lambda\ge 1$ and $n\ge 2$. There exists $\lambda'_n$ (depending on $|S|$, $n$, $\lambda$ and $\delta$) such that, if $(q_1,\dots,q_n)$ is an $n$-gon composed of $\lambda$-quasi-geodesics, then each point of a side $q_i$ is within distance $\lambda'_n$ from the union of the other $n-1$ sides.
\end{enumerate}
\end{prop}

Before proving the main results of this section, we require a couple of lemmas.

\begin{lem}
\label{finding_intersection}
Let $G$ be a finitely generated hyperbolic group and $S$ a finite generating set. Let $B$ and $C$ be quasi-convex subgroups of $G$, with finite generating sets $S_B$ and $S_C$, respectively, and let $f, g \in G$. There exists a constant $L$ with the following property: if $b\in B$, $c\in C$ and $fb = cg$, then either $|b|_{S_B}, |c|_{S_C} \le L$, or there exist $h\in G$ and geodesic factorisations $b = b_1b_2b_3$ and $c = c_1c_2c_3$, in $B$ and $C$ respectively, such that $|h|_S\leqslant L$, $b_2, c_2 \ne 1$ and one of the following statement holds (see Figure~\ref{fig: straight h crossed h}):
\begin{enumerate}[(a)]
\item $fb_1h = c_1$, $b_2^h = c_2$ and $h\inv b_3 = c_3g$

\item $fb_1h = c_1c_2$, $b_2^h = c_2\inv$ and $h\inv b_3 = c_2c_3g$.
\end{enumerate}
The constant $L$ depends on $|S|$, $\max(|f|_S,|g|_S)$, a hyperbolicity constant for $G$ and a quasi-convexity constant for $B$ and $C$.
\end{lem}

\begin{figure}[H]  
\centering
\begin{tikzpicture}[shorten >=1pt, node distance=1cm and 2cm, on grid, decoration={snake, segment length=2mm, amplitude=0.2mm,post length=1.5mm},>=stealth']
\newcommand{\dx}{1/2}
\newcommand{\dy}{1.5}
    \node[] (1) at (0, 0) {$1$};
    \node[] (f) [above = \dy of 1] {$f$};
    \node[smallstate] (b1) [right = 3*\dx of f] {};
    \node[smallstate] (b1b2) [right = 3*\dx of b1] {};

    \node[] (fb) [right = 3*\dx of b1b2] {$fb$};
    \node[] (cg) [below right = 0.05 and 0.65 of fb] {$= cg$};
    \node[smallstate] (c1) [right = 3*\dx of 1] {};
    \node[smallstate] (c1c2) [right = 3*\dx of c1] {};

    \node[] (c) [right = 3*\dx of c1c2] {$c$};
    
    \path[->] (1) edge[snake it] node[pos=0.5,left] {$f$} (f);
    \path[->] (c) edge[snake it] node[pos=0.5,right] {$g$} (fb);
   
    \path[->] (f) edge[snake it] node[pos=0.5,above] {$b_1$} (b1);
    \path[->] (b1) edge[snake it] node[pos=0.5,above] {$b_2$} (b1b2);
    \path[->] (b1b2) edge[snake it] node[pos=0.5,above] {$b_3$} (fb);
    
    \path[->] (1) edge[snake it] node[pos=0.5,below] {$c_1$} (c1);
    \path[->] (c1) edge[snake it] node[pos=0.5,below] {$c_2$} (c1c2);
    \path[->] (c1c2) edge[snake it] node[pos=0.5,below] {$c_3$} (c);
    
    \path[->] (b1) edge[snake it] node[pos=0.5,left] {$h$} (c1);
    \path[->] (b1b2) edge[snake it] node[pos=0.5,right] {$h$} (c1c2);
\end{tikzpicture}
\begin{tikzpicture}[shorten >=1pt, node distance=1cm and 2cm, on grid, decoration={snake, segment length=2mm, amplitude=0.2mm,post length=1.5mm},>=stealth']
\newcommand{\dx}{1/2}
\newcommand{\dy}{1.5}
    \node[] (1) at (0, 0) {$1$};
    \node[] (f) [above = \dy of 1] {$f$};
    \node[smallstate] (b1) [right = 3*\dx of f] {};
    \node[smallstate] (b1b2) [right = 3*\dx of b1] {};

    \node[] (fb) [right = 3*\dx of b1b2] {$fb$};
    \node[] (cg) [below right = 0.05 and 0.65 of fb] {$= cg$};
    \node[smallstate] (c1) [right = 3*\dx of 1] {};
    \node[smallstate] (c1c2) [right = 3*\dx of c1] {};

    \node[] (c) [right = 3*\dx of c1c2] {$c$};
    
    \path[->] (1) edge[snake it] node[pos=0.5,left] {$f$} (f);
    \path[->] (c) edge[snake it] node[pos=0.5,right] {$g$} (fb);
   
    \path[->] (f) edge[snake it] node[pos=0.5,above] {$b_1$} (b1);
    \path[->] (b1) edge[snake it] node[pos=0.5,above] {$b_2$} (b1b2);
    \path[->] (b1b2) edge[snake it] node[pos=0.5,above] {$b_3$} (fb);
    
    \path[->] (1) edge[snake it] node[pos=0.5,below] {$c_1$} (c1);
    \path[->] (c1) edge[snake it] node[pos=0.5,below] {$c_2$} (c1c2);
    \path[->] (c1c2) edge[snake it] node[pos=0.5,below] {$c_3$} (c);
    
    \path[->] (b1) edge[snake it] node[pos=0.2,left] {$h$} (c1c2);
    \path[->] (b1b2) edge[snake it] node[pos=0.7,left] {$h$} (c1);
\end{tikzpicture}
\caption{Case (a) and Case (b) factorisations in Lemma~\ref{finding_intersection}}
    \label{fig: straight h crossed h}
\end{figure}
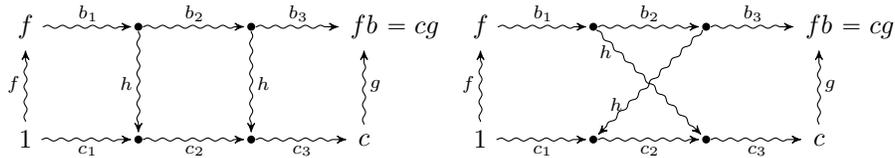

\begin{proof}
Let $\delta > 0$ be a hyperbolicity constant for $G$, relative to $S$. Let $\lambda\geqslant 1$ be given as in Proposition~\ref{prop: standard hyperbolic}\eqref{item: lambda}, simultaneously for the quasi-convex subgroups $B$ and $C$, and let $\lambda'$ be the constant $\lambda'_4$ from Proposition~\ref{prop: standard hyperbolic}\eqref{item: Morse} (which depends on $|S|$, $\lambda$ and $\delta$).

Let also 
$$\lambda'' = \lambda' + \max(|f|_S, |g|_S) + \frac{\lambda(\lambda+1)}2$$
and $L$ the number of elements in a ball of radius $\lambda''$ in $\Gamma$, where $\Gamma$ is the Cayley graph of $G$ relative to $S$.

Suppose that $n = |b|_{S_B}$, $m = |c|_{S_C}$, and let $b = x_1\dots x_n$ and $c = y_1\dots y_m$ ($x_i\in S_B$, $y_j\in S_C$) are the fixed $S_B$- and $S_C$-factorisations of $b$ and $c$, so that $q_b$ (resp. $q_c$) is the concatenation of the appropriate translates of the $S$-geodesics $p_{x_i}$ (resp. $p_{y_j}$).

Consider the quasi-geodesic 4-gon with sides $p_f$, $f\cdot q_b$, $q_c$ and $c\cdot p_g$, see Figure~\ref{fig: rectangle}. For each $i \in [1,n-1]$, the vertex $f\cdot (x_1\dots x_i)$ on $f\cdot q_b$ lies at $S$-distance at most $\lambda'$ from a vertex in $p_f$, $q_c$ or $c\cdot p_g$, and therefore at $S$-distance at most $\lambda' + \max(|f|_S, |g|_S)$ from a vertex $z$ in $q_c$. Note now that $|y|_S \le \lambda(1+\lambda)$ for each $y\in S_C$. Thus there exists $j(i) \in [0,m]$ such that $z$ lies within $S$-distance at most $\frac12\,\lambda(1+\lambda)$ from the vertex $y_1\dots y_{j(i)}$. In particular, if $h_i = (fx_1\cdots x_i)\inv y_1\dots y_{j(i)}$, then $|h_i|_S \le \lambda''$.  
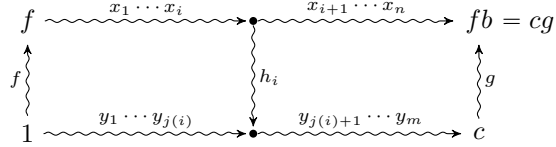
\begin{figure}[H]  
\centering
\begin{tikzpicture}[shorten >=1pt, node distance=1cm and 2cm, on grid, decoration={snake, segment length=2mm, amplitude=0.2mm,post length=1.5mm},>=stealth']
\newcommand{\dx}{1/2}
\newcommand{\dy}{1.5}
    \node[] (1) at (0, 0) {$1$};
    \node[] (f) [above = \dy of 1] {$f$};
    \node[smallstate] (bi) [right = 6*\dx of f] {};

    \node[] (fb) [right = 6*\dx of bi] {$fb$};
    \node[] (cg) [below right = 0.05 and 0.65 of fb] {$= cg$};

    \node[smallstate] (cj) [right = 6*\dx of 1] {};
    
    \node[] (c) [right = 6*\dx of cj] {$c$};   

    \path[->] (1) edge[snake it]
    node[pos=0.5,left] {$f$}
    (f);
   
    \path[->] (f) edge[snake it] 
    node[pos=0.5,above] {$x_1\cdots x_i$}
    (bi);
    
    \path[->] (bi) edge[snake it] 
    node[pos=0.5,above] {$x_{i+1}\cdots x_n$}
    (fb);
   
    \path[->] (1) edge[snake it]
    node[pos=0.5,above] {$y_1 \cdots y_{j(i)}$}
    (cj);
    
    \path[->] (cj) edge[snake it] 
    node[pos=0.5,above] {$ y_{j(i)+1} \cdots y_m$}
    (c);
  
    \path[->] (c) edge[snake it] 
    node[pos=0.5,right] {$g$}
    (fb);

    \path[->] (bi) edge[snake it] 
    node[pos=0.5,right] {$h_i$}
    (cj);
          
\end{tikzpicture}
\caption{A 4-gon between 1, $f$, $c$ and $fb = cg$}
    \label{fig: rectangle}
\end{figure}

If $|b|_{S_B} > L$, there exist $1 \le i < k < n$ such that $h_i = h_j$: let $h$ be this common value. Let also $b_1 = x_1\cdots x_i$, $b_2 = b_{i+1}\dots b_j$ and $b_3 = b_{j+1}\dots b_n$, so that $b = b_1b_2b_3$ is a geodesic factorisation. In particular, $b_2 \ne 1$.

Note that $j(i) \ne j(k)$. If $j(i) < j(k)$, we let $c_1 = y_1\dots y_{j(i)}$, $c_2 = y_{j(i)+1}\dots y_{j(k)}$ and $c_3 = y_{j(k)+1}\dots y_m$. Then $c = c_1c_2c_3$ is a geodesic factorisation in $C$, $c_2 \ne 1$ and we have $fb_1h = c_1$, $b_2^h = c_2$ and $h\inv b_3 = c_3g$.

If instead $j(k) < j(i)$, we let $c_1 = y_1\dots y_{j(k)}$, $c_2 = y_{j(k)+1}\dots y_{j(i)}$ and $c_3 = y_{j(i)+1}\dots y_m$. Again, $c = c_1c_2c_3$ is a geodesic factorisation in $C$, $c_2 \ne 1$ and we have $fb_1h = c_1c_2, b_2^h = c_2\inv$ and $h\inv b_3 = c_2c_3g$.

The case where $|c|_{S_C} > L$ is handled in the same fashion, and this completes the proof.
\end{proof}

\begin{lem}
\label{hyperbolic_lemma}
Let $G$ be a finitely generated hyperbolic group and $S$ a finite generating set. Let $A$, $A'$, $B$ and $C$ be quasi-convex subgroups of $G$, with finite generating sets $S_A$, $S_{A'}$, $S_B$ and $S_C$, respectively, and let $f, g, f', g'\in G$. There exists a constant $L$ such that the following holds. If $a\in A$, $a'\in A'$ are such that 
    \begin{align*}
        |a|_{S_A} &= \min\left\{ |x|_{S_A} \mid x \in (A\cap B^f)\,a\,(A\cap C^g) \right\}\\
        |a'|_{S_{A'}} &= \min\left\{ |x|_{S_{A'}} \mid x \in (A'\cap B^{f'})\,a'\,(A'\cap C^{g'}) \right\},
    \end{align*}
and if $b\in B$ and $c\in C$ satisfy $b\,(fag\inv) = (f'a'(g')\inv)\,c$, then either $|b|_{S_B}, |c|_{S_C}\leqslant L$, or $|a|_{S_A}, |a'|_{S_{A'}}\leqslant L$.

The constant $L$ depends on $|S|$, $\max(|f|_S,|g|_S,|f'|_S,|g'|_S)$, a hyperbolicity constant for $G$ and a quasi-convexity constant for $A$, $A'$, $B$ and $C$.
\end{lem}

\begin{proof}
Let $\delta > 0$ be a hyperbolicity constant for $G$, relative to $S$. Let $\lambda\geqslant 1$ be given as in Proposition~\ref{prop: standard hyperbolic}\eqref{item: lambda}, simultaneously for the quasi-convex subgroups $A$, $A'$, $B$ and $C$, and let $\lambda' = \max(\lambda'_6,\lambda'_8)$, where $\lambda'_6$ and $\lambda'_8$ are given in Proposition~\ref{prop: standard hyperbolic}\eqref{item: Morse}.

Let 
$$\lambda'' = \lambda' + \max(|f|_S, |f'|_S, |g|_S, |g'|_S) + \frac12\,\lambda(\lambda+1)$$
and let $L$ be the constant given by Lemma~\ref{finding_intersection} for the parameters $\max(|f|_S,|g|_S,|f'|_S,|g'|_S, 2\lambda'')$, $|S|$, $\delta$ and $\lambda$. We will use the following claim.

\begin{claim}\label{claim: short b0}
Let $h_0\in G$, $a_0$ a prefix of an $S_A$-geodesic factorisation of $a$ and $b_0$ a suffix of an $S_B$-geodesic factorisation of $b$, such that $b_0fa_0 = h_0$. If $|h_0|_S \le 2\lambda''$, then $|b_0|_{S_B} \le L$, see Figure~\ref{fig: shortcut}.
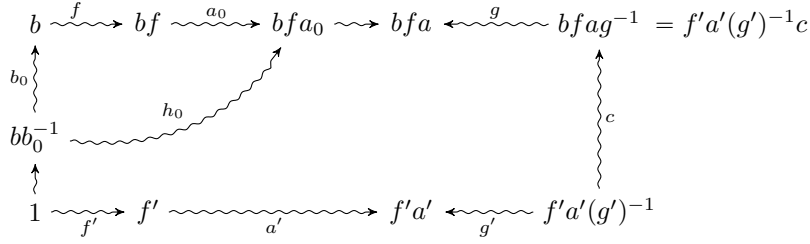
\begin{figure}[H]
\centering
\begin{tikzpicture}[shorten >=1pt, node distance=1cm and 2cm, on grid, decoration={snake, segment length=2mm, amplitude=0.2mm,post length=1.5mm},>=stealth']
\newcommand{\dx}{1}
\newcommand{\dy}{1}
    \node[] (1) at (0, 0) {$1$};
    \node[] (bb0i) [above = \dy of 1] {$bb_0^{-1}$};
    \node[] (b) [above = \dy*1.5 of bb0i] {$b$};
    \node[] (bf) [right = \dx*1.5 of b] {$bf$};

    \node[] (bfa0) [right = \dx*2 of bf] {$bfa_0$};
    \node[] (bfa) [right = \dx*1.5 of bfa0] {$bfa$};
    \node[] (bfagi) [right = \dx*2.5 of bfa] {$bfag^{-1}$};
    \node[] (ffaaggic) [right = 1.75 of bfagi] {$= f'a'(g')^{-1}c$};

    \node[] (ff) [right = \dx*1.5 of 1] {$f'$}; 
    \node[] (ffaa) [right = \dx*3.5 of ff] {$f' a'$};
    \node[] (ffaaggi) [right = \dx*2.5 of ffaa] {$f'a'(g')^{-1}$};
    
    \path[->] (1) edge[snake it]
    node[pos=0.5,left] {$$}
    (bb0i);

    \path[->] (bb0i) edge[snake it]
    node[pos=0.5,left] {$b_0$}
    (b);

    \path[->] (b) edge[snake it]
    node[pos=0.5,above left] {$f$}
    (bf);

    \path[->] (bf) edge[snake it]
    node[pos=0.5,above] {$a_0$}
    (bfa0);

    \path[->] (bfa0) edge[snake it]
    node[pos=0.5,above] {$$}
    (bfa);

    \path[->] (bfagi) edge[snake it]
    node[pos=0.5,above] {$g$}
    (bfa);

    \path[->] (1) edge[snake it]
    node[pos=0.5,below] {$f'$}
    (ff);

    \path[->] (ff) edge[snake it]
    node[pos=0.5,below] {$a'$}
    (ffaa);

   \path[->] (ffaaggi) edge[snake it]
    node[pos=0.5,below] {$g'$}
    (ffaa);

    \path[->] (ffaaggi) edge[snake it]
    node[pos=0.5, right] {$c$}
    (bfagi);

    \path[->] (bb0i) edge[snake it,bend right]
    node[pos=0.5, above left] {$h_0$}
    (bfa0);

\end{tikzpicture}
\caption{The factorisation in Claim~\ref{claim: short b0}}
    \label{fig: shortcut}
\end{figure}
\end{claim}

\begin{proof}
Suppose that $|b_0|_{S_B} > L$. Let $a'_0$ be such that $a = a_0a'_0$ is a geodesic factorisation in $A$.
By Lemma~\ref{finding_intersection}, there exist $h \in G$ and geodesic factorisations $a_0 = a_1a_2a_3$ and $b_0\inv = b_1b_2b_3$ in $A$ and $B$, respectively, such that $|h|_S\le L$, $x_2\ne 1$ and
\begin{enumerate}[(a)]
\item $fa_1h = b_1$, $h\inv a_3 = b_3h_0$ and $h\inv a_2h = b_2$, or

\item $fa_1a_2h = b_1$, $h\inv a_3 = b_2b_3h_0$ and $h\inv a_2h = b_2\inv$.
\end{enumerate}
If (a) holds, then $f\inv b_1b_2b_1\inv f = (a_1h)\,(h\inv a_2h)\,(h\inv a_1\inv) = a_1a_2a_1\inv$, and hence
$a_1a_2a_1\inv \in A \cap B^f$. It follows that $a_1a_3a'_0 = a_1\,((a_1a_2)\inv a_0)\,a'_0 = (a_1a_2a_1\inv)\inv a$ is a representative of $(A\cap B^f)a$, with shorter $S_A$-length than $a$, contradicting our hypothesis on $a$.

If instead (b) holds, then $f\inv b_1b_2b_1\inv f = (a_1a_2h)\,(h\inv a_2\inv h)\,(h\inv a_2\inv a_1\inv) = a_1 a_2\inv a_1\inv$, so that $a_1a_2\inv a_1\inv \in A \cap B^f$. Therefore $a_1a_3a'_0 = a_1\,((a_1a_2)\inv a_0)\,a'_0 = (a_1a_2\inv a_1\inv) a$ is a representative of $(A\cap B^f)a$, with shorter $S_A$-length than $a$, again a contradiction. This concludes the proof of the claim.
\end{proof}

Let $b\in B$ and $c\in C$ such that $b\,(fag\inv) = (f'a'{g'}\inv)\,c$. Consider the $\lambda$-quasi-geodesic 8-gon in Figure~\ref{fig: 8-gon}, with sides the appropriate translates of the quasi-geodesics $q_b, p_f, q_a, p_g, q_c, p_{g'}, q_{a'}, p_{f'}$ so that reading the labels clockwise yields the equation $bfag\inv c\inv g'{a'}\inv{f'}\inv = 1$.
\begin{figure}[H]
\centering
\begin{tikzpicture}[shorten >=1pt, node distance=1cm and 2cm, on grid, decoration={snake, segment length=2mm, amplitude=0.2mm,post length=1.5mm},>=stealth']
\newcommand{\dx}{1}
\newcommand{\dy}{1}
    \node[] (1) at (0, 0) {$1$};
    \node[] (b) [above = \dy*1.5 of 1] {$b$};
    \node[] (bf) [above right = \dy and \dx*1.5 of b] {$bf$};

    \node[] (bfa) [right = \dx*2 of bf] {$bfa$};
    \node[] (bfagi) [below right = \dy and \dx*1.5 of bfa] {$bfag^{-1}$};
    \node[] (ffaaggic) [right = 1.75 of bfagi] {$= f'a'(g')^{-1}c$};

    \node[] (ff) [below right = \dy and \dx*1.5 of 1] {$f'$}; 
    \node[] (ffaa) [right = \dx*2 of ff] {$f' a'$};
    \node[] (ffaaggi) [above right = \dy and \dx*1.5 of ffaa] {$f'a'(g')^{-1}$};

    \path[->] (1) edge[snake it]
    node[pos=0.5,left] {$b$}
    (b);

    \path[->] (b) edge[snake it]
    node[pos=0.5,above left] {$f$}
    (bf);

    \path[->] (bf) edge[snake it]
    node[pos=0.5,above] {$a$}
    (bfa);

    \path[->] (bfagi) edge[snake it]
    node[pos=0.5,above right] {$g$}
    (bfa);

    \path[->] (1) edge[snake it]
    node[pos=0.5,below left] {$f'$}
    (ff);

    \path[->] (ff) edge[snake it]
    node[pos=0.5,below] {$a'$}
    (ffaa);

   \path[->] (ffaaggi) edge[snake it]
    node[pos=0.5,below right] {$g'$}
    (ffaa);

    \path[->] (ffaaggi) edge[snake it]
    node[pos=0.5, right] {$c$}
    (bfagi);         
\end{tikzpicture}
\caption{The 8-gon in the proof of Lemma~\ref{hyperbolic_lemma}}
\label{fig: 8-gon}
\end{figure}
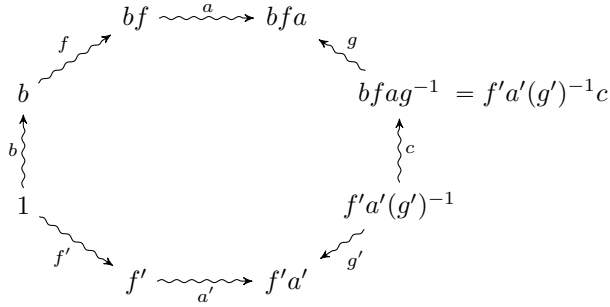

Suppose that $|b|_{S_B} > 2L+1$ and let $b = b_0b_1$ be a geodesic factorisation of $b$ in $B$ such that $|b_0|_{S_B} = \left\lfloor\frac{|b|_{S_B}}2\right\rfloor > L$. Then, reasoning as in the proof of Lemma~\ref{finding_intersection}, vertex $b_0$ is within $S$-distance $\lambda''$ from a vertex $w$ on the translate of $q_{a}$, $q_{a'}$ or $q_{c}$. Letting $h_0 = (b_0)\inv w$, we have $|h_0|_S \le \lambda''$ and one of the following conditions holds:
\begin{enumerate}[(1)]
\item there exists a geodesic factorisation $a = a_0a_1$ in $A$ such that $(b_1f)\,a_0 = h_0$;

\item there exists a geodesic factorisation $a' = a'_0a'_1$ in $A'$ such that $(b_0\inv f')\,a'_0 = h_0$;

\item there exists a geodesic factorisation $c = c_0c_1$ in $C$ such that $(b_0\inv f'a'{g'}\inv)\,c_0 = h_0$.
\end{enumerate}
Claim~\ref{claim: short b0} shows that, in Case (1), we have $|b_0|_{S_B} \le L$, a contradiction. Symmetrically, Case (2) also leads to a contradiction.

It follows that we must be in Case (3). We consider the ($\lambda$-quasi-geodesic) 6-gon with sides the appropriate translates of $q_{b_1}$, $p_f$, $q_a$, $p_g$, $q_{c_0}$ and $p_{h_0}$. Let $a = a_0a_1$ be the geodesic factorisation in $A$ such that $|a_0|_{S_A} = \left\lfloor\frac{|a|_{S_A}}2\right\rfloor$.
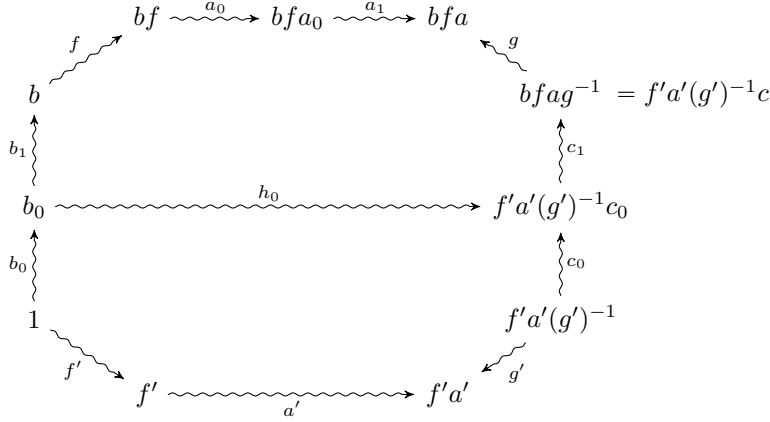
\begin{figure}[H]
\centering
\begin{tikzpicture}[shorten >=1pt, node distance=1cm and 2cm, on grid, decoration={snake, segment length=2mm, amplitude=0.2mm,post length=1.5mm},>=stealth']
\newcommand{\dx}{1}
\newcommand{\dy}{1}
    \node[] (1) at (0, 0) {$1$};
    \node[] (b0) [above = \dy*1.5 of 1] {$b_0$};
    \node[] (b) [above = \dy*1.5 of b0] {$b$};
    \node[] (bf) [above right = \dy and \dx*1.5 of b] {$bf$};

    \node[] (bfa0) [right = \dx*2 of bf] {$bfa_0$};
    \node[] (bfa) [right = \dx*2 of bfa0] {$bfa$};
    \node[] (bfagi) [below right = \dy and \dx*1.5 of bfa] {$bfag^{-1}$};
    \node[] (ffaaggic) [right = 1.75 of bfagi] {$= f'a'(g')^{-1}c$};

    \node[] (ff) [below right = \dy and \dx*1.5 of 1] {$f'$}; 
    \node[] (ffaa) [right = \dx*4 of ff] {$f' a'$};
    \node[] (ffaaggi) [above right = \dy and \dx*1.5 of ffaa] {$f'a'(g')^{-1}$};

    \node[] (ffaaggic0) [above = \dy*1.5 of ffaaggi] {$f'a'(g')^{-1}c_0$};

    \path[->] (1) edge[snake it]
    node[pos=0.5,left] {$b_0$}
    (b0);

    \path[->] (b0) edge[snake it]
    node[pos=0.5,left] {$b_1$}
    (b);

    \path[->] (b) edge[snake it]
    node[pos=0.5,above left] {$f$}
    (bf);

    \path[->] (bf) edge[snake it]
    node[pos=0.5,above] {$a_0$}
    (bfa0);

    \path[->] (bfa0) edge[snake it]
    node[pos=0.5,above] {$a_1$}
    (bfa);

    \path[->] (bfagi) edge[snake it]
    node[pos=0.5,above right] {$g$}
    (bfa);

    \path[->] (1) edge[snake it]
    node[pos=0.5,below left] {$f'$}
    (ff);

    \path[->] (ff) edge[snake it]
    node[pos=0.5,below] {$a'$}
    (ffaa);

   \path[->] (ffaaggi) edge[snake it]
    node[pos=0.5,below right] {$g'$}
    (ffaa);

    \path[->] (ffaaggi) edge[snake it]
    node[pos=0.5, right] {$c_0$}
    (ffaaggic0);

    \path[->] (ffaaggic0) edge[snake it]
    node[pos=0.5, right] {$c_1$}
    (bfagi);

    \path[->] (b0) edge[snake it]
    node[pos=0.5, above] {$h_0$}
    (ffaaggic0);

\end{tikzpicture}
\caption{The 6-gon in the proof of Lemma~\ref{hyperbolic_lemma}}
    \label{fig: 6-gon}
\end{figure}

Then $(bf)a_0$ is within $S$-distance $\lambda''$ from a vertex $w'$ in $B$ and $b_0\cdot q_{b_1}$, in $h_0C$ and $h_0\cdot q_{c_3}$ or in $p_{h_0}$.

If $w'$ is on $p_{h_0}$, then $a_0$ is within distance $2\lambda''$ from vertex $b_0$, that is, there exists $h_1\in G$ such that $|h_1|_S \le 2\lambda''$ and $b_1fa_0 = h_1$. Claim~\ref{claim: short b0} again shows that $|b_0|_{S_B} \le L$, a contradiction. Therefore $w'$ is on $b_0\cdot q_{b_1}$ (and hence in $q_b$) or in $h_0\cdot q_{c_3}$ (and hence in $(f'a'{g'}\inv)\cdot q_c$). A fresh application of Claim~\ref{claim: short b0} then shows that $|a_0|_{S_A} \le L$, and hence $|a|_{S_A}\le 2L$.

Thus, if $|b|_{S_B} > 2L+1$, then $|a|_{S_A}\le 2L$. We have proved, by symmetry, that either $|b|_{S_B}, |c|_{S_C} \le 2L+1$, or $|a|_{S_A}, |a'|_{S_{A'}} \le 2L+1$, as announced.
\end{proof}

\begin{lem}
Let $G$ be a hyperbolic group, let $A\leqslant G$ be a quasi-convex subgroup that has index $k$ in an almost malnormal subgroup and let $B, C\leqslant G$ be two quasi-convex subgroups. Then $B^a\cap C \cap A$ has index at most $k$ in $B^a\cap C$ for all but finitely many double cosets $(A\cap B)\,a\,(A\cap C)$ in $(A\cap B)\backslash A/(A\cap C)$. In particular, if $A$ itself is almost malnormal then $B^a\cap C\leqslant A$ for all but finitely many double cosets $(A\cap B)\, a\, (A\cap C)$.
\end{lem}

\begin{proof}
For each $g\in G$, denote by $B_g = B\cap C^{g^{-1}}\leqslant B$ and $C_g = B^g\cap C\leqslant C$. Note that $B_g^g = C_g$. By Lemma \ref{hyperbolic_lemma}, there exists a constant $L$ such that, for all but finitely many double cosets $(A\cap B)\, a\, (A\cap C)\in (A\cap B)\backslash A/(A\cap C)$, every element of $B_a$ or $C_a$ has $S$-length at most $L$.
In particular $B_a, C_a$ are finite of order at most $(2|S|)^L$. 

Suppose that there are infinitely many double cosets $(A\cap B)\, a\, (A\cap C)$ such that $B_a, C_a\neq 1$. By the pigeonhole principle, there exist finite subgroups $F_B\leqslant B$, $F_C\leqslant C$ and infinitely many $a\in A$ in distinct $(A\cap B, A\cap C)$ double cosets, such that $B_a = F_B$, $C_a = F_C$. If $a_1, a_2\in A$ are such that $B_{a_1} = F_B = B_{a_2}$ and $C_{a_1} = F_C = C_{a_2}$, then since $a_1^{-1}F_Ba_1 = F_C$ and $a_2^{-1}F_Ba_2 = F_C$, we have that $F_B^{a_1a_2^{-1}} = F_B$. Thus, there are infinitely many elements $a\in A$ in distinct $(A\cap B, A\cap B)$ double cosets such that $F_B^a = F_B$. In particular, since $F_B$ is finite, there are also infinitely many such elements $a\in A$ such that $\gamma_a$ induces the identity automorphism on $F_B$. This implies that for each element $f\in F_B$, we have $|A^f\cap A| = \infty$. Since $A$ has index $k$ in an almost malnormal subgroup, say $A'$, it follows that if $|A^f\cap A| = \infty$, we must have that $f\in A'$. Hence $F_B\leqslant A'$ and so $F_B\cap A$ has index at most $k$ in $F_B$ as required. The same argument for $F_C$ implies that $F_C\cap A = B^a\cap C\cap A$ has index at most $k$ in $F_C = B^a\cap C$.
\end{proof}

\begin{cor} 
\label{cor: A-fgip}
If $G$ is locally quasi-convex hyperbolic and $A\leqslant G$ is an almost malnormal subgroup, then $(G, A)$ has the $A$-\fgip\
\end{cor}

Using Lemma \ref{hyperbolic_lemma}, we may now prove that, in a hyperbolic group, pairs of quasi-convex subgroups have finite coset interaction relative to any quasi-convex subgroup.

\begin{thm}
\label{fcip_hyperbolic}
    Let $G$ be a hyperbolic group and let $B, C\leqslant G$ be quasi-convex subgroups. The following holds:
    \begin{enumerate}[(1)]
        \item $(B,C)$ has local finite coset interaction relative to any quasi-convex subgroup;
        \item $(B,C)$ has finite coset interaction relative to any almost malnormal quasi-convex subgroup;
        \item $(B, C)$ has weak finite coset interaction relative to any finite collection $\mathcal{A}$ of at least two quasi-convex subgroups of $G$, provided that $A_{\alpha}^g\cap A_{\beta}$ is finite for all $g\in G$ and for each pair of distinct subgroups $A_{\alpha}, A_{\beta}\in \mathcal{A}$;
        \item $(B,C)$ has $k$-finite coset interaction relative to any quasi-convex subgroup which has index $k$ in an almost malnormal subgroup.
    \end{enumerate}
\end{thm}

\begin{proof}
Let $A$ be a quasi-convex subgroup of $G$, let $S_A$, $S_B$, $S_C$ and $S$ be sets of generators for $A$, $B$, $C$ and $G$, with $S_A$, $S_B$ and $S_C$ finite. Let $\lambda \ge 1$ be given by Proposition~\ref{prop: standard hyperbolic}\eqref{item: lambda}, simultaneously for $A$, $B$ and $C$.

Let $f,g \in G$, let $X, X'$ be double cosets in $\dblcoset{(A\cap B^f)}A{(A \cap C^g)}$ and let $a$ and $a'$ be elements of $X$ and $X'$, with minimum $S_A$-length. By Lemma~\ref{hyperbolic_lemma}, there exists an integer $L$ such that, if $b\,(fag\inv) = (fa'g)\,c$ for some $b\in B$ and $c\in C$ (that is, if $q_{f,g}^A(X) = q^A_{f,g}(X')$), then either $|a|_{S_A}, |a'|_{S_A} \le L$, or $|b|_{S_B}, |c|_{S_C} \le L$.

In the second case, we have $a' = f\inv\,b\,f\,a\,g\inv\,c\inv\,g$, so that
\begin{align*}
|a'|_S &\le |a|_S + |b|_S + |c|_S + 2\,(|f|_S + | g|_S) \\
&\le |a|_S + \lambda(|b|_{S_B}+\lambda) + \lambda(|c|_{S_C}+\lambda) + 2\,(|f|_S + |g|_S) \\
&\le |a|_S + 2\lambda(L+\lambda) + 2(|f|_S + |g|_S).
\end{align*}
This proves that $q_{f,g}\inv\left(q_{f,g}(X)\right)$ is finite, thus establishing (1).

Now suppose in addition that $A$ is almost malnormal, and suppose that the pair $(B,C)$ does not have finite coset interaction relative to $A$. By an application of the pigeonhole principle, there exist $f, f', g, g' \in G$ such that $B\,f\,A \ne B\,f'\,A$ and $C\,g\,A \ne C\,g'\,A$, and infinitely many pairs $(a_i,a'_i)$ ($i\in \N$) of elements of $A$ such that $B\,(fa_ig\inv)\,C = B\,(f'a'_i{g'}\inv)\,C$. Note that the condition on $f, f'$ (resp. $g, g'$) implies that ${f'}\inv Bf$ (resp. ${g'}\inv Cg$) does not meet $A$. Without loss of generality, we may assume that each $a_i$ (resp. $a'_i$) has minimal $S_A$-length in the double coset $(A\cap B^f)a_i(A \cap C^g)$ (resp. $(A\cap B^{f'})a'_i(A \cap C^{g'})$). For each $i$, we have $f'a'_i{g'}\inv = b_i\,(fa_ig\inv)\,c_i$ for some $b_i\in B$ and $c_i\in C$.

Let $L$ be the constant given by Lemma~\ref{hyperbolic_lemma}. Using the pigeonhole principle again, we may assume that, for each $i$, $|a_i|_{S_A}, |a'_i|_{S_A} > L$. Lemma~\ref{hyperbolic_lemma} then shows that $|b_i|_{S_B}, |c_i|_{S_C} \le L$. A new application of the pigeonhole principle shows that we can assume that the $b_i$ (resp. $c_i$) are all equal, say to $b\in B$ (resp. $c\in C$).

Thus, for every $i\in \N$, $f'a'_i{g'}\inv = b\,(fa_ig\inv)\,c$. Let $i \ne j$. Then 
\begin{align*}
a_i &= (f\inv b\inv f')\,a'_i\,({g'}\inv c\inv g) \\
a_j &= (f\inv b\inv f')\,a'_j\,({g'}\inv c\inv g) \\
a_i{a_j}\inv &= (f\inv b\inv f')\,(a'_i{a'_j}\inv)\,(f\inv b\inv f')\inv.
\end{align*}
As observed before, ${f'}\inv b\inv f \not\in A$, so the almost malnormality of $A$ implies that the $(a_ia_j\inv)$ ($i,j\in \N$) take finitely many values. This contradicts the fact that the $a_i$ are pairwise distinct. Thus (2) is established.

Now we prove (3). It suffices to assume that $\mathcal{A} = \{A_1, A_2\}$ consists of two subgroups. Suppose for a contradiction that the pair $(B,C)$ does not have weak finite coset interaction relative to the pair $\{A_1, A_2\}$. By an application of the pigeonhole principle, there exist $f, f', g, g' \in G$ and infinitely many pairs $(a_i,a'_i)$ ($i\in \N$) of elements of $A$ such that $B\,(fa_ig\inv)\,C = B\,(f'a'_i{g'}\inv)\,C$. Without loss of generality, we may assume that each $a_i$ (resp. $a'_i$) has minimal $S_{A_1}$-length (resp. $S_{A_2}$-length) in the double coset $(A_1\cap B^f)a_i(A_1 \cap C^g)$ (resp. $(A_2\cap B^{f'})a'_i(A_2 \cap C^{g'})$). For each $i$, we have $f'a'_i{g'}\inv = b_i\,(fa_ig\inv)\,c_i$ for some $b_i\in B$ and $c_i\in C$. In the same way as before, we may use Lemma~\ref{hyperbolic_lemma} and the pigeonhole principle to find infinitely many distinct elements $a_ia_j^{-1}\in A_1^{f^{-1}b^{-1}f'}\cap A_2$, contradicting our assumption that $A_1^g\cap A_2$ is finite for all $g\in G$. This establishes (3).

Finally (4) follows from the combination of (2), Proposition~\ref{index_k} and the following elementary fact: if $A, A'$ are subgroups of $G$ and $A'$ has index $k$ in $A$, then $A$ is quasi-convex if and only if $A'$ is quasi-convex.
\end{proof}

Recall that a group is \emph{locally quasi-convex} if every finitely generated subgroup is quasi-convex. Theorem~\ref{fcip_hyperbolic} immediately yields the following corollary.

\begin{cor}
\label{fcip_qc_hyperbolic}
    If $G$ is a locally quasi-convex hyperbolic group, then:
    \begin{enumerate}[(1)]
        \item $G$ has the local \fcip\ relative to any quasi-convex subgroup;
        \item $G$ has the \fcip\ relative to any almost malnormal quasi-convex subgroup;
        \item $G$ has weak \fcip\ relative to any finite collection $\mathcal{A}$ of at least two quasi-convex subgroups such that $A_{\alpha}^g\cap A_{\beta}$ is finite for all $g\in G$ and all distinct pairs $A_{\alpha}, A_{\beta}\in \mathcal{A}$;
        \item $G$ has the $k$-\fcip\ relative to any quasi-convex subgroup which has index $k$ in an almost malnormal subgroup.
    \end{enumerate}
\end{cor}

A finite collection of subgroups $\mathcal{A}$ of $G$ is an \emph{almost malnormal collection} if for any $A_{\alpha}, A_{\beta}\in \mathcal{A}$ and for any $g\in G$, either the subgroup $A_{\alpha}\cap A_{\beta}^g$ is finite or $\alpha = \beta$ and $g\in A_{\beta}$. In particular (choosing $\alpha = \beta$), every $A_{\alpha}$ is almost malnormal. The proof of the following theorem proceeds in exactly the same way as that of Theorem \ref{fcip_hyperbolic}.

\begin{thm}
\label{fcip_hyperbolic_2}
    Let $G$ be a hyperbolic group and let $B, C$ be quasi-convex subgroups. Then $(B, C)$ has finite coset interaction relative to any finite almost malnormal collection of quasi-convex subgroups.
\end{thm}

\begin{cor}
\label{fcip_qc_hyperbolic_2}
    If $G$ is a locally quasi-convex hyperbolic group then $G$ has the \fcip\ relative to any finite almost malnormal collection of subgroups.
\end{cor}

\subsection{The \fgip\ for graphs of locally quasi-convex hyperbolic groups}
\label{sec: hyperbolic_new}

Combining Theorem \ref{fgip_criterion_3} and Corollaries \ref{cor: A-fgip} and \ref{fcip_qc_hyperbolic}, we obtain the following two criteria for one-edge graphs of locally quasi-convex hyperbolic groups to have the \fgip\ Recall that locally quasi-convex hyperbolic groups have the \sfgip\ relative to the collection of their finite subgroups (see Proposition~\ref{prop: hyperbolic_sfgip_2}).

\begin{thm}
\label{hyperbolic_amalgam}
    Let $G \cong A*_CB$ be an amalgamated free product where $A$ is a locally quasi-convex hyperbolic group, $B$ has the \fgip\ and $C$ is finitely generated and almost malnormal in $A$. Then $G$ has the \fgip\ relative to $C$ if and only if $G$ has the \fgip\
\end{thm}

\begin{thm}
\label{hyperbolic_HNN}
    Let $G = A*_{\varphi}$ be a HNN extension, where $\varphi\colon C\to \varphi(C)$ identifies two subgroups of $A$, where $A$ is locally quasi-convex hyperbolic and where $C$ is an almost malnormal subgroup of $A$. Then $G$ has the \fgip\ relative to $C$ if and only if $G$ has the \fgip\ 
\end{thm}

We obtain more precise statements for graphs of locally quasi-convex hyperbolic groups when the edge groups are virtually infinite cyclic.

\def\WW{\mathbb{W}}

Let $\AA$ be such a graph, let $u_0\in V(\gr A)$ and let $A = \pi_1(\AA,u_0)$. We start with the construction of $\WW(\AA)$, a graph of virtually $\ZZ$ groups with virtually $\ZZ$ edge groups, associated with $\AA$. This should be thought of as a generalisation of Wise's graph of cyclic groups from \cite[Definition 4.16]{Wi00}, a construction which proved useful in the literature for characterising properties in graphs of free groups with cyclic edge groups. Indeed, it will turn out that $\pi_1(\AA, u_0)$ will have the \fgip\ if and only if $\pi_1(\WW(\AA))$ does.

Recall that if $G$ is a group and $H\leqslant G$ is a subgroup, then the commensurator of $H$ in $G$ is the subgroup
    \[
    \comm_G(H) = \{ g\in G \mid [H\colon H\cap H^g], [H^g \colon H\cap H^g]<\infty\}.
    \]
For each edge $e\in E(\gr{A})$, let $C_e = \comm_{A_{o(e)}}(\alpha_e(A_e))$. For convenience, we call $C_e$ the \defin{commensurator} of $e$. We note that, since a subgroup of a virtually $\Z$ group has finite index if and only if it is infinite (see Fact~\ref{fact: elementary virtually Z}), we have
\[
C_e = \{
y \in A_{o(e)} \mid \alpha_e(A_e) \cap \alpha_e(A_e)^y \text{ is infinite}  
\}.
\]

\begin{lem}
    For each edge $e \in E(\gr{A})$, the commensurator $C_e$ is virtually $\Z$, $\comm_{A_{o(e)}}(C_e) = C_e$ and $C_e$ is almost malnormal in $A_{o(e)}$.
\end{lem}

\begin{proof}
    Since every vertex group in $\AA$ is locally quasi-convex hyperbolic and every edge group $A_e$ is virtually $\ZZ$, the image $\alpha_e(A_e)$ is an infinite quasi-convex subgroup of a hyperbolic group, and hence it has finite index in its commensurator $C_e$ (see, for example, \cite[Theorem 1]{ks96}). Using Fact~\ref{fact: elementary virtually Z}, it follows that, for each edge $e \in E(\gr{A})$, the $C_e$ is virtually $\Z$, $\comm_{A_{o(e)}}(C_e) = C_e$ and hence $C_e$ is almost malnormal in $A_{o(e)}$.
\end{proof}

Next, we say that two edges $e, e' \in E(\gr{A})$ are \emph{conjugate-commensurable (c.c.)} if they have the same origin ($o(e) = o(e')$) and there exists $x\in A_{o(e)}$ such that the subgroups $\alpha_e(A_e)^x$ and $\alpha_{e'}(A_{e'})$ are commensurable
(or equivalently, in our case, if $\alpha_e(A_e)^x\cap \alpha_{e'}(A_{e'})$ is infinite). It is routine to check that this is an equivalence relation in $E(\gr{A})$. Moreover, if $o(e) = o(e')$ and $\alpha_e(A_e)^x$ and $\alpha_{e'}(A_{e'})$ are commensurable, then
\begin{align*}
C_e &= \comm_{A_{o(e)}}(\alpha_e(A_e)\cap \alpha_{e'}(A_{e'})^{x^{-1}})\\
C_{e'} &= \comm_{A_{o(e)}}(\alpha_e(A_e)^x\cap \alpha_{e'}(A_{e'}))
\, ,
\end{align*}
which implies that $C_{e}^x = C_{e'}$. Thus the commensurators of all the edges in any given c.c.-class are conjugate to each other.

Now, we successively modify our graph of groups $\AA$ into graphs of groups $\AA_1, \AA_2$ and $\AA_3$ as follows.

\begin{enumerate}[(a)]
    \item For each c.c.\ class of edges, pick a representative $e$, and for each other edge $e' \in [e]$, replace $\alpha_{e'}$ with a suitable $\gamma_x\circ\alpha_{e'}$ so that all the edges in the c.c.-class $[e]$ have the same commensurator $C_e$. Let $\AA_1$ be the resulting graph of groups.

    \item Subdivide every edge $e$ of $\AA_1$ by introducing two auxiliary intermediate vertices $v_e$ and~$v_{e\inv}$, with associated vertex groups $C_e$ and $C_{e^{-1}}$, realising the inclusions 
$\alpha_{e}(A_{e}) \leqslant C_{e} \leqslant A_{o(e)}$ and $\omega_{e}(A_{e}) \leqslant C_{e^{-1}} \leqslant A_{t(e)}$, respectively. More explicitly, the edge $e$ in~$\AA_1$
\begin{figure}[H]
\centering
\begin{tikzpicture}[shorten >=3pt, node distance=.6cm and 1.5cm, on grid,auto,-latex]
   
   \node[] (Au) {$A_u$};
   \node[] (Ae) [right = of Au] {$A_e$};
   \node[] (Auu) [right = of Ae] {$A_{u'}$};

    \path[>->] (Ae) edge[] 
    node[pos=0.5,above=-.2mm]
    {$\alpha_{e}$}(Au);

    \path[>->] (Ae) edge[] 
    node[pos=0.5,above=-.2mm]
    {$\omega_{e}$}(Auu);

    \node[] (u) [above = of Au] {$u$};
    \node[] (uu) [above = of Auu] {$u'$};
    \path[-stealth] (u) edge[] 
    node[pos=0.5,above=-.2mm]
    {$e$}(uu);
\end{tikzpicture}
\end{figure}
\vspace{-15pt}
is replaced by the three-edge path of groups:
\begin{figure}[H]
\centering
\begin{tikzpicture}[shorten >=0pt, node distance=.6cm and 1.5cm, on grid,auto,-latex]
   
    \node[] (Au) [right = of Auu] {$A_u$};
    \node[] (Ce) [right = of Au] {$C_e$};
    \node[] (Ae) [right = of Ce] {$A_e$};
    \node[] (Cee) [right = of Ae] {$C_{e^{-1}}$};
    \node[] (Av) [right = of Cee] {$A_{u'}$};

    \path[{Hooks[left, length=2pt, width=7pt]}->] (Ce) edge[] 
    node[pos=0.5,above=-.2mm]
    {}(Au);

    \path[>->] (Ae) edge[] 
    node[pos=0.5,above=-.2mm]
    {$\alpha_e$}(Ce);

    \path[>->] (Ae) edge[] 
    node[pos=0.5,above=-.2mm]
    {$\omega_e$}(Cee);

    \path[{Hooks[right, length=2pt, width=7pt]}->] (Cee) edge[] 
    node[pos=0.5,above=-.2mm]
    {}(Av);

    \node[] (u) [above = of Au] {$u$};
    \node[] (ve) [above = of Ce] {$v_e$};
    \node[] (vee) [above = of Cee] {$v_{e\inv}$};
    \node[] (uu) [above = of Av] {$u'$};

    \path[>=stealth] (u) edge[] 
    node[pos=0.5,above=-.2mm]
    {$e_{init}$}(ve);

    \path[>=stealth] (ve) edge[] 
    node[pos=0.5,above=-.2mm]
    {$\tilde{e}$}(vee);

    \path[>=stealth] (uu) edge[] 
    node[pos=0.5,above=-.2mm]
    {$e\inv_{init}$}(vee);
\end{tikzpicture}
\end{figure}
\vspace{-15pt}
where the edge group of $e_{init}$ is $C_e$, $\alpha_{e_{init}}$ is the inclusion of $C_e$ in $A_{o(e)}$ and $\omega_{e_{init}}$ is the identity. We denote the resulting graph of groups by $\AA_2$. We also let $V_2$ be the set of auxiliary vertices used in the subdivisions, $V_2 = \{v_e \mid e\in E(\gr A_1)\}$.

\item Now we quotient $\AA_2$ by the initial segments of edges in the same c.c.-class. That is, for every c.c.-class $[e]$ of $E(\gr{A})$ we identify in $\gr{A}_2$ the edges $e'_{init}$ (and their end vertices $v_{e'}$) for all $e'\in [e]$, to a single edge $[e]_{init}$ with end vertex $v_{[e]}$:
\begin{figure}[H]
\centering
\begin{tikzpicture}[shorten >=0pt, node distance=.6cm and 1.5cm, on grid,auto,-latex]

   \begin{scope}
    \node[] (Au) {$A_u$};
    \node[] (Ce) [right = of Au] {$C_{e}$};

    \path[{Hooks[left, length=2pt, width=7pt]}->] (Ce) edge[] 
    node[pos=0.5,above=-.2mm]
    {}(Au);

    \node[] (u) [above = of Au] {$u$};
    \node[] (ve) [above = of Ce] {$v_{e'}$};
    \node[] (fa) [left = 1.25 of u] {for all $e' \in [e]\,,$};
   
    \path[>=stealth] (u) edge[] 
    node[pos=0.5,above=-.2mm]
    {$e'_{init}$}(ve);

    \node[] (i) [right = 1 of ve]{};
    \node[] (f) [right = 2 of ve] {};
    \path[->]
        (i) edge[bend left] (f);
    \end{scope}

    \begin{scope}[xshift=4.5cm ]
    \node[] (Au) {$A_u$};
    \node[] (Ce) [right = of Au] {$C_{e}$};

    \path[{Hooks[left, length=2pt, width=7pt]}->] (Ce) edge[] 
    node[pos=0.5,above=-.2mm]
    {}(Au);

    \node[] (u) [above = of Au] {$u$};
    \node[] (ve) [above = of Ce] {$v_{[e]}$};
   
    \path[>=stealth] (u) edge[] 
    node[pos=0.5,above=-.2mm]
    {$[e]_{init}$}(ve);
    \end{scope}
\end{tikzpicture}
\end{figure}
\vspace{-10pt}
We denote the resulting graph of groups by $\AA_3$, and by $V_3$ the set of vertices of the form $v_{[e]}$. In particular, $V(\gr A_3)$ is the disjoint union of $V(\gr A)$ and $V_3$. 

\item Finally, we let $\WW(\AA)$ be the subgraph of groups of $\AA_3$ induced by $V_3$.
\end{enumerate}

\begin{prop}\label{prop: properties of W(A)}
    Let $(\AA,u_0)$ be a pointed graph of locally quasi-convex hyperbolic groups with virtually $\ZZ$ edge groups as above.
    \begin{enumerate}[(1)]
    \item\label{item: pi1 of A3 is pi1 of A} $\pi_1(\AA_3,u_0) = \pi_1(\AA,u_0)$.
    \item\label{item: W(A) is virtually Z}$\WW(\AA)$ is a graph of virtually $\ZZ$ groups with virtually $\ZZ$ edge groups.
    \end{enumerate}    
\end{prop}

\begin{proof}
    The transformations from $\AA$ to $\AA_1$, $\AA_2$ and $\AA_3$ do not modify the fundamental groups, so \eqref{item: pi1 of A3 is pi1 of A} holds.

    The vertex groups of $\WW(\AA)$ are the vertex groups of $\AA_3$ at the vertices in $V_3$, namely the commensurators of the edges of $\AA$. We already observed that these groups are virtually $\ZZ$. The edges of $\WW(\AA)$ are the edges of $\AA_3$ between vertices in $V_3$, namely the $\tilde e$ ($e \in E(\gr A)$), and their edge groups are the $A_e$, which are virtually $\ZZ$ by assumption. This concludes the verification of~\eqref{item: W(A) is virtually Z}.
\end{proof}

We also record a computability result for $\WW(\AA)$. This requires the following technical (and well-known) results on hyperbolic groups.

\begin{fact}\label{fact: compute Ce}
    If a hyperbolic group $G$ is given by a finite presentation and $H$ is a virtually $\ZZ$ subgroup of $G$, given by a finite set of generators, then one can compute a generating set of $\comm_G(H)$ and the index of $\comm_G(H)$ in $G$. Moreover, if $H$ and $K$ are quasi-convex subgroups of $G$, one can decide whether $H$ and $K$ are conjugates.
\end{fact}

\begin{proof}
    By \cite[Corollary III.$\Gamma$.3.10]{bh99}, a virtually $\Z$ subgroup $H$ of a hyperbolic group $G$ is quasi-convex. Since $H$ has finite index in its commensurator $\comm_G(H)$ and since $\comm_G(H)$ is almost malnormal in $G$, we see that $\comm_G(H)$ is the union of all (finitely many) cosets $gH$ so that $H\cap H^g$ is infinite (and hence, finite index in $H$ and $H^g$). Thus, the index of $H$ in $\comm_G(H)$ is precisely the height of $H$ which is computable by \cite[Corollary 6.10]{kmw17}. By combining \cite[Proposition 6.8 \& Corollary 6.10]{kmw17}, we see that a complete list of $(H, H)$ double cosets $g_1, \ldots, g_n$ so that $H\cap H^g$ is infinite is computable. Since $\comm_G(H) = \langle H, g_1, \ldots, g_n\rangle$, a generating set for $\comm_G(H)$ is computable as desired.

    Finally, \cite[Corollary 6.10]{kmw17} also establishes the decidability of conjugacy for all quasi-convex subgroups of $G$.
\end{proof}

\begin{prop}\label{prop: A3 computable}
    Let $(\AA,u_0)$ be a pointed graph of locally quasi-convex hyperbolic groups with virtually $\ZZ$ edge groups. If $\gr A$ is finite and its vertex and edge groups are given explicitly, then $\AA_3$ and $\WW(\AA)$ are computable.
\end{prop}

\begin{proof}
    In view of Fact~\ref{fact: compute Ce}, we can compute the commensurators $C_e$ of all the edges of $A$ and the c.c.-classes. The computability of $\AA_1$, $\AA_2$, $\AA_3$ and $\WW(\AA)$ follows.
\end{proof}

We can now establish a characterisation of the \fgip\ for $\pi_1(\AA,u_0)$ in terms of the properties of the fundamental groups of the connected components of $\WW(\AA)$.

\begin{prop}\label{prop: charact fgip locally qcv hyp}
    Let $(\AA,u_0)$ be a pointed graph of locally quasi-convex hyperbolic groups with virtually $\Z$ edge groups. Then $\pi_1(\AA, u_0)$ has the \fgip\ if and only if the fundamental group of every connected component of $\WW(\AA)$ has the \fgip
\end{prop}

\begin{proof}
    Let $A = \pi_1(\AA, u_0)$.
    Note that $\WW(\AA)$ is a (possibly non-connected) graph of virtually $\ZZ$ groups, with virtually $\ZZ$ edge groups. We let $(\WW_j)_j$ be the collection of connected components of $\WW(\AA)$. For each $j$, let $W_j = \pi_1(\WW_j,v_j)$ where $v_j$ is a vertex of $\WW_j$ chosen arbitrarily.
    
    Of course, each $W_j$ is isomorphic to a subgroup of $A$. In particular, if $A$ has the \fgip, then so does each $W_j$.

    Suppose, conversely, that each $W_j$ has the \fgip\  Let $\AA_4$ be the graph of groups obtained from $\AA_3$ by quotienting out the $\mathbb{W}_j$s. That is, we collapse each $\mathbb{W}_j$ into a single vertex $w_j$ with associated vertex group $W_j$, naturally updating the adjacent edges and respective edge maps so that the resulting pointed graph of groups $(\AA_4,u_0)$ has the same fundamental group as $(\AA_3,u_0)$, and hence as $(\AA,u_0)$.

    By construction, $V(\gr A_4)$ is the disjoint union of $V(\gr A)$ and $V_4 = \{w_j\}_j$, and every edge of $\gr A_4$ connects a vertex of $V(\gr A)$ with one in $V_4$. Let $E^+$ be the orientation of $\AA_4$ which selects each edge from $V(\gr A)$ to $V_4$. Note that the edge groups of $\AA_4$ are all commensurators of edges in $E(\gr A)$, and hence are virtually $\ZZ$. We want to use Theorem~\ref{fgip_criterion_3} to establish that $A$ has the \fgip\ It suffices to verify the following.

    \begin{itemize}
    \item If $u\in V(\gr A)$, then $A_u$ has the \fcip\ relative to the collection $\{\alpha_e(A_e) \mid e\in E^+, \ o(e) = u\}$. This holds by Corollary \ref{fcip_qc_hyperbolic_2}. Indeed, $A_u$ is locally quasi-convex hyperbolic, and the $\alpha_e(A_e)$ form an almost malnormal collection of subgroups (because there are commensurators of virtually $\Z$ groups that lie in distinct c.c.-classes).

    \item If $u\in V(\gr A)$ and $e\in E^+$, then $(A_u, \alpha_e(A_e))$ has the $\alpha_e(C_e)$-\fgip\ This holds by Corollary~\ref{cor: A-fgip}.
    
    \item For each edge $e\in E^+$, $(A_e, \mathcal F_e)$ has the \sfgip, where $\mathcal F_e$ is the collection of finite subgroups of $A_e$, and there is a bound on the order of groups in $\mathcal F_e$. The \sfgip\ holds by Proposition~\ref{prop: hyperbolic_sfgip_2}. When $e$ is fixed, the corresponding virtually $\ZZ$ edge group is fixed, and the order of the finite groups in it is bounded.
    
    \item $A$ has the \fgip\ relative to each edge group, and this holds because each edge group is virtually $\ZZ$ and hence has only finitely generated subgroups.
    \end{itemize}
Theorem \ref{fgip_criterion_3} can now be used, to show that $A$ has the \fgip
\end{proof}

This yields the following important corollaries. The first is a characterisation of the \fgip\ for graphs of locally quasi-convex hyperbolic groups with virtually $\Z$ edge groups.

\begin{cor}\label{cor: main locally qc characterization}
    Let $(\AA,u_0)$ be a pointed graph of locally quasi-convex hyperbolic groups with virtually $\Z$ edge groups.
    Then
    $\pi_1(\AA, u_0)$ has the \fgip\ if and only if it does not contain $F_2 \times \ZZ$
\end{cor}

\begin{proof}
    If $A$ has the \fgip, then it cannot contain $F_2\times \ZZ$. Conversely, if $A$ does not contain $F_2\times \ZZ$, neither do the fundamental groups of the connected components of $\WW(\AA)$. By Theorem~\ref{thm: vcyclic_fgip}, it follows that these fundamentaly groups have the \fgip, and Proposition~\ref{prop: charact fgip locally qcv hyp} then shows that $A$ has the \fgip
\end{proof}

The second corollary is a decidability result for the \fgip

\begin{cor}\label{cor: decidability for locally qcv hyperbolic}
    Let $(\AA,u_0)$ be a pointed graph of locally quasi-convex hyperbolic groups with virtually $\Z$ edge groups. If $\gr A$ is finite and the vertex and edge groups of $\AA$ are given explicitly, then one can decide whether $\pi_1(\AA,u_0)$ has the \fgip
\end{cor}

\begin{proof}
    By Proposition~\ref{prop: charact fgip locally qcv hyp}, $A$ has the \fgip\ if and only if each connected component of $\WW(\AA)$ does. This in turn is decidable in view of Remark~\ref{rk: decide fgip for virtually Z}, provided that $\WW(\AA)$ is computable. This was established in Proposition~\ref{prop: A3 computable}.
\end{proof}



\renewcommand*{\bibfont}{\small}
\printbibliography

@article{ass15,
  title = {Finiteness results for subgroups of finite extensions},
  author = {Araújo, Vítor and Silva, Pedro V. and Sykiotis, Mihalis},
  journal = {Journal of Algebra},
  volume = {423},
  pages = {592--614},
  year = {2015},
  month = {Feb},
  doi = {10.1016/j.jalgebra.2014.10.033},
  url = {http://www.sciencedirect.com/science/article/pii/S0021869314006243}
}

@article{Ar98,
 author = {Arzhantseva, G. N.},
 title = {Generic properties of finitely presented groups and {Howson}'s theorem},
 fjournal = {Communications in Algebra},
 journal = {Commun. Algebra},
 issn = {0092-7872},
 volume = {26},
 number = {11},
 pages = {3783--3792},
 year = {1998},
 language = {English},
 doi = {10.1080/00927879808826374}
}

@article{Ar2000,
 ISSN = {00029939, 10886826},
 URL = {http://www.jstor.org/stable/2668655},
 author = {G. N. Arzhantseva},
 journal = {Proceedings of the American Mathematical Society},
 number = {11},
 pages = {3205--3210},
 publisher = {American Mathematical Society},
 title = {A Property of Subgroups of Infinite Index in a Free Group},
 urldate = {2026-03-13},
 volume = {128},
 year = {2000}
}

@article{bas93,
  title = {Covering theory for graphs of groups},
  author = {Bass, Hyman},
  journal = {Journal of Pure and Applied Algebra},
  volume = {89},
  number = {1},
  pages = {3--47},
  year = {1993},
  month = {Oct}
}

@article{bau66,
  title = {Intersections of Finitely Generated Subgroups in Free Products},
  author = {Baumslag, B.},
  journal = {Journal of the London Mathematical Society},
  volume = {s1-41},
  pages = {673--679},
  year = {1966},
  month = {Jan}
}

@article{bb79,
  title = {Two remarks on the group property of Howson},
  author = {Burns, R. G. and Brunner, A. M.},
  journal = {Algebra and Logic},
  volume = {18},
  number = {5},
  pages = {319--325},
  year = {1979},
  month = {Sep},
  doi = {10.1007/BF01673500},
  url = {https://doi.org/10.1007/BF01673500}
}

@article{bf91,
  title = {Bounding the complexity of simplicial group actions on trees},
  author = {Bestvina, Mladen and Feighn, Mark},
  journal = {Inventiones Mathematicae},
  volume = {103},
  number = {3},
  pages = {449--469},
  year = {1991},
  doi = {10.1007/BF01239522},
  url = {https://zbmath.org/?q=an:0724.20019}
}

@book{bh99,
  title = {Metric spaces of non-positive curvature},
  author = {Bridson, Martin R. and Haefliger, André},
  series = {Grundlehren der mathematischen Wissenschaften [Fundamental Principles of Mathematical Sciences]},
  volume = {319},
  publisher = {Springer-Verlag, Berlin},
  pages = {xxii+643},
  year = {1999},
  doi = {10.1007/978-3-662-12494-9},
  url = {https://doi.org/10.1007/978-3-662-12494-9}
}

@article{bur72,
  title = {On the Finitely Generated Subgroups of an Amalgamated Product of Two Groups},
  author = {Burns, R. G.},
  journal = {Transactions of the American Mathematical Society},
  volume = {169},
  pages = {293--306},
  year = {1972},
  month = {Jul},
  doi = {10.2307/1996244},
  url = {http://www.jstor.org/stable/1996244}
}

@article{bur73,
  title = {Finitely Generated Subgroups of HNN Groups},
  author = {Burns, R. G.},
  journal = {Canadian Journal of Mathematics},
  volume = {25},
  number = {5},
  pages = {1103--1112},
  year = {1973},
  month = {Oct},
  doi = {10.4153/CJM-1973-117-7},
  url = {https://www.cambridge.org/core/journals/canadian-journal-of-mathematics/article/finitely-generated-subgroups-of-hnn-groups/B92A3CA80972336B2E1EB585F87A865C}
}

@article{bw2022,
	title = {Failure of the finitely generated intersection property for ascending {HNN} extensions of free groups},
	volume = {32},
	issn = {0218-1967, 1793-6500},
	url = {https://www.worldscientific.com/doi/10.1142/S0218196722500370},
	doi = {10.1142/S0218196722500370},
	number = {5},
	journaltitle = {International Journal of Algebra and Computation},
	shortjournal = {Int. J. Algebra Comput.},
	author = {Bamberger, Jacob and Wise, Daniel T.},
	urldate = {2025-11-04},
	date = {2022-08},
	langid = {english},
}

@article{coh74,
  title = {Subgroups of HNN groups},
  author = {Cohen, D. E.},
  journal = {Journal of the Australian Mathematical Society},
  volume = {17},
  number = {4},
  pages = {394--405},
  year = {1974},
  month = {Jun},
  doi = {10.1017/S1446788700018036},
  url = {https://www.cambridge.org/core/journals/journal-of-the-australian-mathematical-society/article/subgroups-of-hnn-groups/2965E066B2A843E2276907ED5ED2A51A}
}

@article{coh76,
  title = {Finitely generated subgroups of amalgamated free products and HNN groups},
  author = {Cohen, Daniel E.},
  journal = {Journal of the Australian Mathematical Society},
  volume = {22},
  number = {3},
  pages = {274--281},
  year = {1976},
  month = {Nov},
  doi = {10.1017/S1446788700014737},
  url = {https://www.cambridge.org/core/journals/journal-of-the-australian-mathematical-society/article/finitely-generated-subgroups-of-amalgamated-free-products-and-hnn-groups/13A3EE75F81B0F81627CCE744CAE0295}
}

@article{dah2003,
	title = {Combination of convergence groups},
	volume = {7},
	pages = {933--963},
	journaltitle = {Geometry \& Topology},
	author = {Dahmani, François},
	date = {2003-11-12}
}

@book{dd89,
  title = {Groups acting on graphs},
  author = {Dicks, Warren and Dunwoody, M. J.},
  publisher = {Cambridge University Press},
  address = {Cambridge},
  series = {Cambridge Studies in Advanced Mathematics},
  volume = {17},
  year = {1989},
  url = {http://www.ams.org/mathscinet-getitem?mr=1001965}
}

@article{df05,
  title = {The Grushko decomposition of a finite graph of finite rank free groups: an algorithm},
  author = {Diao, Guo-An and Feighn, Mark},
  journal = {Geometry and Topology},
  volume = {9},
  pages = {1835--1880},
  year = {2005},
  doi = {10.2140/gt.2005.9.1835},
  url = {https://doi.org/10.2140/gt.2005.9.1835}
}

@article{dllrw_pullback,
  title = {Pullbacks and intersections in categories of graphs of groups},
  author = {Delgado, J. and Lopez de Gamiz, J. and Linton, M. and Roy, M. and Weil, P.},
  journal = {arXiv:2508.04362},
  year = {2025},
  %shorthand = {DLLRW}
}

@incollection{Du93,
 author = {Dunwoody, Martin J.},
 title = {An inaccessible group},
 booktitle = {Geometric group theory. Volume 1. Proceedings of the symposium held at the Sussex University, Brighton (UK), July 14-19, 1991},
 isbn = {0-521-43529-3},
 pages = {75--78},
 year = {1993},
 publisher = {Cambridge: Cambridge University Press},
 language = {English},
}

@book{fgmrs2014,
	title = {The Elementary Theory of Groups: A Guide through the Proofs of the Tarski Conjectures},
	rights = {De Gruyter expressly reserves the right to use all content for commercial text and data mining within the meaning of Section 44b of the German Copyright Act.},
	isbn = {978-3-11-034203-1},
	url = {https://www.degruyterbrill.com/document/doi/10.1515/9783110342031/html},
	shorttitle = {The Elementary Theory of Groups},
	publisher = {De Gruyter},
	author = {Fine, Benjamin and Gaglione, Anthony and Myasnikov, Alexei and Rosenberger, Gerhard and Spellman, Dennis},
	urldate = {2025-11-04},
	date = {2014-10-29},
	langid = {english},
	doi = {10.1515/9783110342031},
}

@article{fri15,
  title = {Sheaves on graphs, their homological invariants, and a proof of the Hanna Neumann conjecture: with an appendix by Warren Dicks},
  author = {Friedman, Joel},
  journal = {Memoirs of the American Mathematical Society},
  volume = {233},
  number = {1100},
  year = {2015},
  month = {Jan},
  doi = {10.1090/memo/1100},
  url = {http://www.ams.org/memo/1100/}
}

@book{gh90,
  title = {Sur les groupes hyperboliques d'après Mikhael Gromov},
  editor = {Ghys, É. and de la Harpe, P.},
  series = {Progress in Mathematics},
  volume = {83},
  note = {Papers from the Swiss Seminar on Hyperbolic Groups held in Bern, 1988},
  publisher = {Birkhäuser Boston, Inc., Boston, MA},
  pages = {xii+285},
  year = {1990},
  doi = {10.1007/978-1-4684-9167-8},
  url = {https://doi.org/10.1007/978-1-4684-9167-8}
}

@article{gre60,
  title = {Discrete Groups of Motions},
  author = {Greenberg, Leon},
  journal = {Canadian Journal of Mathematics},
  volume = {12},
  pages = {415--426},
  year = {1960},
  month = {Jan},
  doi = {10.4153/CJM-1960-036-8},
  url = {https://www.cambridge.org/core/journals/canadian-journal-of-mathematics/article/discrete-groups-of-motions/43883F46C039C2CE498F55D1FB5FC7D3}
}

@book{hem76,
  title = {{$3$}-{M}anifolds},
  author = {Hempel, John},
  series = {Annals of Mathematics Studies},
  volume = {No. 86},
  publisher = {Princeton University Press, Princeton, NJ; University of Tokyo Press, Tokyo},
  pages = {xii+195},
  year = {1976}
}

@article{how54,
  title = {On the intersection of finitely generated free groups},
  author = {Howson, A. G.},
  journal = {Journal of the London Mathematical Society},
  volume = {s1-29},
  number = {4},
  pages = {428--434},
  year = {1954},
  month = {Oct},
  doi = {10.1112/jlms/s1-29.4.428},
  url = {http://jlms.oxfordjournals.org/content/s1-29/4/428.full.pdf+html?frame=sidebar}
}

@article{hw21,
  title = {A note on finiteness properties of graphs of groups},
  author = {Haglund, Frédéric and Wise, Daniel T.},
  journal = {Proceedings of the American Mathematical Society. Series B},
  volume = {8},
  pages = {121--128},
  year = {2021},
  doi = {10.1090/bproc/81},
  url = {https://doi.org/10.1090/bproc/81}
}

@article{imr84,
  title = {Grushko's theorem},
  author = {Imrich, Wilfried},
  journal = {Archiv der Mathematik},
  volume = {43},
  number = {5},
  pages = {385--387},
  year = {1984},
  doi = {10.1007/BF01193843},
  url = {https://doi.org/10.1007/BF01193843}
}

@article{kap97,
  title = {Amalgamated products and the Howson property},
  author = {Kapovich, Ilya},
  journal = {Canadian Mathematical Bulletin},
  volume = {40},
  number = {3},
  pages = {330--340},
  year = {1997},
  month = {Sep},
  doi = {10.4153/CMB-1997-039-3},
  url = {http://www.cms.math.ca/10.4153/CMB-1997-039-3}
}

@article{kap99,
	title = {Howson property and one-relator groups},
	volume = {27},
	issn = {0092-7872, 1532-4125},
	url = {http://www.tandfonline.com/doi/abs/10.1080/00927879908826481},
	doi = {10.1080/00927879908826481},
	pages = {1057--1072},
	number = {3},
	journaltitle = {Communications in Algebra},
	shortjournal = {Communications in Algebra},
	author = {Kapovich, Ilya},
	urldate = {2025-11-04},
	date = {1999-01},
	langid = {english},
}

@article{kap02,
  title = {Subgroup properties of fully residually free groups},
  author = {Kapovich, Ilya},
  journal = {Transactions of the American Mathematical Society},
  volume = {354},
  number = {1},
  pages = {335--362},
  year = {2002},
  doi = {10.1090/S0002-9947-01-02840-9},
  url = {http://www.ams.org/tran/2002-354-01/S0002-9947-01-02840-9/}
}

@article{kmw17,
  title = {Stallings graphs for quasi-convex subgroups},
  author = {Kharlampovich, O. and Miasnikov, A. and Weil, P.},
  journal = {Journal of Algebra},
  volume = {488},
  pages = {442--483},
  year = {2017},
  month = {Oct}
}

@article{ks70,
  title = {The subgroups of a free product of two groups with an amalgamated subgroup},
  author = {Karrass, A. and Solitar, D.},
  journal = {Transactions of the American Mathematical Society},
  volume = {150},
  pages = {227--255},
  year = {1970},
  doi = {10.2307/1995492},
  url = {https://zbmath.org/?q=an:0223.20031}
}

@article{ks71,
  title = {Subgroups of HNN groups and groups with one defining relation},
  author = {Karrass, A. and Solitar, D.},
  journal = {Canadian Journal of Mathematics},
  volume = {23},
  pages = {627--643},
  year = {1971},
  doi = {10.4153/CJM-1971-070-x},
  url = {https://zbmath.org/?q=an:0232.20051}
}

@article{ks96,
  title = {Greenberg's theorem for quasiconvex subgroups of word hyperbolic groups},
  author = {Kapovich, Ilya and Short, Hamish},
  journal = {Canadian Journal of Mathematics},
  volume = {48},
  number = {6},
  pages = {1224--1244},
  year = {1996},
  doi = {10.4153/CJM-1996-065-6},
  url = {https://doi.org/10.4153/CJM-1996-065-6}
}

@misc{li25,
      title={The geometry of subgroups of mapping tori of free groups}, 
      author={Marco Linton},
      year={2025},
      eprint={2510.03145},
      note={arXiv:2510.03145},
      archivePrefix={arXiv},
      primaryClass={math.GR},
      url={https://arxiv.org/abs/2510.03145}, 
}

@misc{Li21,
      title={On the intersections of finitely generated subgroups of free groups: reduced rank to full rank}, 
      author={Marco Linton},
      year={2021},
      eprint={2108.10814},
      archivePrefix={arXiv},
      note={arXiv:2108.10814},  
      primaryClass={math.GR},
      url={https://arxiv.org/abs/2108.10814}, 
}

@article{min12,
  title = {Submultiplicativity and the Hanna Neumann Conjecture},
  author = {Mineyev, Igor},
  journal = {Annals of Mathematics},
  volume = {175},
  number = {1},
  pages = {393--414},
  year = {2012},
  month = {Jan},
  doi = {10.4007/annals.2012.175.1.11},
  url = {http://annals.math.princeton.edu/2012/175-1/p11}
}

@article{mol68,
  title = {The intersection of finitely generated subgroups},
  author = {Moldavanskiĭ, D. I.},
  journal = {Sibirsk. Mat. Zh.},
  volume = {9},
  year = {1968},
  pages = {1422--1426},
}

@article{neu57,
  title = {On the intersection of finitely generated free groups. Addendum},
  author = {Neumann, Hanna},
  journal = {Publicationes Mathematicae Debrecen},
  volume = {5},
  pages = {128},
  year = {1957},
  doi = {10.5486/pmd.1957.5.1-2.14},
  url = {https://doi.org/10.5486/pmd.1957.5.1-2.14}
}

@incollection{neu90,
  title = {On intersections of finitely generated subgroups of free groups},
  author = {Neumann, Walter D.},
  booktitle = {Groups—Canberra 1989},
  editor = {Kovács, L. G.},
  publisher = {Springer Berlin Heidelberg},
  volume = {1456},
  pages = {161--170},
  year = {1990},
  url = {http://www.springerlink.com/index/10.1007/BFb0100737}
}

@article{pa12,
  title   = {On the Howson property of HNN-extensions with abelian base group and amalgamated free products of abelian groups},
  author  = {Paramantzoglou, Panagiotis A.},
  journal = {Archiv der Mathematik},
  volume  = {98},
  number  = {2},
  pages   = {115--128},
  year    = {2012},
  doi     = {10.1007/s00013-011-0344-0},
  url     = {https://doi.org/10.1007/s00013-011-0344-0}
}

@book{ser80,
  author       = {Serre, Jean-Pierre},
  title        = {Trees},
  translator   = {Stillwell, John},
  publisher    = {Springer-Verlag},
  location     = {Berlin; New York},
  year         = {1980},
  isbn         = {3-540-10103-9},
  language     = {English},
  origtitle    = {Arbres, amalgames, $SL_2$},
  origdate     = {1977},
  related      = {ser77},        % THIS is the correct link
  relatedtype  = {original},
}

@incollection{sds90,
	location = {Berlin, Heidelberg},
	title = {The finite basis extension property and graph groups},
	isbn = {978-3-540-46296-5},
	url = {https://doi.org/10.1007/BFb0084450},
	pages = {52--58},
	booktitle = {Topology and Combinatorial Group Theory: Proceedings of the Fall Foliage Topology Seminars held in New Hampshire 1986–1988},
	publisher = {Springer},
	author = {Servatius, Herman and Droms, Carl and Servatius, Brigitte},
	editor = {Latiolais, Paul},
	urldate = {2025-11-02},
	date = {1990},
	langid = {english},
	doi = {10.1007/BFb0084450},
}

@incollection{sho91,
  title = {Quasiconvexity and a theorem of Howson's},
  author = {Short, Hamish},
  booktitle = {Group theory from a geometrical viewpoint (Trieste, 1990)},
  pages = {168--176},
  publisher = {World Sci. Publ., River Edge, NJ},
  year = {1991}
}

@article{sta65,
  title = {A topological proof of Grushko's theorem on free products},
  author = {Stallings, John R.},
  journal = {Mathematische Zeitschrift},
  volume = {90},
  pages = {1--8},
  year = {1965},
  doi = {10.1007/BF01112046},
  url = {https://doi.org/10.1007/BF01112046}
}

@article{To25,
 author = {Tomar, Ravi},
 title = {A remark on complexes of hyperbolic groups with finite edge groups},
 fjournal = {Geometriae Dedicata},
 journal = {Geom. Dedicata},
 issn = {0046-5755},
 volume = {219},
 number = {6},
 pages = {19},
 note = {Id/No 92},
 year = {2025},
 language = {English},
 doi = {10.1007/s10711-025-01054-x},
}

@article{Wi00,
 author = {Wise, Daniel T.},
 title = {Subgroup separability of graphs of free groups with cyclic edge groups},
 fjournal = {The Quarterly Journal of Mathematics},
 journal = {Q. J. Math.},
 issn = {0033-5606},
 volume = {51},
 number = {1},
 pages = {107--129},
 year = {2000},
 language = {English},
 doi = {10.1093/qmathj/50.1.107}
}


\appendix

\makeatletter
\renewcommand{\@seccntformat}[1]{}
\makeatother

\section{Appendix: Remarks on Burns subgroups}\label{sec: Burns}

Burns \cite{bur72} considered almost malnormal and finitely involved subgroups, which Cohen called \emph{Burns subgroups} \cite{coh76}. Kapovich generalised this property to tuples of subgroups in \cite[Definition 5.3]{kap02} as follows. 

\begin{defn}\label{def: Burns collection}
A tuple $(A_1, \ldots, A_n)$ of subgroups of a group $G$ is a \emph{Burns collection of subgroups} if the following holds:
\begin{enumerate}[(1)]
\item for each $i\in [1,n]$, $A_i$ admits a left transversal $T_i\subset G$ containing $1$ (that is, a set of unique representatives of the cosets $xA_i$ when $x$ runs through $G$) such that
\begin{enumerate}[(a)]
\item There exists a finite subset $F_i$ of $A_i$ such that $A_i\,(T_i-\{1\})\subseteq T_i\,F_i$,
\item\label{item: item 2 of Burns} For any finitely generated subgroup $H$ of $G$ and any element $g\in G$, there exists a finite subset $F$ of $A_i$ such that $H\,g\subseteq T_i\,F\,(H^g\cap A_i)$;
\end{enumerate}
\item for each $i,j \in [1,n]$ such that $i \ne j$,  $A_i\cap A_j = 1$;
\item for each $i \in [1,n]$, there exists a finite subset $S_i$ of $A_i$ such that, for each $j\in [1,n]$, $A_j\,(T_i - \{1\})\subseteq T_i\,S_i$.
\end{enumerate}
If $A$ is a subgroup of $G$ and the 1-tuple $(A)$ is a Burns collection, we simply say that $A$ is a \emph{Burns subgroup} (or is \emph{almost malnormal finitely involved}).
\end{defn}

The following proposition shows that Burns' condition is subsumed by our \fcip

\begin{prop}\label{prop: Burns collection}
If $(A_1, \ldots, A_n)$ is a Burns collection of subgroups of $G$, then $G$ has the \fcip\ relative to $(A_1,\dots, A_n)$. In particular, if $A$ is a Burns subgroup, then $G$ has the \fcip\ relative to $A$.
\end{prop}

\def\calF{\mathcal{F}}
\def\calG{\mathcal{G}}
\begin{proof}
Let $B, C$ be two finitely generated subgroups of $G$ and, for each $i\in [1,n]$, let $\calF_i$ and $\calG_i$ be finite subsets of $G$ whose elements lie in pairwise distinct $(B,A_i)$- and $(C,A_i)$-double cosets, respectively. It suffices to show that there exists a finite subset $F$ of $G$ such that 
\begin{enumerate}[(i)]
\item\label{item: fg = f'g' collection} if $i\in [1,n]$, $a, a'\in A_i$, $f\in \calF_i$, $g\in \calG_i$, $B\,(f\,a'\,g\inv)\,C = B\,(f\,a\,g\inv)\,C$ and $(A_i\cap B^f)\,a\,(A_i\cap B^g) \ne (A_i\cap B^f)\,a'\,(A_i\cap B^g)$, then $a \in (A_i\cap B^f)\,F\,(A_i\cap B^g)$ (so that each $|(q^i_{f,g})\inv(B\,h\,C)|$ is finite), and hence $a \in B^f\,F\,C^g$ (so that only finitely many are greater than 1);

\item\label{item: fg not= f'g' same i} if $i\in [1,n]$, $a, a' \in A_i$, $f, f'\in \calF_i$, $g, g'\in \calG_i$, $(f,g) \ne (f',g')$ and $B\,(f\,a\,g\inv)\,C = B\,(f'\,a'\,(g')\inv)\,C$, then $a \in B^f\,F\,C^g$ (so that finitely many $(B,C)$-double cosets are in the range of two distinct $q^i_{f,g}$).

\item\label{item: fg not= f'g' ij} if $i,j\in [1,n]$ satisfy $i\ne j$, $a_i \in A_i$, $a_j\in A_j$, $f_i\in \calF_i$, $f_j\in \calF_j$, $g_i\in \calG_i$, $g_j\in \calG_j$ and $B\,(f_i\,a_i\,g_i\inv)\,C = B\,(f_j\,a_j\,g_j\inv)\,C$, then $a_i \in B^{f_i}\,F\,C^{g_i}$ (so that finitely many $(B,C)$-double cosets are in the range of $q^i_{f_i,g_i}$ and $q^j_{f_j,g_j}$ with $i\ne j$).
\end{enumerate}
Fix $i\in [1,n]$. By Definition~\ref{def: Burns collection}, there exists a left transveral $T_i$ of $A_i$, a finite subset $F_i$ of $A_i$ and, for all $h,h'\in G$,  finite subsets of $F_{i,B,h,h'}$ and $F_{i,C,h,h'}$ of $A_i$ such that
\begin{align*}
    A_i\, (T_i - \{1\})&\enspace\subseteq\enspace T_i\, F_i\\
    h\inv\,B\,h' \enspace=\enspace B^h\,(h\inv\,h') &\enspace\subseteq\enspace T_i\,F_{i,B,h,h'}\,(B^{h'} \cap A_i)\\
    h\inv\,B\,h' \enspace=\enspace C^h\,(h\inv\,h') &\enspace\subseteq\enspace T_i\,F_{i,C,h,h'}\,(C^{h'} \cap A_i).
\end{align*}
Let $f,f'\in \calF_i$ and $b\in B$. Then $f\inv\,b\,f' = t\,x\,y$ for some $t \in T_i$, $x\in F_{i,B,f,f'}$ and $y\in A_i\cap B^{f'}$. Our hypothesis on $\calF_i$ shows that, if $f\ne f'$, then $B\,f\,A_i$ and $B\,f'\,A_i$ do not meet, and hence $f\inv\,B\,f'$ does not meet $A_i$. Therefore $t = 1$ if and only if $f = f'$ and $b^f \in A_i$. A similar result holds for $g, g'\in \calG_i$ and $c\in C$.

Suppose that $f \in \calF_i$, $g\in \calG_i$, $a,a' \in A_i$ and $B\,(f'\,a'\,(g')\inv)\,C = B\,(f\,a\,g\inv)\,C$. Then $f'\,a'\,(g')\inv \in B\,(f\,a\,g\inv)\,C$, so
\begin{align*}
a' &= (f')\inv\,b\,f\,a\,g\inv\,c\,g'\quad\textrm{ for some $b\in B$ and $c\in C$, and hence} \\
a'&= t'\,x'\,y'\,a\,y\inv\,x\inv\,t\inv,
\end{align*}
where $t, t'\in T_i$, $x'\in F_{i,B,f',f}$, $x\in F_{i,C,g',g}$, $y'\in B^f\cap A_i$ and $y\in C^g\cap A_i$.

We note that $t = 1$ if and only if $t' = 1$, since $T \cap A_i = \{1\}$. In that situation, as we saw, $f = f'$, $g = g'$, and $b^f, c^g \in A_i$ and $a' \in (B^f\cap A_i)\,a\,(C^g \cap A_i)$. Thus we never have $t = 1$ or $t' = 1$ under the assumptions of Item~\eqref{item: fg = f'g' collection} or~\eqref{item: fg not= f'g' same i}.

If $t, t'\ne 1$, then $t'\,(x'\, y'\,a\,y\inv\,x\inv) = a'\,t = t''\,z$ for some $t''\in T_i$ and $z\in F_i$. Since $T_i$ is a left transversal, this implies $t' = t''$, and hence $a \in (B^f\cap A_i)\,(F_{i,B,f',f}\inv\,F_i\,F_{i,C,g',g})\,(C^g \cap A_i)$.
Thus Items~\eqref{item: fg = f'g' collection} and~\eqref{item: fg not= f'g' same i} hold for any finite set $F$ containing the $F_{i,B,f',f}\inv\,F_i\,F_{i,C,g',g}$, for all $i\in [1,n]$, $f, f'\in \calF_i$ and $g,g'\in \calG_i$.

Let now $i, j, a_i, a_j, f_i, f_j, g_i, g_j$ as in Item~\eqref{item: fg not= f'g' ij}, with $B\,(f_j\,a_j\,g_j\inv)\,C = B\,(f_i\,a_i\,g_i\inv)\,C$. As above,
\begin{align*}
a_j &= f_j\inv\,b\,f_i\,a_i\,(g_j\inv\,c\,g_i)\inv\textrm{ for some $b\in B$ and $c\in C$, and hence} \\
a_j&= t'\,x'\,y'\,a_i\,y\inv\,x\inv\,t\inv,
\end{align*}
with $t, t'\in T_i$, $x' \in F_{i,B,f_j,f_i}$, $x \in F_{i,C,g_j,g_i}$, $y' \in B^{f_i} \cap A_i$ and $y \in C^{g_i} \cap A_i$. Then $a_j\,t = t'\,(x'\, y'\,a_i\,y\inv\,x\inv)$. If $t\ne 1$, then $a_j\,t = t''\,z$ where $t''\in T_i$ and $z \in S_i$. Since $z$ and $x'\, y'\,a_i\,y\inv\,x\inv$ are in $A_i$, it follows that $t' = t''$, $x'\, y'\,a_i\,y\inv\,x\inv = z$, and hence
$$a_i \in (B^{f_i}\cap A_i)\,(F_{i,B,f_j,f_i}\inv\,S_i\,F_{i,C,g_j,g_i})\,(C^{g_i}\cap A_i).$$
Therefore $a_i \in B^{f_i}\,(F_{i,B,f_j,f_i}\inv\,S_i\,F_{i,C,g_j,g_i})\,C^{g_i}$.

If $t = 1$ and $a_j \ne 1$, then $t'\ne 1$ since $A_i \cap A_j$ is trivial. Moreover $a_j\inv\,t' = (x'\, y'\,a_i\,y\inv\,x\inv)\inv \in A_i$, and $a_j\inv\,t' = t''\,z$ for some $t''\in T_i$ and $z\in S_i$. As above, this implies $t'' = 1$ and $a_i \in (B^{f_i}\cap A_i)\,(F_{i,B,f_j,f_i}\inv\,S_i\inv\,F_{i,C,g_j,g_i})\,(C^{g_i}\cap A_i)$. Therefore $a_i, a_j \in B^{f_i}\,(F_{i,B,f_j,f_i}\inv\,S_i\inv\,F_{i,C,g_j,g_i})\,C^{g_i}$.

Finally, if $t = a_j = 1$, then $t' \in A_i$, so $t' = 1$, $a_i \in (B^{f_i}\cap A_j)\,(F_{i,B,f_j,f_i}\inv\,F_{i,C,g_j,g_i})\,(C^{g_i}\cap A_j)$ and $a_i, a_j \in B^{f_i}\,(F_{i,B,f_j,f_i}\inv\,F_{i,C,g_j,g_i})\,C^{g_i}$. This concludes the proof.
\end{proof}

We will also use the following property of Burns subgroups.

\begin{lem}
\label{lem:Burns_A_fgip}
    If $A\leqslant G$ is a Burns subgroup, then $(G, A)$ has the $A$-\fgip\
\end{lem}

\begin{proof}
Let $H, K\leqslant G$ be finitely generated subgroups and let $H\,g\,K$ be a double coset which meets $A$ and satisfies $H^g\cap K\neq 1$. In particular, there exists $a\in H\,g\,K\cap A$ such that $H^a\cap K\neq 1$.

Since $A$ is a Burns subgroup, there exists a left transversal $T$ of $A$ in $G$ (containing $1$) and finite subsets $F, F_H, F_K\subset A$ such that
\begin{align*}
    A(T - 1)&\subset TF\\
    H &\subset TF_H(H\cap A)\\
    K &\subset TF_K(K\cap A)
\end{align*}
Suppose that $H^a\cap K$ is not contained in $A$. There exist elements $t_H, t_K\in T$, $f_H\in F_H$, $f_K\in F_K$, $a_H\in H\cap A$ and $a_K\in K\cap A$ such that
\[
a^{-1}t_Hf_Ha_Ha = t_Kf_Ka_K \notin A.
\]
Since $f_K, a_K \in A$, we see that $t_K\ne 1$. Moreover
\[
t_Hf_Ha_Ha = at_Kf_Ka_K \notin A,
\]
and hence $t_H \ne 1$ (since $a, a_H, f_H \in A$). It follows that $a\inv t_H = t\,f$ for some $t\in T$ and $f\in F$. Therefore
\[
tff_Ha_Ha = t_Kf_Ka_K.
\]
Since $a, f, f_H, a_H\in A$ and $T$ is a left transversal of $A$, we have $t = t_K$, and hence
\[
ff_Ha_Ha = f_Ka_K.
\]
Thus
\[
a\in (H\cap A)F_H^{-1}F^{-1}F_K(K\cap A)
\]
and, in particular, $a$ lies in one of finitely many $(H, K)$-double cosets. This completes the proof. 
\end{proof}

\bigskip
\subsection*{Acknowledgements}
The first and fourth authors acknowledge support from the Spanish Agencia Estatal de Investigación through grant PID2021-126851NB-100 (AEI/FEDER, UE).
The second author acknowledges support from the grant 202450E223 (Impulso
de l\'{i}neas cient\'{i}ficas estrat\'{e}gicas de ICMAT).
The third author has been supported by the Spanish Government
grant PID2020-117281GB-I00, partly by the European Regional Development Fund (ERDF), and the Basque Government, grant IT1483-22.
The fourth author expresses gratitude for the generous hospitality received from Harish-Chandra Research Institute, India and Universidad del Pais Vasco, Spain. We thank Ravi Tomar for comments on a previous version of this article.





\end{document}